%% file: paper.tex
\documentclass[11pt]{amsart}
\usepackage{graphicx} 

\usepackage{amsfonts}
\usepackage{amsthm}
\usepackage{tikz}
\usepackage{lipsum} 
\usepackage{caption}
\usepackage{subcaption}
\usepackage{xcolor}
\definecolor{orange}{RGB}{255, 102, 0}
\usepackage[bookmarks,breaklinks,colorlinks=true, allcolors=blue]{hyperref}
\usepackage{comment}
\usepackage{amssymb,amsmath,array,diagbox}

\usepackage{yhmath}
\usepackage{mathdots}
\usepackage{MnSymbol}

\usepackage{setspace}

\usepackage{enumitem}

\newtheorem{theorem}{Theorem}[section]

\newtheorem{lemma}[theorem]{Lemma}

\newtheorem{conjecture}[theorem]{Conjecture}
\newtheorem{proposition}[theorem]{Proposition}

\theoremstyle{definition}
\newtheorem{definition}[theorem]{Definition}

\newtheorem{remark}[theorem]{Remark}

\usetikzlibrary{arrows, shapes, decorations, decorations.markings, backgrounds, patterns, hobby, knots, calc, positioning,shapes.geometric}

\pgfdeclarelayer{background}
\pgfdeclarelayer{background2}
\pgfdeclarelayer{background2a}
\pgfdeclarelayer{background2b}
\pgfdeclarelayer{background3}
\pgfdeclarelayer{background4}
\pgfdeclarelayer{background5}
\pgfdeclarelayer{background6}
\pgfdeclarelayer{background7}

\pgfsetlayers{background7,background6,background5,background4,background3,background2b,background2a,background2,background,main}

\definecolor{lsupurple}{RGB}{70,29,124}
\definecolor{lsugold}{RGB}{253,208, 35}

\def\Span{\operatorname{span}}

\definecolor{myblue}{RGB}{73,190,255}

\usepackage{float}
\usepackage{subfloat}

\makeatletter
\@namedef{subjclassname@2020}{%
  \textup{2020} Mathematics Subject Classification}
\makeatother

\begin{document}

\title{On the arc index and Turaev genus of a link}

\author{\'Alvaro del Valle V\'ilchez}
\address{Departamento de Algebra de la Universidad de Sevilla \& Instituto de Matem\'aticas de la Universidad de Sevilla (IMUS). Av. Reina Mercedes s/n. 41012, Sevilla. Spain.}
\email{adelvalle3@us.es}
 \thanks{The first author was partially supported by the grant VII PPIT-US, and by the projects PID2020-117971GB-C21 and PID2024-157173NB-I00 funded by MCIN/AEI/10.13039/501100011033 and by FEDER,
EU}

\author{Adam M. Lowrance}
\address{Department of Mathematics and Statistics, Vassar College, Poughkeepsie, NY 12604}
\email{adlowrance@vassar.edu}





\begin{abstract}
We compute the arc index of an adequate link and establish bounds on the arc index of the closure of a positive 3-braid. We also conjecture an inequality between the crossing number, arc index, and Turaev genus of a link and show the conjecture is true for several infinite families of links including alternating links, links with Turaev genus one, adequate links, closures of positive 3-braids, torus links, and most Kanenobu knots. 
\end{abstract}

\maketitle

\section{Introduction}
An \textit{arc presentation} of a link $L$ is an embedding of $L$ into finitely many half-planes whose boundary is the $z$-axis such that each half-plane intersects the link in a single properly embedded arc. The \textit{arc index} $\alpha(L)$ of the link $L$ is the minimum number of half-planes in any arc presentation of $L$. 

Let $L$ be a non-split prime link. Bae and Park \cite{Bae_Park_2000} proved that if $L$ is alternating, then $\alpha(L) = c(L)+2$. Jin and Park \cite{Jin_Park_2010} proved that if $L$ is non-alternating, then $\alpha(L) \leq c(L)$.  Hence one can view the quantity $c(L)+2 - \alpha(L)$ as a way to measure how far a link is from being alternating.

Lowrance \cite{Lowrance_2021, Lowrance_2015} studied various ways of measuring the distance between a link and the set of alternating links. One such measure is the Turaev genus $g_T(L)$ of a non-split link $L$, defined as follows. A \textit{Kauffman state} of the link diagram $D$ is the set of simple closed curves obtained from a choice at each crossing $\tikz[baseline=.6ex, scale = .4]{
\draw (0,0) -- (1,1);
\draw (1,0) -- (.7,.3);
\draw (.3,.7) -- (0,1);
}
$ of an $A$-resolution $ \tikz[baseline=.6ex, scale = .4]{
\draw[rounded corners = 1mm] (0,0) -- (.45,.5) -- (0,1);
\draw[rounded corners = 1mm] (1,0) -- (.55,.5) -- (1,1);
}$ or a $B$-resolution $\tikz[baseline=.6ex, scale = .4]{
\draw[rounded corners = 1mm] (0,0) -- (.5,.45) -- (1,0);
\draw[rounded corners = 1mm] (0,1) -- (.5,.55) -- (1,1);
}$. The state $s_A D$ with an $A$-resolution at every crossing is the \textit{all-$A$ state} of $D$, and the state $s_B D$ with a $B$-resolution at every crossing is the \textit{all-$B$ state} of $D$. The \textit{Turaev genus} of $L$ is defined as
\[g_T(L) = \min\left\{  \frac{1}{2}\left(2+c(D) -|s_A D| - |s_B D| \right) \; \middle| \; D \text{ is a diagram of }  L \right\},\]
where $c(D)$ is the number of crossings in $D$, and $|s_A D|$ and $|s_B D|$ are the number of components in the all-$A$ and all-$B$ Kauffman states of $D$, respectively. 

We propose a conjecture relating the crossing number $c(L)$, the arc index $\alpha(L)$, and the Turaev genus $g_T(L)$ of a non-split prime link $L$.

\begin{conjecture}\label{conj:c+2-a_gT}
    For any non-split prime link $L$, 
    \[ c(L) + 2 - \alpha(L) \geq 2 g_T(L).\]
\end{conjecture}
Turaev \cite{Turaev_1987} showed that $g_T(L)=0$ if and only if $L$ is alternating. Combining this fact with Bae and Park's computation of the arc index of an alternating link and Jin and Park's inequality for the arc index of a non-alternating link establishes Conjecture~\ref{conj:c+2-a_gT} when $L$ is alternating or when $g_T(L)=1$. The computations of Turaev genus and arc index in KnotInfo \cite{knotinfo} confirm Conjecture~\ref{conj:c+2-a_gT} for all knots with at most twelve crossings. In this paper, we prove Conjecture~\ref{conj:c+2-a_gT} for adequate links, closures of positive 3-braids, torus links, and most Kanenobu knots. 

We find a formula for the arc index of an adequate link. A link diagram $D$ is \textit{$A$-adequate} (respectively \textit{$B$-adequate}) if no two arcs in the $A$-resolution ($B$-resolution) of any crossing lie on the same component of $s_A D$ ($s_B D$). A link diagram is \textit{adequate} if it is both $A$- and $B$-adequate, and a link is \textit{adequate} if it has an adequate diagram. Adequate links were defined by Lickorish and Thistlethwaite \cite{Lickorish_Thistlethwaite_1988}.

 Park and Seo \cite{Park_Seo_2000} conjectured a formula for the arc index of an adequate link $L$ involving the quantity $\rho(D)$ whose precise definition uses the Betti numbers of the checkerboard graph and is given in Section~\ref{sec:adeq}. We confirm Park and Seo's conjecture in the following result.
 \begin{theorem}
 \label{thm:adequate}
     Let $L$ be a non-split adequate link with adequate diagram $D$. Then
     \[\alpha(L) = |s_A D| + |s_B D| =  c(L)+ 2\rho(D).\]
 \end{theorem}
 
The proof that $\alpha(L)=|s_A D| + |s_B D|$ follows from work of K\'alm\'an \cite{Kalman_2008} on the maximum Thurston-Bennequin number of adequate links and from work of Dynnikov and Prasolov \cite{Dynnikov_Prasolov_2013} expressing the arc index $\alpha(L)$ in terms of the maximum Thurston-Bennequin numbers of $L$ and its mirror $L^*$.

 In Theorem~\ref{th:arc_index_closed_positive_3-braids}, we give new upper and lower bounds on the arc index of the closure of a positive $3$-braid. We combine these bounds on arc index with previously known bounds on Turaev genus \cite{Lowrance_2011} to confirm Conjecture~\ref{conj:c+2-a_gT} for the closures of positive 3-braids.

This paper is organized as follows. Section~\ref{sec:pre} gives background information on arc index, Turaev genus, and Khovanov homology. In Section~\ref{sec:adeq}, we prove Theorem~\ref{thm:adequate} and confirm Conjecture~\ref{conj:c+2-a_gT} for adequate links. In Sections~\ref{sec:3braids}, \ref{sec:torus_links}, and \ref{sec:Kanenobu}, we confirm Conjecture~\ref{conj:c+2-a_gT} for the closure of positive $3$-braids, torus links, and most Kanenobu knots, respectively.

\section{Preliminaries}
\label{sec:pre}

In this section, we cover background material on the arc index of a link, the Turaev surface and Turaev genus of a link.

\subsection{Arc index} We represent arc presentations in several different equivalent ways, as shown in Figure~\ref{fig:arcpres}. The top left diagram of Figure~\ref{fig:arcpres} shows an arc presentation depicted as a collection of semicircles attached to a binding. The integer labels on the semicircles indicate the cyclic order of the pages meeting the binding. The top right diagram of Figure~\ref{fig:arcpres}, called a \textit{spoke diagram}, shows the view of the arc presentation from the $xy$-plane, so that the $z$-axis binding becomes a single point. In this representation, each simple arc becomes a line segment, and each line segment is labeled by a pair of integers indicating the two heights at which the simple arc intersects the binding. The bottom left of Figure~\ref{fig:arcpres} shows an arc presentation where each simple arc is three sides of a rectangle. Projecting this rectilinear embedding in the correct way results in the diagram on the bottom right of Figure~\ref{fig:arcpres}, known as a \textit{grid diagram} of $L$.

\begin{figure}[h]
\input{arcpres.tex}
\caption{Four different ways of representing the trefoil as an arc presentation.}
\label{fig:arcpres}
\end{figure}
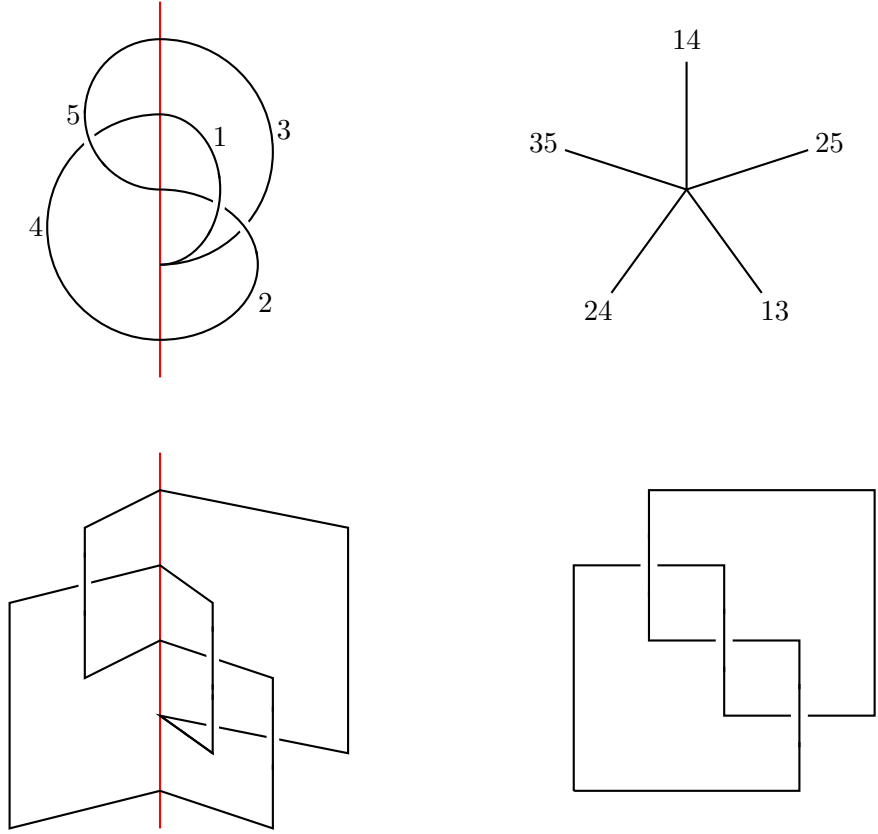

Brunn \cite{Brunn_1913} introduced arc presentations in the late nineteenth century. Nearly a century later, Birman and Menasco \cite{Birman_Menasco_1994} used arc presentations to study the braid index of satellite links and also defined the arc index of a link. Cromwell \cite{Cromwell_1995} proved that the arc index of the connected sum $L_1\# L_2$ of links $L_1$ and $L_2$ is determined by $\alpha(L_1\# L_2)=\alpha(L_1)+\alpha(L_2)-2$, which allows us to restrict our focus to prime links.

Cromwell \cite{Cromwell_1995} also described a set of Reidemeister-like moves on arc presentations such that any two arc presentations of the same link can be connected by a sequence of these Cromwell moves. One Cromwell move increases the number of pages in the arc presentation, while the rest of the moves either decrease or do not change the number of pages. Dynnikov \cite{Dynnikov_2006} proved that any arc presentation of the unknot can be transformed into the two-page arc presentation of the unknot using only Cromwell moves that preserve or decrease the number of pages, yielding an algorithm to detect the unknot.

Because the arc index of a link is defined as a minimum, its computation often relies on a lower bound. Beltrami and Morton \cite{Beltrami_Morton_1998} proved that the difference between the maximum and minimum $a$-degree of the two-variable Kauffman polynomial $F_L(a,z)$ gives the following lower bound on arc index:
\begin{equation}
    \label{eq:KauffmanBound}
    \Span_a F_L(a,z) + 2 \leq \alpha(L).
\end{equation}
Nutt \cite{Nutt_1997} previously proved a version of Inequality \eqref{eq:KauffmanBound} for the Kauffman polynomial with coefficients reduced modulo 2. Let $H(L)$ be the Khovanov homology (with coefficients in $\mathbb{Z}$) of $L$, first discovered by Khovanov in \cite{Khovanov_2000}, and let $H^{i,j}(L)$ denote the summand in homological grading $i$ and polynomial grading $j$. Ng \cite{Ng_2012} proved that Khovanov homology gives the following lower bound:
\begin{equation}
    \label{eq:KhBound}
    \max\{j-i~|~H^{i,j}(L)\neq 0\} - \min\{j-i~|~H^{i,j}(L)\neq 0\} \leq \alpha(L).
\end{equation}

Both lower bounds \eqref{eq:KauffmanBound} and \eqref{eq:KhBound} can be viewed as consequences of a relationship between arc presentations and Legendrian knot theory. A \textit{Legendrian link} is a link $\mathcal{L}$ in $\mathbb{R}^3$ that is everywhere tangent to the standard contact structure $dz-y\;dx$. Projecting a Legendrian link to the $xz$-plane results in a \textit{Legendrian front} diagram $F$ with the following properties. The front $F$ has no vertical tangencies, and at every transverse double point the segment with a negative slope passes over the segment with a positive slope.

A grid diagram $D$ of a link $L$ can be transformed into a Legendrian front $F_D$ whose underlying classical link type is $L$ by rotating $D$ counterclockwise by $\pi/4$ and then smoothing north and south corners. Similarly, the grid diagram $D$ can be transformed into a Legendrian front $F_{D^*}$ whose underlying classical link type is the mirror $L^*$ by rotating $D$ clockwise by $\pi/4$, changing every crossing, and then smoothing north and south corners. An example is shown in Figure~\ref{fig:front}.
\begin{figure}[h]
\input{front.tex}
\caption{The Legendrian front diagrams $F_D$ and $F_{D^*}$ obtained from the grid diagram in Figure~\ref{fig:arcpres}.}
\label{fig:front}
\end{figure}
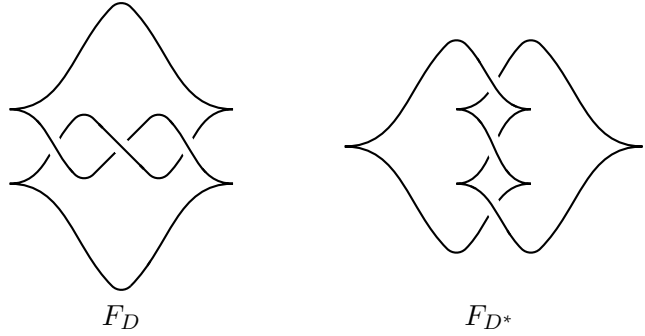

The Thurston-Bennequin number $tb(\mathcal{L})$ of the Legendrian link $\mathcal{L}$ can be readily computed from any front diagram $F$ of $\mathcal{L}$ via the formula
\[tb(\mathcal{L}) = tb(F) =  w(F) - \operatorname{cusp}(F),\]
where $w(F)$ is the writhe of $F$ and $ \operatorname{cusp}(F)$ is the number of right cusps of $F$. The \textit{maximum Thurston-Bennequin number} $\overline{tb}(L)$ of a classical link $L$ is defined as
\[\overline{tb}(L) = \max\{tb(\mathcal{L})~|~\mathcal{L}~\text{is a Legendrian representative of}~L\}.\]

The maximum Thurston-Bennequin number and arc index of a link are related as follows. Let $D$ be a grid diagram of a link $L$ with $\alpha(L)$ vertical line segments. Each segment can be associated with its top vertex. This vertex forms either an upper-right corner, giving rise to a right cusp in $F_D$, or an upper-left corner, giving rise to a right cusp in $F_{D^*}$. Then the total number of right cusps in $F_D$ and $F_{D^*}$ coincides with $\alpha(L)$. Therefore
\begin{align}
\begin{split}
\label{eq:tbineq}
\alpha(L)  = & \;  \operatorname{cusp}(F_D) +  \operatorname{cusp}(F_{D^*}) \\
 = & \; -w(F_D)+ \operatorname{cusp}(F_D) - w(F_{D^*})+ \operatorname{cusp}(F_{D^*}) \\
 = & \; -tb(F_D) - tb(F_{D^*}) \\
 \geq & \;  -\overline{tb}(L)-\overline{tb}(L^*) .
\end{split}
\end{align}
Dynnikov and Prasolov  showed that equality holds in Inequality \eqref{eq:tbineq}.
\begin{theorem}[{\cite[Cor. 3]{Dynnikov_Prasolov_2013}}]
\label{thm:tb} For any link $L$,
 \[\alpha(L) = -\overline{tb}(L) - \overline{tb}(L^*).\]
 \end{theorem}

K\'alm\'an computed the maximum Thurston-Bennequin number of an $A$-adequate link.
\begin{theorem}[{\cite[Cor. 6]{Kalman_2008}}]
\label{thm:Aadeq}
 Let $D$ be an $A$-adequate diagram of the link $L$. The maximum Thurston-Bennequin number of $L$ is $\overline{tb}(L) = w(D) - |s_A D|$, where $w(D)$ is the writhe of $D$ and $|s_A D|$ is the number of components in the all-$A$ Kauffman state of $D$.
 \end{theorem}

We now describe a procedure to transform any link diagram into an arc presentation. Consider the diagram $D$ as a $4$-regular graph whose vertices are the crossings and whose edges are arcs of $D$ going between the crossings. As a graph, $D$ has $c(D)$ vertices and $2c(D)$ edges. Let $T$ be a spanning tree of $D$, and let $\eta$ be the boundary of a regular neighborhood of $T$ in the plane. Since $T$ is a tree, $\eta$ is a simple closed curve that we isotope to be a circle. All crossings of $D$ are on the interior of $\eta$. Figure~\ref{fig:tree1} shows an example of a spanning tree and its regular neighborhood.
\begin{figure}[h]
\input{tree1.tex}
\caption{On the left, a spanning tree  $T$ of a link diagram $D$ is indicated by bold edges. On the right, the portion of $D$ in a regular neighborhood of $T$ is blue, while the rest of $D$ is red. The dotted curve $\eta$ is the boundary of a regular neighborhood of $T$.}
\label{fig:tree1}
\end{figure}
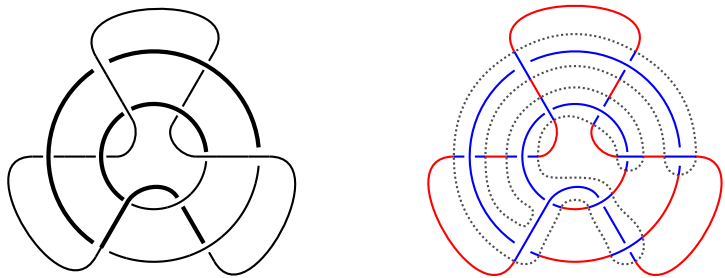

Assign an integer height to each arc on the interior of $\eta$ as follows. Arbitrarily pick any arc and arbitrarily assign it a height. Now iteratively choose an arc $\gamma$ that crosses some arc $\beta$ that has already been assigned a height. If $\gamma$ passes over $\beta$, the height of $\gamma$ is defined to be one greater than the largest height so far, and if $\gamma$ passes under $\beta$, the height of $\gamma$ is defined to be one less than the smallest height so far. Continue in this fashion until every arc has been assigned a height.  Figure~\ref{fig:tree2} shows the example from Figure~\ref{fig:tree1} where $\eta$ is a circle and each interior arc has been assigned a height.

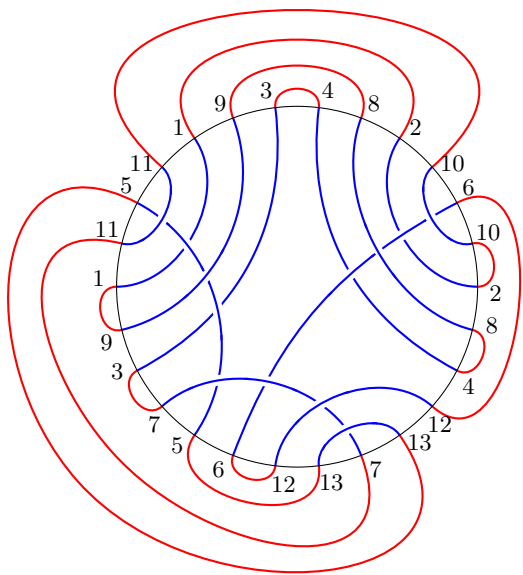
\begin{figure}[h]
\input{tree2.tex}
\caption{An isotopy of the diagram from Figure~\ref{fig:tree1} on the sphere yields the above diagram. The outer face of the diagram from Figure~\ref{fig:tree1} is labeled by *. Each interior arc has been assigned a height.}
\label{fig:tree2}
\end{figure}

The exterior of $\eta$ consists of $c(D)+1$ arcs with no crossings. The next step in the process of creating an arc presentation is to transform the diagram so that each exterior arc is unnested. Each exterior arc is labeled by the two heights of the interior arcs it meets. Suppose two nested arcs are labeled by heights $\{a,d\}$ and $\{b,c\}$ respectively with $a<d$ and $b<c$, as in Figure~\ref{fig:extarcs}. If $a<\min\{b,c\}$ or $d>\max\{b,c\}$, then a Reidemeister II move makes the two arcs unnested without adding any new exterior arcs. Otherwise, we have $b<a<d<c$. In this case, perform a Reidemeister II move that creates a new exterior arc and a new interior arc. The height $M$ of the new interior arc is one greater than the maximum height. On the exterior, there are now three unnested arcs, labeled by heights $\{a,M\}$, $\{b,c\}$, and $\{M,d\}$. Figure~\ref{fig:extarcs} depicts these isotopies. 
\begin{figure}[h]
\input{extarcs.tex}
\caption{Isotopies resulting in unnested exterior arcs.}
\label{fig:extarcs}
\end{figure}
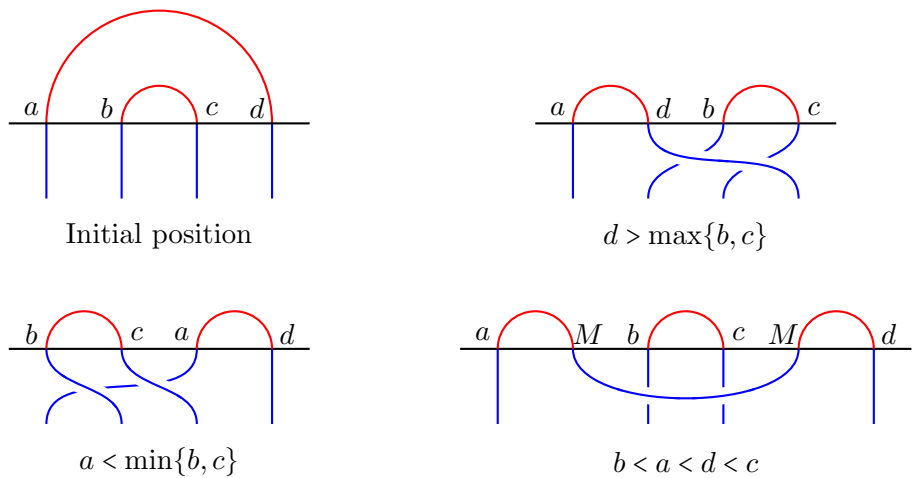

Applying these isotopies repeatedly results in a diagram $D'$ where all exterior arcs are unnested. The left diagram in Figure~\ref{fig:tree3} shows the diagram from Figure~\ref{fig:tree2} after applying the unnesting moves from Figure~\ref{fig:extarcs}. If an exterior arc is labeled by two adjacent heights $h$ and $h+1$, then push it inside of $\eta$, assign the resulting arc a height of $h$, and reduce by one the heights of all arcs with height at least $h+2$. Each such exterior arc is called a \textit{reducible edge} of the tree $T$, and the number of reducible edges of $T$ is denoted by $r(T)$. The right diagram in Figure~\ref{fig:tree3} shows the result of pushing the four reducible edges inside of $\eta$. Finally, shrink the circle $\eta$ to a point and turn each unnested arc labeled by heights $a$ and $b$ into a spoke labeled by $a\; b$ to obtain a spoke diagram, as in Figure~\ref{fig:spoke}.

\begin{figure}[h]
\input{tree3.tex}
\caption{The diagram on the left is obtained from the diagram in Figure~\ref{fig:tree2} by applying the unnesting moves in Figure~\ref{fig:extarcs}. The diagram on the right is obtained from the diagram on the left by pushing in reducible edges.}
\label{fig:tree3}
\end{figure}
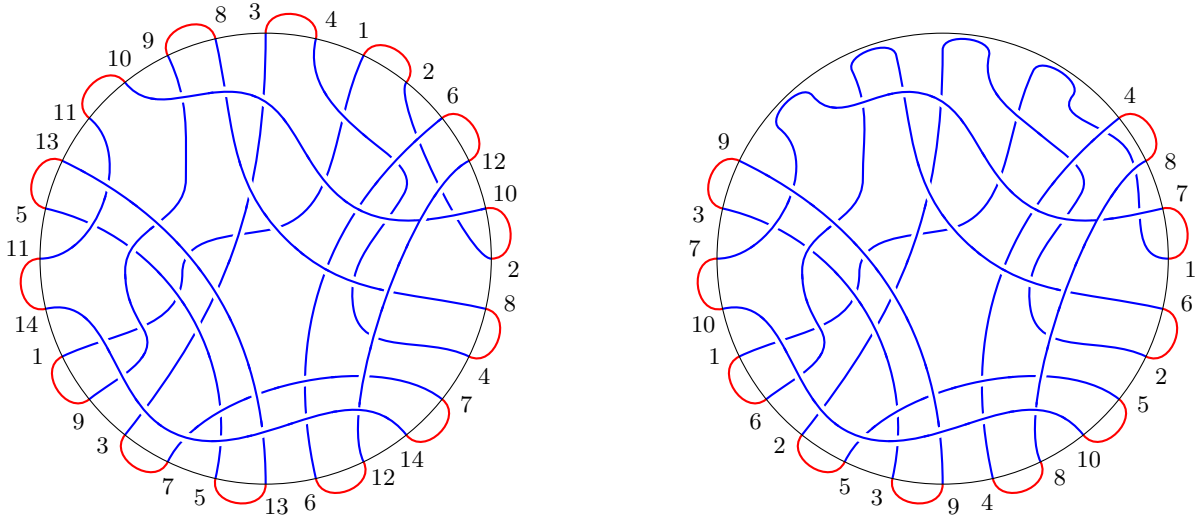
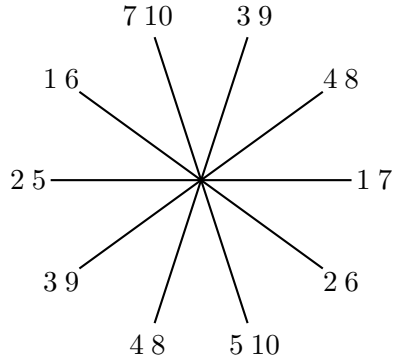
\begin{figure}[h]
\input{arcexample.tex}
\caption{The spoke diagram obtained from the diagram on the right in Figure~\ref{fig:tree3}.}
\label{fig:spoke}
\end{figure}

Since our goal is to minimize the number of arcs or spokes, we hope to find a spanning tree with many reducible edges that results in few exterior arcs like those in the bottom-right of Figure~\ref{fig:extarcs}. Jin and Lee \cite{Jin_Lee_2012} showed that if $D$ is a prime diagram, then there is a spanning tree $T$ where the unnesting move that adds an exterior arc only needs to be applied once, resulting in an arc presentation with $c(D)+2$ arcs. To do this, they build the tree $T$ as a sequence of trees $T_0 \subset T_1 \subset \cdots\subset T_k = T$ where the tree $T_i$ is obtained from the tree $T_{i-1}$ by adding a single edge $e_i$. Thus the edge set of each $T_i$ is $E(T_i) = \{e_1,\dots,e_i\}.$ The vertex set $V(T_i)$ is defined to be  the set of vertices incident to the edges in $E(T_i)$. Let $v_i$ be the vertex that is in $V(T_i)$ but not in $V(T_{i-1})$. With this convention, the vertex set is $V(T_i)=\{v_0,v_1,\dots,v_i\}$. Not every spanning tree $T$ obtained in this way yields an arc presentation with $c(D)+2$ arcs. To ensure this property, Jin and Lee impose two conditions on the trees in the sequence, which we describe now.

The closure $\overline{T_i}$ of $T_i$ is the subgraph of $D$ induced by the vertices of $T_i$. An edge $e$ of $D$ is in $\overline{T_i}$ if both vertices incident to $e$ are in $T_i$.  An edge $f \in E(\overline{T_i}) \setminus E(T_i)$ is \textit{good} if it meets the edge $e_i$ transversely at the vertex $v_i$. An edge $f\in E(\overline{T_i}) \setminus E(T_i)$ is \textit{bad} if it extends the edge $e_i$ at the vertex $v_i$. Figure~\ref{fig:filter} shows an example of a bad edge.

Let $\Gamma$ be an arc in $S^2\setminus D$ whose endpoints are non-adjacent vertices $v_j$ and $v_\ell$. We say $\Gamma$ is a \textit{cutting arc} of $T_i$ if the simple closed curve of  $\Gamma\cup T_i$ has edges of $D\setminus \overline{T_i}$ in both its interior and exterior. Figure~\ref{fig:filter} shows an example of a cutting arc. The sequence of trees $T_0\subset T_1 \subset \cdots \subset T_{c(D)-1}$ is a \textit{good filtered spanning tree} if for each $j< c(D)-1$ the tree $T_j$ has no cutting arc and $\overline{T_j}$ has no bad edge. Since $\overline{T_{c(D)-1}} = D$, it follows that $T_{c(D)-1}$ cannot have a cutting arc and it also follows that $\overline{T_{c(D)-1}}$ has a bad edge. Figure~\ref{fig:tree4} shows how to achieve the tree depicted in Figure~\ref{fig:tree1} as a good filtered spanning tree.

\begin{figure}[h]
\input{filtered.tex}
\caption{In both pictures, a partial filtered spanning tree is indicated by bold edges. On the left, the edge $f$ is bad. On the right, the dotted curve is a cutting arc.}
\label{fig:filter}
\end{figure}
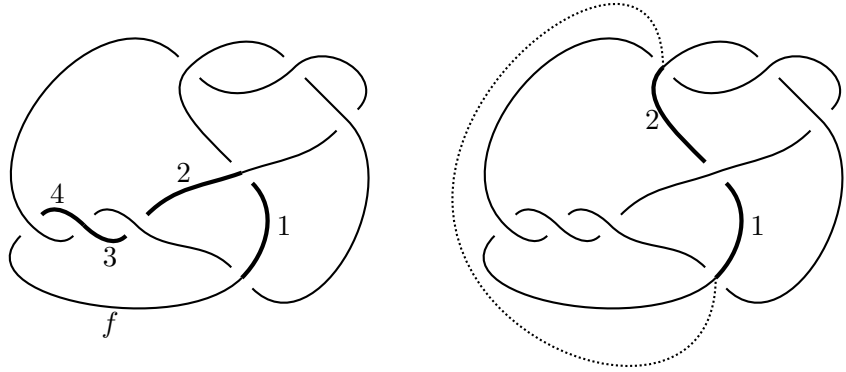

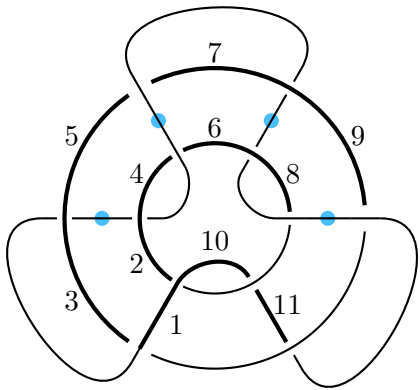
\begin{figure}[h]
\input{tree4.tex}
\caption{The edge $e_i$ of the filtered spanning tree is labeled with $i$. Edges marked with a small disk are reducible edges. }
\label{fig:tree4}
\end{figure}

Jin and Lee \cite{Jin_Lee_2012} proved that, given a filtered spanning tree $T$, the nested exterior arcs - such as those in the bottom right of Figure~\ref{fig:extarcs} - correspond to bad edges in $\overline{T_j}$, and they proved that every prime diagram admits a good filtered spanning tree. Therefore, the good filtered spanning tree construction produces an arc diagram with $c(D)+2$ arcs. Jin and Lee's construction is combinatorially equivalent to an earlier construction of Bae and Park \cite{Bae_Park_2000}.

Furthermore, if $f$ is a non-alternating edge of $D$ that transversely intersects $e_{i-1}$ and $e_i$, then $f$ is reducible. Jin, Lee, and other collaborators \cite{Jin_Lee_2022, Jin_Lee_2025, Jin_Lee_2024_1, Jin_Lee_2024_2} used this technique to find arc presentations and compute the arc index of many knots with crossing number at most 14.

The following result summarizes the above construction.
\begin{proposition}
\label{prop:summary}
Let $D$ be a prime diagram of the link $L$ with $c(D)$ crossings. Then $D$ contains a good filtered spanning tree $T$ with $r(T)$ reducible edges, and the arc index of $L$ satisfies
\[\alpha(L) \leq c(D) + 2 - r(T).\]
\end{proposition}

Applying Proposition~\ref{prop:summary} to the link $L$ in Figures~\ref{fig:tree1} through \ref{fig:tree4} yields $\alpha(L) \leq c(D)+2-r(T) = 12 + 2 - 4 = 10$. One can see this directly from the spoke diagram in Figure~\ref{fig:spoke} or with much less work from the filtered tree diagram in Figure~\ref{fig:tree4}  using Proposition~\ref{prop:summary}. Since $L$ is adequate, Theorem~\ref{thm:adequate} shows that $\alpha(L)=10$.

\subsection{Turaev genus}

Following the discovery of the Jones polynomial \cite{Jones_1987}, Kauffman \cite{Kauffman_1987}, Murasugi \cite{Murasugi_1987}, and Thistlethwaite \cite{Thistlethwaite_1987} independently proved several of Tait's conjectures about alternating knots. A key element in each of their proofs is the inequality $\Span V_L(t) \leq c(L)$ where equality holds if and only if $L$ is alternating. Turaev \cite{Turaev_1987} gave an alternate proof of this inequality using a surface that is now known as the Turaev surface.

The Turaev surface $\Sigma_D$ of the diagram $D$ is constructed as follows. Consider $D$ as a subset of the plane. Embed the all-$A$ and all-$B$ Kauffman states of $D$ just above and below the projection plane, respectively. First build a cobordism between the all-$A$ and all-$B$ states where the cobordism consists of a saddle in the neighborhood of a crossing of $D$, as in Figure~\ref{fig:saddle}, and consists of bands outside of neighborhoods of the crossings. The Turaev surface $\Sigma_D$ is obtained by capping off each boundary component of the cobordism with a disk. The intersection of the Turaev surface $\Sigma_D$ and the projection plane is the diagram $D$. The genus of the Turaev surface is 
\[g_T(D) := g(\Sigma_D) = \frac{1}{2}\left( c(D) + 2 - |s_A D| - |s_B D|\right),\]
where $c(D)$ is the number of crossings in $D$ and $|s_A D|$ and $|s_B D|$ are the number of components in the all-$A$ and all-$B$ states of $D$ respectively. As mentioned in the introduction, the \textit{Turaev genus} $g_T(L)$ of the non-split link $L$ is defined as 
\[g_T(L) = \min \{g_T(D)~|~D~\text{is a diagram of}~L\}.\]
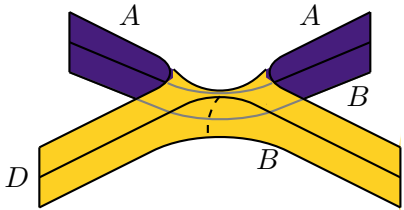
\begin{figure}[h]
        \input{saddle.tex}
    \caption{The cobordism between the all-$A$ and all-$B$ states of $D$ is a saddle in a neighborhood of each crossing of $D$.}
       \label{fig:saddle}
\end{figure}

Turaev showed that a non-split link is alternating if and only if $g_T(L)=0$, and so the Turaev genus of a link can be seen as a measure of how far the link is from being alternating. Kim \cite{Kim_2018} and Armond and Lowrance \cite{Armond_Lowrance_2017} independently showed that $g_T(D)$ only depends on a decomposition of $D$ into maximally alternating tangles. As the number of non-alternating edges of $D$ increases, the quantity $g_T(D)$ tends to increase as well. As mentioned in Proposition~\ref{prop:summary}, non-alternating edges can  be reducible and thus sometimes result in a smaller arc index. This observation was the initial intuition behind exploring the relationship between crossing number, arc index, and Turaev genus. Bae \cite{Bae_2014} also used alternating tangle decompositions to get bounds on arc index. Bae's work on alternating decompositions applies to a proper subset of adequate links.

Turaev genus has been computed for several infinite families of links. Non-alternating pretzel and Montesinos links have Turaev genus one because their standard diagrams $D$ satisfy $g_T(D)=1$. Abe and Kishimoto \cite{Abe_Kishimoto_2010} proved that the Turaev genus of the torus knot $T_{3,q}$ is $g_T(T_{3,q}) = \left\lfloor |q|/3\right\rfloor$. Jin, Lowrance, Polston, and Zhang \cite{JLPZ_2017} computed the Turaev genus for many other torus knots $T_{p,q}$ with $p\leq 6$. Lowrance \cite{Lowrance_2011} also computed the Turaev genus for many closed 3-braids. Abe \cite{Abe_2009} proved that any adequate diagram minimizes Turaev genus.

The Turaev surface or the Turaev genus of a link has connections to the Jones polynomial \cite{Bae_Morton_2003,DFKLS_2008, DFKLS_2010, Dasbach_Lowrance_2018, Lowrance_Spyropoulos_2017}, the colored Jones polynomial \cite{Kalfagianni_2018, Kalfagianni_Lee_2023}, Khovanov homology \cite{Manturov_2003, Manturov_2006, CKS_2007, DL_2014,DL_2020, BDLMV_2024}, and knot Floer homology \cite{Lowrance_2008, DL_2011, JKK_2022}. For more information about the Turaev surface, refer to these survey articles \cite{CK_2014, KK_2021}.

\section{Arc index of adequate links}
\label{sec:adeq}

In this section, we compute the arc index of an adequate link, confirming the conjecture of Park and Seo \cite{Park_Seo_2000}. Using this arc index computation and Abe's computation of the Turaev genus for adequate links \cite{Abe_2009}, we verify Conjecture~\ref{conj:c+2-a_gT} for adequate links.

Park and Seo conjectured that if $D$ is an adequate diagram of a link $L$, then its arc index is given by $\alpha(L) = c(L)+2\rho(D)$, where $\rho(D)$ takes some work to define; we provide the definition below. Shade the complementary regions of $D$ in a checkerboard fashion so that the shading at each crossing looks like $~
\tikz[baseline=.6ex, scale = .4]{
\fill[white!80!black] (0,0) -- (.5,.5) -- (1,0);
\fill[white!80!black] (0,1) -- (.5,.5) -- (1,1);
\draw[thick] (0,0) -- (1,1);
\draw[thick] (1,0) -- (.7,.3);
\draw[thick] (.3,.7) -- (0,1);
}
~$ or $~
\tikz[baseline=.6ex, scale = .4]{
\fill[white!80!black] (0,0) -- (.5,.5) -- (1,0);
\fill[white!80!black] (0,1) -- (.5,.5) -- (1,1);
\draw[thick] (0,0) -- (.3,.3);
\draw[thick] (.7,.7) -- (1,1);
\draw[thick] (1,0) -- (0,1);
}
~$. The checkerboard graph $G$ of $D$ has vertex set corresponding to the shaded regions and edge set corresponding to the crossings. Edges associated with crossings of the form $~
\tikz[baseline=.6ex, scale = .4]{
\fill[white!80!black] (0,0) -- (.5,.5) -- (1,0);
\fill[white!80!black] (0,1) -- (.5,.5) -- (1,1);
\draw[thick] (0,0) -- (1,1);
\draw[thick] (1,0) -- (.7,.3);
\draw[thick] (.3,.7) -- (0,1);
}
~$ are positive, and edges associated with crossings of the form $~
\tikz[baseline=.6ex, scale = .4]{
\fill[white!80!black] (0,0) -- (.5,.5) -- (1,0);
\fill[white!80!black] (0,1) -- (.5,.5) -- (1,1);
\draw[thick] (0,0) -- (.3,.3);
\draw[thick] (.7,.7) -- (1,1);
\draw[thick] (1,0) -- (0,1);
}
~$ are negative.

Define $G_+$ to be the subgraph of $G$ consisting of all the vertices and the positive edges of $G$, and define $\widetilde{G_+}$ to be the graph obtained from $G_+$ by deleting isolated vertices. Similarly define $G_-$ and $\widetilde{G_-}$. The subgraph $\widetilde{G_+}\cap \widetilde{G_-}$ of $G$ is a graph consisting only of vertices, and we let $\left|\widetilde{G_+}\cap \widetilde{G_-}\right|$ denote the number of vertices in $\widetilde{G_+}\cap \widetilde{G_-}$. For any graph $\Gamma$, define $p_0(\Gamma)$ to be the number of components of $\Gamma$, and define the cyclomatic number $p_1(\Gamma)$ of $\Gamma$ by $p_1(\Gamma) = |E(\Gamma)| - |V(\Gamma)| + p_0(\Gamma)$, where $|E(\Gamma)|$ and $|V(\Gamma)|$ denote the number of edges and vertices in $\Gamma$ respectively. If we consider $\Gamma$ as a $1$-dimensional complex, then $p_0(\Gamma)$ and $p_1(\Gamma)$ are the ranks of the homology groups $H_0(\Gamma)$ and $H_1(\Gamma)$ respectively. Finally, Park and Seo \cite{Park_Seo_2000} define
\begin{equation}
\label{eq:rho}
\rho(D) = p_0(\widetilde{G_+}) + p_0(\widetilde{G_-}) - \left| \widetilde{G_+} \cap \widetilde{G_-}\right|.
\end{equation}

Now that $\rho(D)$ is defined, we are ready to show that if $L$ is a non-split adequate link, then its arc index is given by 
\[\alpha(L) = |s_A D| + |s_B D| = c(L) + 2\rho(D),\]
proving Theorem~\ref{thm:adequate}.

\begin{proof}[Proof of Theorem~\ref{thm:adequate}]
Let $L$ be a non-split adequate link with adequate diagram $D$. We first show that $\alpha(L) = |s_A D| + |s_B D|$, and then show that $|s_A D| + |s_B D| = c(L) + 2\rho(D)$. 

Since $D$ is $A$-adequate, it follows from Theorem~\ref{thm:Aadeq} that its maximum Thurston-Bennequin number is $\overline{tb}(L) = w(D) - |s_A D|$, where $w(D)$ is the writhe of $D$. Since $D$ is $B$-adequate, its mirror $D^*$ is $A$-adequate, and hence $L^*$ is an $A$-adequate link. Thus 
\[\overline{tb}(L^*) = w(D^*) - |s_A D^*| = -w(D) - |s_B D|.\]
We have $\alpha(L)= -\overline{tb}(L) - \overline{tb}(L^*)$ by Theorem~\ref{thm:tb}. Therefore,
\begin{align*}
\alpha(L) = & \; -\overline{tb}(L) - \overline{tb}(L^*)\\
= & \; -w(D) + |s_A D| + w(D) + |s_B D|\\
= & \; |s_A D| + |s_B D|.
\end{align*}

Now we show that $|s_A D| + |s_B D | = c(L) + 2\rho(D)$, completing the proof. Lickorish and Thistlethwaite \cite{Lickorish_Thistlethwaite_1988} proved that since $D$ is adequate, $c(L) = c(D)$. 

Let $G$ be the checkerboard graph of $D$, and let $G_\pm$ and $\widetilde{G_\pm}$ be defined as above. Since $p_0(G_+) + p_0(G_-) = p_0(\widetilde{G_+}) + p_0(\widetilde{G_-}) + (|V(G)| - \left| \widetilde{G_+}\cap \widetilde{G_-}\right|)$, it follows that 
\[\rho(D) =  p_0(\widetilde{G_+}) + p_0(\widetilde{G_-}) - \left| \widetilde{G_+} \cap \widetilde{G_-}\right|=  p_0(G_+) + p_0(G_-) - |V(G)|.\]
Thistlethwaite \cite[Prop. 2]{Thistlethwaite_1988} showed that $|s_A D| = p_0(G_+) + p_1(G_+)$ and $|s_B D| = p_0(G_-) + p_1(G_-)$. Therefore
\begin{align*}
|s_A D| + |s_B D| = & \; p_0(G_+) + p_1(G_+) + p_0(G_-) + p_1(G_-)\\
= & \; p_0(G_+) + \left(|E(G_+)| - |V(G_+)| + p_0(G_+) \right)\\
& + p_0(G_-) + \left( |E(G_-)| - |V(G_-)| + p_0(G_-) \right)\\
= & \; |E(G)| + 2p_0(G_+) + 2p_0(G_-) -2|V(G)| \\
= & \; c(D) + 2\rho(D),
\end{align*}
where the third equality follows from the facts that $|V(G_+)|=|V(G_-)|=|V(G)|$ and $|E(G)| = |E(G_+)| + |E(G_-)|$ and the fourth equality follows from the fact that $c(D) = |E(G)|$.
\end{proof}

The formula for the arc index of an adequate link in Theorem~\ref{thm:adequate} implies Conjecture~\ref{conj:c+2-a_gT} for adequate links.

\begin{proposition}
\label{prop:adeq}
If $L$ is an adequate link, then $c(L) + 2 -\alpha(L) = 2g_T(L)$. In particular, Conjecture~\ref{conj:c+2-a_gT} holds for adequate links.
\end{proposition}
\begin{proof}
Let $D$ be an adequate diagram of $L$. Theorem~\ref{thm:adequate} implies that $\alpha(L) = |s_A D| + |s_B D|$. Abe \cite{Abe_2009} proved that $D$ is Turaev genus minimizing, that is,
\[2g_T(L) = 2g_T(D) = c(D) + 2 - |s_A D| - |s_B D|.\]
Since $D$ is adequate, $c(L)=c(D)$, and thus
\[c(L) + 2 - \alpha(L) = c(D) + 2 - |s_A D| - |s_B D| =2 g_T(L).  \qedhere \]
\end{proof}

Our proof of Theorem~\ref{thm:adequate} uses maximum Thurston-Bennequin numbers and Theorems~\ref{thm:tb} and \ref{thm:Aadeq}. However, our initial attempt at computing the arc index for an adequate link used filtered spanning trees and Proposition~\ref{prop:summary}. We hoped to show that if $D$ is an adequate diagram of the link $L$, then there is a good filtered spanning tree $T$ of $D$ such that the number of reducible edges is $r(T)=2g_T(D)=2g_T(L)$, which would imply Theorem~\ref{thm:adequate}. However, using a brute force argument that we omit, one can show that the adequate, Turaev genus two diagram $D$ in Figure~\ref{fig:gt2} has no such spanning tree. Therefore one cannot use Proposition~\ref{prop:summary} alone to prove Conjecture~\ref{conj:c+2-a_gT}. Thankfully, our efforts were not in vain because Proposition~\ref{prop:summary} plays a crucial role in verifying Conjecture~\ref{conj:c+2-a_gT} for closed positive 3-braids in Section~\ref{sec:3braids}.

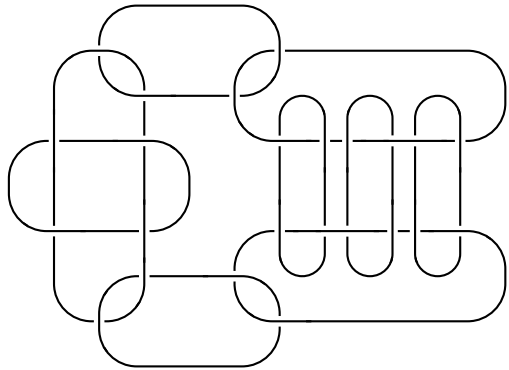
\begin{figure}[h!]
\input{gt2.tex}
\caption{An adequate, Turaev genus two diagram that does not contain a spanning tree with four reducible edges.}
\label{fig:gt2}
\end{figure}

\section{Closed positive 3-braids}
\label{sec:3braids}
 The main goal of this section is to prove Conjecture~\ref{conj:c+2-a_gT} for a link that is the closure of a positive 3-braid. 

\subsection{Garside normal forms}\label{sec:Garside_normal_forms} The \textit{braid group on $n$ strands}, $\mathbb{B}_n$, has the following presentation $$ \mathbb{B}_n=\left\langle\sigma_1, \ldots, \sigma_{n-1} \left\lvert\, \begin{array}{cc}
\sigma_i \sigma_j=\sigma_j \sigma_i, & |i-j|>1 \\
\sigma_i \sigma_j \sigma_i=\sigma_j \sigma_i \sigma_j, & |i-j|=1
\end{array}\right.\right\rangle .$$  

Every braid admits infinitely many words (in the generators $\sigma_i$ and their inverses) that represent it. Let $\mathbb{B}_n^+$ be the submonoid comprising all $n$-braids having a positive representative, also known as \textit{positive braids}. In $\mathbb{B}_n$, a partial order can be defined as follows: given $\alpha, \beta \in \mathbb{B}_n$, $\alpha \preceq \beta$ if there exists $\gamma \in \mathbb{B}_n^+$ such that $\beta = \alpha \gamma$.  With this order, it is possible to prove that, given $\alpha, \beta \in \mathbb{B}_n$, there exists an element $d$ such that $d \preceq \alpha,\beta$ and satisfying $d'\preceq d$ for any other $d' \preceq \alpha,\beta$. The element $d$ constitutes a common greatest divisor of $\alpha$ and $\beta$, which we usually denote by $d = \alpha \wedge \beta$. 

Within every braid group there is a special element $\Delta = \sigma_1 (\sigma_2 \sigma_1) \cdots (\sigma_{n-1} \cdots \sigma_2 \sigma_1) \in \mathbb{B}_n$. This element, known as the \textit{Garside element}, satisfies a number of important properties, such as the two following: for each $i =1, \dots, n-1$, $\sigma_i^{-1}=x_i\Delta^{-1}$ for some $x_i \in \mathbb{B}_n^+$, and $\sigma_i \Delta^{\pm 1} = \Delta^{\pm 1} \sigma_{n-i}$. Combining these two properties allows us to decompose every braid word as $\Delta^{p} x$, with $x \in \mathbb{B}_n^+$. This is refined a bit further in the following result.

\begin{theorem}[{\cite{Elrifai_Morton_1994},\cite{Epstein_1992}}]\label{th:normal_form}
    Every braid can be written uniquely in the form $\Delta^{p} a_1 \cdots a_{\ell}$, where $p \in \mathbb{Z}$, $\ell \geq 0$, $1 \prec a_i \prec \Delta$ for every $i=1, \dots,\ell$, and $a_i a_{i+1} \wedge \Delta = a_i$ for every $i=1, \dots, \ell-1$.
\end{theorem}

\begin{definition}
The \textit{(left) normal form of} $\beta \in \mathbb{B}_n$ is the decomposition $\beta=\Delta^{p} a_1 \cdots a_{\ell}$ given in Theorem~\ref{th:normal_form}. Given such a decomposition, every $a_i$ is called a \textit{simple factor} of $\beta$, and the \textit{infimum} of $\beta$ is defined as $\inf(\beta)=p$. Moreover, denoting the conjugacy class of $\beta$ in $\mathbb{B}_n$ by $\beta^{\mathbb{B}_n}$, the \textit{summit infimum} of $\beta$ is defined as $\inf_s(\beta) = \min \{ \inf (\gamma) \; |\; \gamma \in \beta^{\mathbb{B}_n} \}$.
\end{definition}

As a consequence of the uniqueness of the normal form, a braid is positive if and only if its infimum is greater than or equal to 0. Regarding normal forms in $\mathbb{B}_3$, note that the Garside element is $\Delta = \sigma_1 \sigma_2 \sigma_1 = \sigma_2 \sigma_1 \sigma_2$ and thus the only possible simple factors are $\sigma_1$, $\sigma_2$, $\sigma_1\sigma_2$ and $\sigma_2\sigma_1$. Then, the condition ${a_i a_{i+1} \wedge \Delta = a_i}$ is equivalent to saying that the last letter of $a_i$ coincides with the first letter of $a_{i+1}$.

We now recall two classifications of conjugacy classes of $3$-braids. For convenience, the first classification is presented for 3-braids which are positive, but a more general result is included in \cite{dVV-Gonzalez-Meneses_Silvero_2025}. The second one corresponds to the well-known Murasugi's classification; the version appearing below comes from Baldwin \cite{Baldwin_2008}.

\begin{proposition}\label{prop:positive_infimum_conjugate_5_families}
Every braid in $\mathbb{B}_3^+$ is conjugate to a braid in one of the following families:
\begingroup
\renewcommand{\arraystretch}{1.8}
$$
\begin{array}{l}
\Lambda_1= \left\{\Delta^p \ | \ p \geq 0 \right\}, \\
\Lambda_2= \left\{\Delta^p\sigma_1^{k_1} \ | \ p \geq 0,\ k_1\geq 1\right\}, \\
\Lambda_3= \left\{\Delta^{2u}\sigma_1\sigma_2 \ | \ u \geq 0 \right\}, \\
\Lambda_4= \left\{\Delta^{2u}\sigma_1^{k_1}\sigma_2^{k_2}\cdots \sigma_2^{k_{2t}} \ | \ u \geq 0,\ t\geq 1,\ k_1,\ldots,k_{2t}\geq 2\right\}, \\
\Lambda_5= \left\{\Delta^{2u+1}\sigma_1^{k_1}\sigma_2^{k_2}\cdots \sigma_1^{k_{2t+1}} \ | \ u\geq 0,\ t\geq 1,\ k_1,\ldots,k_{2t+1}\geq 2\right\}. 
\end{array}
$$
\endgroup
\end{proposition}

\begin{theorem}[{\cite[Prop. 2.1]{Murasugi_1974}}]\label{th:Murasugi_classification}
 Every braid in $\mathbb{B}_3$ is conjugate to exactly one of the following:
\begin{enumerate}
    \item $\Delta^{2v} \sigma_1^{a_1} \sigma_2^{-b_1} \sigma_1^{a_2} \sigma_2^{-b_2} \cdots \sigma_1^{a_r} \sigma_2^{-b_r}$, with $v \in \mathbb{Z}$, and $a_i,b_i, r >0$.
    \item $\Delta^{2v} \sigma_1^{a}$, with $v,a \in \mathbb{Z}$.
    \item $\Delta^{2v} \sigma_1^{a} \sigma_2^{-1}$, with $v \in \mathbb{Z}$ and $a \in \{ -1,-2,-3 \}.$
\end{enumerate}
\end{theorem}

In the first classification, all representatives are written in normal form (collecting the powers of the generators instead of indicating the simple factors). However, this is not the case for all the representatives in the second one. We may need a conjugate of a positive 3-braid aligning with one of the two classifications, or even pass from one to the other. Our next purpose is to explicitly obtain conjugates of the braids considered in Theorem~\ref{th:Murasugi_classification} belonging to some family $\Lambda_i$. 

\begin{lemma}\label{lemma:braid_classifications}
Let $\beta\in\mathbb{B}_3$ be a braid in Murasugi's classification.
\begin{enumerate}[label=(\roman*), leftmargin=*, labelindent=0pt]
\item Suppose $\beta = \Delta^{2v}\sigma_1^{a_1}\sigma_2^{-b_1}\sigma_1^{a_2}\sigma_2^{-b_2}\cdots \sigma_1^{a_r}\sigma_2^{-b_r}$ where $\mathbf{b}=\sum_{i=1}^r b_i$. The braid $\beta$ is positive if and only if $2v-\mathbf{b}\geq 0$. If $\beta$ is positive, then $\beta$ is conjugate to $\gamma\in\Lambda_4$ when $\mathbf{b}$ is even and $\beta$ is conjugate to $\gamma\in\Lambda_5$ when $\mathbf{b}$ is odd.
\item Suppose $\beta = \Delta^{2v}\sigma_1^a$. The braid is positive if and only if $a,v\geq 0$, or $a<0$ and $2v-|a|\geq 0$. If $v\geq 0$ and $a=0$, then $\beta\in\Lambda_1$. If $v\geq 0$ and $a > 0$, then $\beta$ is conjugate to $\gamma\in\Lambda_2$. If $a<0$ and $2v-|a|\geq 0$, then $\beta$ is conjugate to $\gamma\in\Lambda_4$ when $a$ is even and $\gamma\in\Lambda_5$ when $a$ is odd.
\item Suppose $\beta = \Delta^{2v} \sigma_1^{a} \sigma_2^{-1}$, with $a \in \{-1,-2,-3\}$. The braid $\beta$ is positive if and only if $2v-|a|\geq 0$. Suppose $\beta$ is positive. If $a=-1$, then $\beta$ is conjugate to $\gamma \in \Lambda_2$. If $a=-2$, then $\beta$ is conjugate to $\gamma \in \Lambda_1$. If $a=-3$, then $\beta$ is conjugate to $\gamma \in \Lambda_3$.
\end{enumerate}
\end{lemma}
\begin{proof}
Let $b$ be a non-negative integer, and let $[b]$ denote 1 if $b$ is even and 2 if $b$ is odd. Note that $\sigma_1^{-1}=\sigma_2\sigma_1\Delta^{-1}$ and $\sigma_2^{-1} = \sigma_1 \sigma_2 \Delta^{-1}$. Then, for every $b > 0$ it is straightforward to check the following equalities:
\begin{equation}\label{generator^{-q}}
     \renewcommand{\arraystretch}{1.6}
    \begin{array}{rcccl}
         \sigma_1^{-b} & = & (\sigma_2 \sigma_1 \Delta^{-1})^b & = & \Delta^{-b} \sigma_{[b+1]} \sigma_{[b]}^2 \sigma_{[b-1]}^2 \cdots \sigma_{[3]}^2 \sigma_{[2]}^2 \sigma_{[1]},    \\
         
         \sigma_2^{-b} & = & (\sigma_1 \sigma_2 \Delta^{-1})^b & = & \Delta^{-b} \sigma_{[b]} \sigma_{[b-1]}^2 \sigma_{[b-2]}^2 \cdots \sigma_{[2]}^2 \sigma_{[1]}^2 \sigma_{[0]}.
    \end{array}
\end{equation}

\begin{enumerate}[label=(\roman*), leftmargin=*, labelindent=0pt]
    \item Let $\beta = \Delta^{2v} \sigma_1^{a_1} \sigma_2^{-b_1} \sigma_1^{a_2} \sigma_2^{-b_2} \cdots \sigma_1^{a_r} \sigma_2^{-b_r}$. Let us use (\ref{generator^{-q}}) to display some subsequent transformations of the given expression for $\beta$:

\begin{small}
    \[ \renewcommand{\arraystretch}{1.6}
    \begin{array}{l}
         \Delta^{2v} \sigma_1^{a_1} \Delta^{-b_1}             (\sigma_{[b_1]} \sigma_{[b_1-1]}^2 \cdots \sigma_{[1]}^2 \sigma_{[0]}) \sigma_1^{a_2} \sigma_2^{-b_2} \cdots \sigma_1^{a_r} \sigma_2^{-b_r},\\

         \Delta^{2v-b_1} (\sigma_{[b_1]}^{a_1+1} \sigma_{[b_1-1]}^2 \cdots \sigma_{[1]}^2) \sigma_{[0]}^{a_2+1}  \sigma_2^{-b_2} \cdots \sigma_1^{a_r} \sigma_2^{-b_r},\\
         
         \Delta^{2v-b_1} (\sigma_{[b_1]}^{a_1+1} \sigma_{[b_1-1]}^2 \cdots \sigma_{[1]}^2) \sigma_{[0]}^{a_2+1}  \Delta^{-b_2} (\sigma_{[b_2]} \sigma_{[b_2-1]}^2 \cdots \sigma_{[1]}^2 \sigma_{[0]}) \sigma_1^{a_3} \sigma_2^{-b_3} \cdots \sigma_1^{a_r} \sigma_2^{-b_r}, \\

          \Delta^{2v-b_1-b_2}  (\sigma_{[b_1+b_2]}^{a_2+1} \sigma_{[b_1+b_2-1]}^2 \cdots \sigma_{[b_2+1]}^2)  \sigma_{[b_2]}^{a_2+2} (\sigma_{[b_2-1]}^2 \cdots \sigma_{[1]}^2 \sigma_{[0]}) \sigma_1^{a_3} \sigma_2^{-b_3} \cdots \sigma_1^{a_r} \sigma_2^{-b_r}, \\

          \vdots \\

          \Delta^{2v-\mathbf{b}} \left( \sigma_{[\mathbf{b}]}^{a_1+1} \sigma_{[\mathbf{b}-1]}^2 \cdots \sigma_{[\mathbf{b}-b_1+1]}^2 \right) \left( \sigma_{[\mathbf{b}-b_1]}^{a_2+2} \sigma_{[\mathbf{b}-b_1-1]}^2 \cdots \sigma_{[\mathbf{b}-b_1-b_2+1]}^2 \right) \cdots \left( \sigma_{[b_r]}^{a_r+2} \sigma_{[b_r-1]}^2 \cdots \sigma_{[1]}^2  \right) \sigma_{[0]}.
    \end{array}
    \]
 \end{small}

The last expression is actually the normal form of $\beta$, so the braid would be positive if and only if $2v-\mathbf{b} \geq 0$. If this is the case, conjugating the last expression by $\sigma_{[0]}=\sigma_1$, we obtain  \[ \gamma = \Delta^{2v-\mathbf{b}} \left( \sigma_{[\mathbf{b}]}^{a_1+2} \sigma_{[\mathbf{b}-1]}^2 \cdots \sigma_{[\mathbf{b}-b_1+1]}^2 \right) \left( \sigma_{[\mathbf{b}-b_1]}^{a_2+2} \sigma_{[\mathbf{b}-b_1-1]}^2 \cdots \sigma_{[\mathbf{b}-b_1-b_2+1]}^2 \right) \cdots \left( \sigma_{[b_r]}^{a_r+2} \sigma_{[b_r-1]}^2 \cdots \sigma_{[1]}^2  \right), \] which is a conjugate of $\beta$ belonging to $\Lambda_4$ if $\mathbf{b}$ is even or to $\Lambda_5$ if $\mathbf{b}$ is odd.

    \item Let $\beta = \Delta^{2v} \sigma_1^a$. If $a \geq 0$, then $\beta$ is written in normal form, and it is positive if and only if $v \geq 0$. If this is the case, $\beta \in \Lambda_1$ when $a=0$ and  $\beta \in \Lambda_2$ when $a>0$.
    
    Otherwise, $a=-|a|<0$ and, by applying the same method as in the previous case, the normal form of $\beta$ is
    \[
    \Delta^{2v-|a|} \sigma_{[|a|]} \sigma_{[|a|-1]}^2 \cdots \sigma_{[1]}^2 \sigma_{[0]} .
    \]
    In this case, $\beta$ is a positive braid if and only if $2v-|a| \geq 0$. If this is the case, conjugating the last expression by $\sigma_{[0]}=\sigma_1$, we obtain \[ \gamma = \Delta^{2v-|a|} \sigma_{[|a|]}^2 \sigma_{[|a|-1]}^2 \cdots \sigma_{[1]}^2 , \] which is a conjugate of $\beta$ belonging to $\Lambda_4$ if $a$ is even or to $\Lambda_5$ if $a$ is odd.

   \item Let $\beta = \Delta^{2v} \sigma_1^{a} \sigma_2^{-1}$, with $a \in \{-1,-2,-3\}$. Doing the same as in the previous cases, the normal form of $\beta$ would be:
    \begin{enumerate}
        \item[(iii.a)] If $a=-1$, $\Delta^{2v-1} \sigma_1$. Then, $\beta$ is a positive braid if and only if $2v-1 \geq 0$. If this is the case, $\beta$ belongs to $\Lambda_2$.
        \item[(iii.b)] If $a=-2$, $ \Delta^{2v-2} \sigma_2 \sigma_1^2$. Then $\beta$ is a positive braid if and only if $2v - 2 \geq 0$. If this is the case, conjugating the normal form by $\sigma_1$, we obtain $ \gamma = \Delta^{2v-1}$, which is a conjugate of $\beta$ belonging to $\Lambda_1$.
        \item[(iii.c)] If $a=-3$, $\Delta^{2v-3} \sigma_1 \sigma_2^2 \sigma_1^2$. Then $\beta$ is a positive braid if and only if $2v-3 \geq 0$. If this is the case, conjugating the normal form by $\Delta\sigma_1$, we obtain $ \gamma = \Delta^{2v-2} \sigma_1 \sigma_2$, which is a conjugate of $\beta$ belonging to $\Lambda_3$. \qedhere

    \end{enumerate}

\end{enumerate}

\end{proof}

\subsection{Crossing number and Turaev genus of closed 3-braids} Let $\widehat{\beta}$ be the closure of the braid $\beta$. An $n$-braid word $\sigma_{i_1}^{\varepsilon_1} \sigma_{i_2}^{\varepsilon_2} \cdots \sigma_{i_m}^{\varepsilon_m}$, with $\varepsilon_x = \pm 1$ for every $1 \leq x \leq m$, is said to be \textit{homogeneous} if $\varepsilon_{x} = \varepsilon_{x'}$ whenever $i_{x} = i_{x'}$. In this sense, a braid is \textit{homogeneous} if it admits a homogeneous representative.

\begin{theorem}[{\cite[Prop. 7.4]{Murasugi_1991}}]\label{th:crossing_number_braid_closures}
    Let $\beta$ be a homogeneous $n$-braid and $\sigma_{i_1}^{\varepsilon_1} \sigma_{i_2}^{\varepsilon_2} \cdots \sigma_{i_r}^{\varepsilon_r}$ a homogeneous representative of $\beta$. If the braid index of $\widehat{\beta}$ is $n$, then the crossing number of $\widehat{\beta}$ is given by
    \[c(\widehat{\beta}) = \sum_{x=1}^r |\varepsilon_x|.\]
\end{theorem}

\begin{theorem}[{\cite[Cor. 5.2]{Gonzalez-Meneses_Manchon_2014}}]\label{th:braid_index_closed_3-braids} 
    Given a positive 3-braid $\beta$, the braid index of $\widehat{\beta}$ is smaller than three if and only if $\beta$ is conjugate to $\sigma_1^k\sigma_2$ for some $k \geq 0$.
\end{theorem}

The two previous results allow us to compute the crossing number of $\widehat{\beta}$ from the normal form of $\beta \in\mathbb{B}_3^+$ just by adding all the exponents involved, except for those braids conjugate to $\sigma_1^k\sigma_2$. The closures of these last ones correspond to the torus links $T(2,k)$, whose crossing number is $k$. Knowing the crossing number of closed positive $3$-braids plays an important role in the proof of Conjecture~\ref{conj:c+2-a_gT} for this family, but we also need to know, or at least estimate, the Turaev genus, for which we will use the following result.

\begin{theorem}[{\cite[Prop. 4.15 - 4.17]{Lowrance_2011}}]\label{th:turaev_genus_closed_3-braids}
If $\beta$ is a braid listed in Murasugi's classification, then
\[ |v|-1 \leq g_T(\widehat{\beta}) \leq |v|. \]
Additionally, given a braid $\beta$ in the third family of Murasugi's Classification, $g_T(\widehat{\beta}) = v-1$ if $v \geq 1$.
\end{theorem}

\subsection{Arc index of closed positive 3-braids}

In this section we estimate the arc index of any closed positive 3-braid.
We give lower bounds on the arc index using Thurston-Bennequin numbers, which we in turn estimate by finding certain nontrivial classes in Khovanov homology.  We give upper bounds of the arc index by using good filtered spanning trees and Proposition~\ref{prop:summary} on diagrams derived from representatives in Proposition~\ref{prop:positive_infimum_conjugate_5_families}. 

The Khovanov homology $H(L)=\bigoplus_{i,j\in\mathbb{Z}} H^{i,j}(L)$ of a link $L$,
where $H^{i,j}(L)$ denotes the summand in homological grading $i$ and polynomial grading $j$, constitutes a categorification of the Jones polynomial $V_L$. In fact, 
\[\sum_{i,j\in\mathbb{Z}} (-1)^i \; \operatorname{rank} H^{i,j}(L) \cdot t^j = (t+t^{-1}) V_L(t^2).\]

While we will not give a detailed construction of Khovanov homology, we do remind the reader of some Khovanov homology notation. For an in-depth construction of $H(L)$, see \cite{Khovanov_2000, BN_2002,Viro2_2004}. For an oriented diagram $D$ of the link $L$, define $H(D)$ to be $H(L)$ and let $C(D)$ be the standard cochain complex whose homology is $H(L)$. Define
\[\begin{array}{rclcrcl}
              i_{\min}(D) & = & \min \{ i \mid C^{i,j}(D) \neq 0 \}, & \quad & j_{\min}(D) & = & \min \{ j \mid C^{i,j}(D) \neq 0 \},  \\
         \underline{i}(D) & = & \min \{ i \mid H^{i,j}(D) \neq 0 \}, & \quad & \underline{j}(D) & = & \min \{ j \mid H^{i,j}(D) \neq 0 \},  \\
              i_{\max}(D) & = & \max \{ i \mid C^{i,j}(D) \neq 0 \}, & \quad &  j_{\max}(D) & = & \max \{ j \mid C^{i,j}(D) \neq 0 \},  \\
          \overline{i}(D) & = & \max \{ i \mid H^{i,j}(D) \neq 0 \}, & \quad & \overline{j}(D) & = & \max \{ j \mid H^{i,j}(D) \neq 0 \}.
\end{array} \]
The maximum and minimum polynomial gradings in $C(D)$ are given in the following proposition.

\begin{proposition}[{\cite[Cor. 4.2]{Gonzalez-Meneses_Manchon_Silvero_2018}}]\label{prop:j_min_j_max}
    For any oriented link diagram $D$, 
    \[j_{\min }(D)=c(D)-3c_-(D)-\left|s_A D\right| \quad \text{ and } \quad j_{\max }(D) = -c(D) + 3c_+(D) + \left|s_B D\right|,\]
    where $c_{\pm}(D)$ denotes the number of positive and negative crossings in $D$.
\end{proposition}

In \cite{Ng_2005} Ng showed that
\begin{equation}\label{eq:Ng_2005}
\overline{tb}(L) \leq \min\{j-i~|~H^{i,j}(L)\neq 0\},
\end{equation}
and as the arc index admits a decomposition in terms of Thurston-Bennequin numbers (see Theorem~\ref{thm:tb}), it will be useful for us to identify nontrivial classes in Khovanov homology. We use two methods that potentially yield nontrivial Khovanov classes. The first method always gives a nontrivial homology class when the diagram is $A$- or $B$-adequate.

\begin{proposition}[{\cite[Prop. 6.1]{Khovanov_2003}, \cite[Th. 2.4]{DL_2020}}]
\label{prop:adeqKh}
If $D$ is an $A$-adequate diagram of the link $L$, then $j_{\min}(D) = \underline{j}(D)$ and
 \[\bigoplus_{i\in\mathbb{Z}} H^{i, \underline{j}(D)} = H^{-c_-(D),\underline{j}(D)}(L) \cong \mathbb{Z}.\]
If $D$ is a $B$-adequate diagram of the link $L$, then $j_{\max}(D)=\overline{j}(D)$ and \[\bigoplus_{i\in\mathbb{Z}} H^{i,\overline{j}(D)}=H^{c_+(D)-c_-(D),\overline{j}(D)}(L)\cong \mathbb{Z}.\]
\end{proposition}

The second method results in a computation of Khovanov homology in its extreme polynomial grading $j_{\max}(D)$ for any link diagram $D$, but does not guarantee nontriviality. Gonz\'alez-Meneses, Manch\'on, and Silvero \cite{Gonzalez-Meneses_Manchon_Silvero_2018} developed this construction of extreme Khovanov homology based on similar work for the extreme coefficients of the Jones polynomial by Manch\'on \cite{Manchon_2004}.

For any graph $G$, a subset $\Omega \subseteq V(G)$ of the vertices of $G$ is \textit{independent} if $\Omega$ does not contain a pair of adjacent vertices. The empty set $\emptyset$ is considered to be independent. The  \textit{independence simplicial complex} $X_G$ of $G$ is the simplicial complex whose simplices are the independent subsets of $G$. The dimension of a simplex of $X_G$ is considered to be the number of vertices minus one.

There is a (standard) cochain complex associated to $X_G$
\begin{equation}\label{cochain_complex_X_G} 
\cdots \longrightarrow C^i(X_G) \xrightarrow{\; \delta_i \;} C^{i+1}(X_G) \longrightarrow \cdots, 
\end{equation}
where $C^i(X_G)$ is defined to be the free $\mathbb{Z}$-module with basis the set containing all the simplices of dimension $i$, and the (standard) differential $\delta_i$ is given by $\delta_i(\Omega) = \sum_{v} (-1)^{k(\Omega,v)} \cup \{ v \}$, where the sum ranges over vertices of $G$ that are not adjacent to any vertex in $\Omega$, and $k(\Omega,v)$ is the number of vertices of $\Omega$ coming after $v$ in a fixed ordering of the vertex set of $G$ \cite{Gonzalez-Meneses_Manchon_Silvero_2018}.  Let $\{H^i(X_G)\}_i$ denote the cohomology of the cochain complex \eqref{cochain_complex_X_G}. 

Decorate the all-$B$ state by replacing each crossing 
$\tikz[baseline=.6ex, scale = .4]{\draw (0,1) -- (.3,.7); \draw (.7,.3) -- (1,0);\draw (0,0) -- (1,1);}$ 
with its $B$-resolution and an arc, called a \textit{$B$-chord}, connecting the two arcs of the $B$-resolution 
$\tikz[baseline=.6ex, scale = .4]{\draw[red, thick] (.5,.1) -- (.5,.9) ; \draw[rounded corners = 1.5mm] (0,0) -- (.5,.45) -- (1,0); \draw[rounded corners = 1.5mm] (0,1) -- (.5,.55) -- (1,1);}$. 
The $B$-chords in $s_BD$ connecting a component to itself are called \textit{admissible}. The \textit{$B$-Lando graph} of $D$, denoted by $G_D^{B}$, is constructed from $s_BD$ as follows: one vertex for each admissible $B$-chord, and one edge between two vertices if the corresponding $B$-chords are incident to the same component of $s_B D$ and have endpoints that alternate as one travels around that component.

Gonz\'alez-Meneses, Manch\'on, and Silvero showed that this cohomology of the independence complex of the $B$-Lando graph of $D$ is isomorphic to the extreme Khovanov homology of $D$.
\begin{theorem}[{\cite[Th. 4.4]{Gonzalez-Meneses_Manchon_Silvero_2018}}]\label{geometrization_extreme_Kh}
   Let L be an oriented link represented by a diagram D having $c_+$ positive crossings. Let $G_D^B$ be the $B$-Lando graph of $D$. Then
\[H^{i,j_{\max}(D)}(D)\cong H^{c_+ -1-i}(X_{G_D^B}).\]
\end{theorem}

Our computations are aided by the following result that determines the independence cohomology of a path.
\begin{proposition}[{\cite{Kozlov_1999}}, {\cite[Cor. 3.4]{Przytycki_Silvero_2018}}]\label{indepence_complex_path}
 Let $P_n$ denote a path with $n+1$ vertices and $n$ edges. Then $X_{P_n}$ is homotopy equivalent to a point $*$ if $n=3k$, and to $S^k$ if $n=3k+1, 3k+2$. Hence, 
\[ H^i(X_{P_n}) \cong
\left\{ \begin{array}{ll}
    H^i(*)  & \text{if } \; n=3k,\\
   H^i(S^k) & \text{if } \; n=3k+1, \ 3k+2,
\end{array} \right. \]
where $H(*)$ and  $H(S^k)$ correspond to the classical (reduced) cohomology of $*$ and $S^k$, respectively.
\end{proposition}

We use Theorem~\ref{geometrization_extreme_Kh} to compute some Khovanov homology classes in the closures of positive $3$-braids.

\begin{proposition}\label{prop:upper_extreme_Kh_3-braids} The (upper) extreme Khovanov homology of the closure of a braid in a family $\Lambda_i$ follows the structure depicted in Table~\ref{tab:extreme_Kh_positive_3-braids}.
\begin{table}[h!]
    \centering
    \begin{footnotesize}

 \begin{tabular}{|c|c|c|c|} \hline
 & & & \\ [-1em]
    $\beta$ & $\overline{j}(\widehat{\beta})$ & $H^{*,j}(\widehat{\beta})$ for $j=\overline{j}(\widehat{\beta})$  & $i_0$ \\  \hline
     & & & \\ [-1em]
    
         $\Delta^{p}$  &    $\begin{array}{rl}
                p=0: & 3    \\
                p>0 \text{ odd}: & 6p    \\
                p>0 \text{ even}: & 6p+1 
                \end{array}$ & $\begin{array}{rl}
             p=0,1: &  H^{i_0,j}(\widehat{\beta}) \cong \mathbb{Z}   \\
             p>1  \text{ odd}:  &  H^{i_0-1,j}(\widehat{\beta}) \oplus H^{i_0,j}(\widehat{\beta}) \cong \mathbb{Z} \oplus \mathbb{Z} \\
             p>1  \text{ even}: & H^{i_0,j}(\widehat{\beta}) \cong \mathbb{Z}^2  
        \end{array}$ & $2p$ \\ \hline
         & & & \\ [-1em]
        
         $\Delta^{p}\sigma_1^{k_1}$  & $\begin{array}{rl}         
               p \text{ odd}: & 6p + 3k_1 \\ 
              p \text{ even}: & 6p + 3k_1+1   \\
        \end{array}$ & $H^{i_0,j}(\widehat{\beta}) \cong \mathbb{Z}$ & $2p + k_1$ \\ \hline 
        
      $\Delta^{2u}\sigma_1\sigma_2$   & $\begin{array}{rl}
              u=0: & 1   \\
              u>0: & 12u+5  
        \end{array}$ & $H^{i_0,j}(\widehat{\beta}) \cong \mathbb{Z}$ & $\begin{array}{rl}
             u=0: & 0   \\
             u>0: & 4u+1  
        \end{array}$ \\ \hline
         & & & \\ [-1em]
        
         $\Delta^{2u}\sigma_1^{k_1}\sigma_2^{k_2}\cdots \sigma_2^{k_{2t}}$ & $12u + 3\sum_{i=1}^{2t} k_i -2t +1$ & $H^{i_0,j}(\widehat{\beta}) \cong \mathbb{Z}$ & $4u + \sum_{i=1}^{2t} k_i$  \\ \hline
          & & & \\ [-1em]
         
         $\Delta^{2u+1}\sigma_1^{k_1}\sigma_2^{k_2}\cdots \sigma_1^{k_{2t+1}}$ & $12u + 3\sum_{i=1}^{2t+1} k_i - 2t +6$ & $H^{i_0,j}(\widehat{\beta}) \cong \mathbb{Z}$ & $4u + 2 + \sum_{i=1}^{2t+1} k_i$ \\ \hline
    \end{tabular}

    \end{footnotesize}
    \caption{Extreme Khovanov homology of closed positive 3-braids.}
    \label{tab:extreme_Kh_positive_3-braids}
\end{table}
\end{proposition}
\begin{proof}
    For the first row in Table~\ref{tab:extreme_Kh_positive_3-braids}, in {\cite[Cor. 5.7]{Chandler_Lowrance_Sazdanovic_Summers_2022}} the Khovanov homology of links $\widehat{\Delta^p}$ is computed, and all for which $p\geq 0$ have the expected extreme homology, in degrees $\overline{i}(\widehat{\beta})=2p$ and $$\overline{j}(\widehat{\beta})= \left\{          \begin{array}{ll}
                3    & \text{if } \ p=0; \\
                6p   & \text{if } \ p>0 \; \text{odd}; \\
                6p+1 & \text{if } \ p>0 \; \text{even}.
            \end{array} \right.$$

    For the second row in Table~\ref{tab:extreme_Kh_positive_3-braids}, let $D$ be the diagram derived from the given word with the convention $\Delta^{p}=(\sigma_2\sigma_1)^{3(\frac{p-1}{2})+1} \sigma_2$ if $p$ is odd, as in Figure~\ref{fig:Delta_p_sigma1_k1_p_odd}. Similarly, let $D'$ be the diagram derived by using the convention $\Delta^{p}=(\sigma_1\sigma_2)^{3(\frac{p}{2})}$ if $p$ is even, as in Figure~\ref{fig:Delta_p_sigma1_k1_p_even}. By Proposition~\ref{prop:j_min_j_max}, we have $j_{\max}(D)=6p+3k_1$ and $j_{\max}(D')=6p+3k_1+1$. Note that $G_D^B = P_{3(p-1)+2}$, $G_{D'}^B = P_{3(p-1)+1}$, and hence by Theorem~\ref{geometrization_extreme_Kh} and Proposition~\ref{indepence_complex_path}, we have
     \[ H^{i,j_{\max}(D)}(D) \cong H^{3p+k_1-1-i}(X_{G_D^B}) \cong H^{3p+k_1-1-i}(S^{p-1}), \] \[ H^{i,j_{\max}(D')}(D') \cong H^{3p+k_1-1-i}(X_{G_{D'}^B}) \cong H^{3p+k_1-1-i}(S^{p-1}). \]
    Therefore, $H^{i,j_{\max}(D)}(D)$ and $H^{i,j_{\max}(D')}(D')$ are nontrivial when $i=i_0=2p + k_1$. It also follows that $\overline{j}(D)=j_{\max}(D)$ and $\overline{j}(D')= j_{\max}(D')$.

    For the third row in Table~\ref{tab:extreme_Kh_positive_3-braids}, in {\cite[Cor. 5.7]{Chandler_Lowrance_Sazdanovic_Summers_2022}} the Khovanov homology of closures of braids of the form $\Delta^{2u}\sigma_1\sigma_2$ is computed, and all for which $u>0$ have the expected extreme homology, in degrees $\overline{i}(\widehat{\beta})=4u+1$ and $\overline{j}(\widehat{\beta})=12u+5$. For $u=0$, $\widehat{\beta}$ is the unknot, whose unique nontrivial Khovanov homology groups are $H^{0,-1}(\bigcirc) \cong H^{0,1}(\bigcirc) \cong \mathbb{Z}$, and the result holds.

    For the fourth row in Table~\ref{tab:extreme_Kh_positive_3-braids}, do the same as for the second row. Note that if $D$ is as in Figure~\ref{fig:Lambda_4}, then $G_D^B=P_{3(2u-1)+2}$ if $u > 0$. For $u=0$, $D$ is adequate and hence $\overline{j}(\widehat{\beta}) = j_{\max}(D) = 3\left( \sum_{i=1}^{2t} k_i \right) -2t +1$ and $H^{*, \overline{j}(\widehat{\beta})}(\widehat{\beta}) =H^{c_+(D),\overline{j}(\widehat{\beta})} \cong \mathbb{Z}$, so $i_0=c_+(D)=\sum_{i=1}^{2t} k_i$ and the statement holds (see Proposition~\ref{prop:adeqKh}). 

    For the fifth row in Table~\ref{tab:extreme_Kh_positive_3-braids}, use the same argument as for the second and fourth rows. Note that if $D$ is as in Figure~\ref{fig:Lambda_5}, then $G_D^B=P_{3(2u)+2}$. 
    \end{proof}
    
    \begin{figure}[h!]
    \begin{subfigure}[b]{.47\textwidth}
        \centering
        \includegraphics[width=0.999\linewidth]{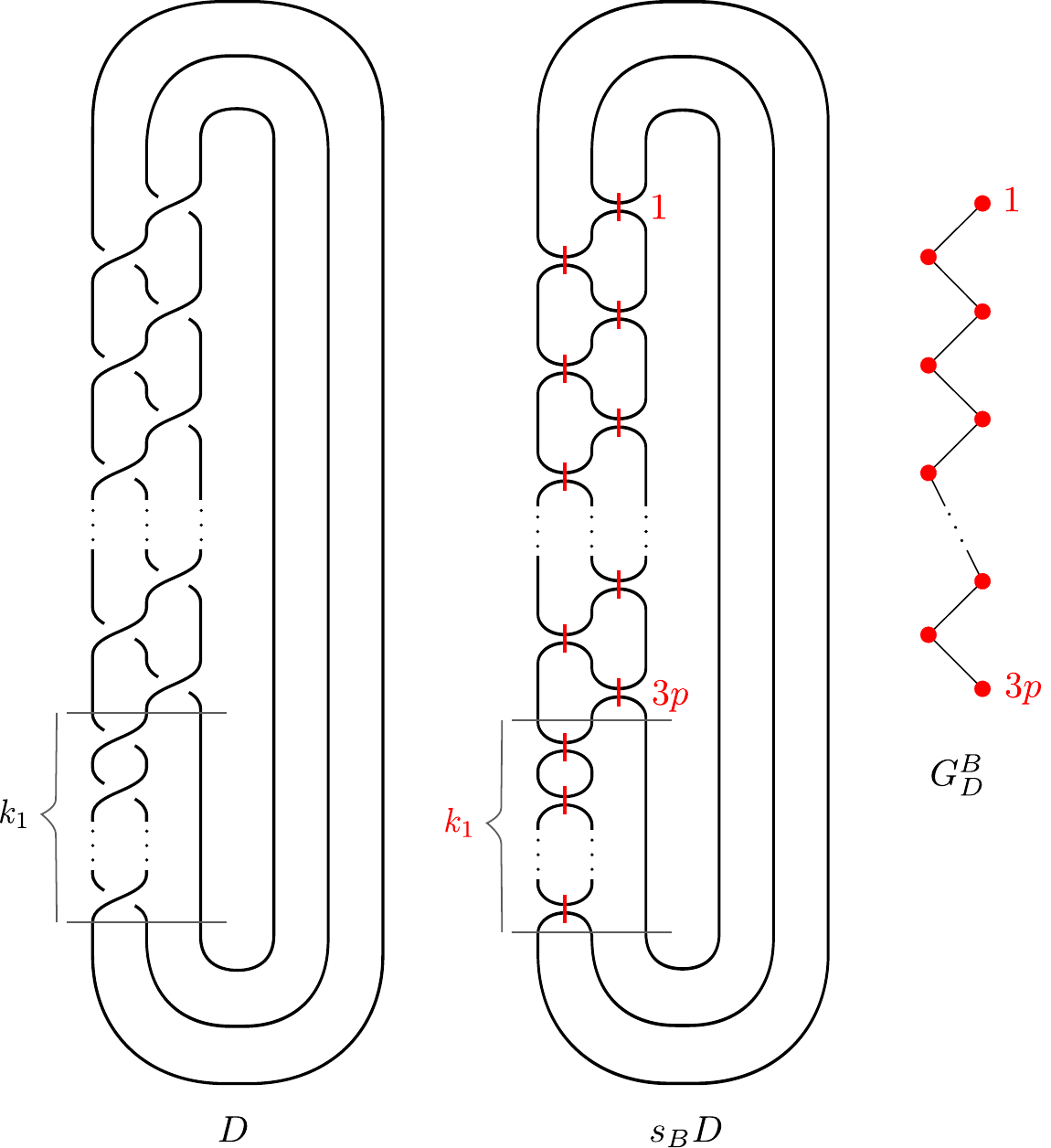}
        \caption{$p$ odd}
        \label{fig:Delta_p_sigma1_k1_p_odd}
    \end{subfigure}
    \begin{subfigure}[b]{.47\textwidth}
        \centering
        \includegraphics[width=0.999\linewidth]{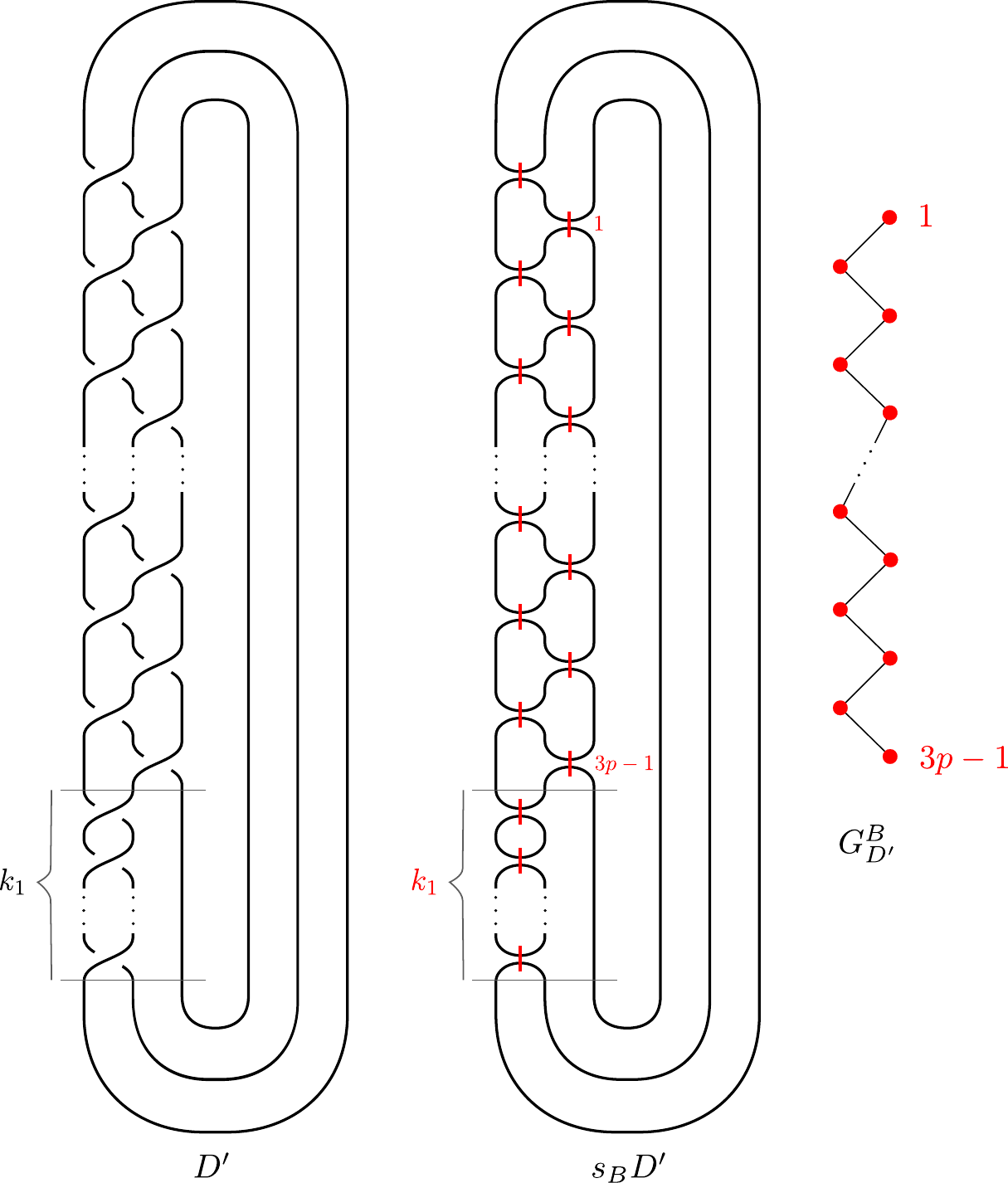}
        \caption{$p$ even}
        \label{fig:Delta_p_sigma1_k1_p_even}
    \end{subfigure}
    \caption{Diagrams $D$ and $D'$ (of $\beta = \Delta^p \sigma_1^{k_1}$, depending on the parity of $p$), their all-$B$ states, and their $B$-Lando graphs.}
    \label{fig:diagram1_Lambda5}
\end{figure}

     \begin{figure}[h!]
    \begin{subfigure}[b]{.47\textwidth}
        \centering
        \includegraphics[width=0.999\linewidth]{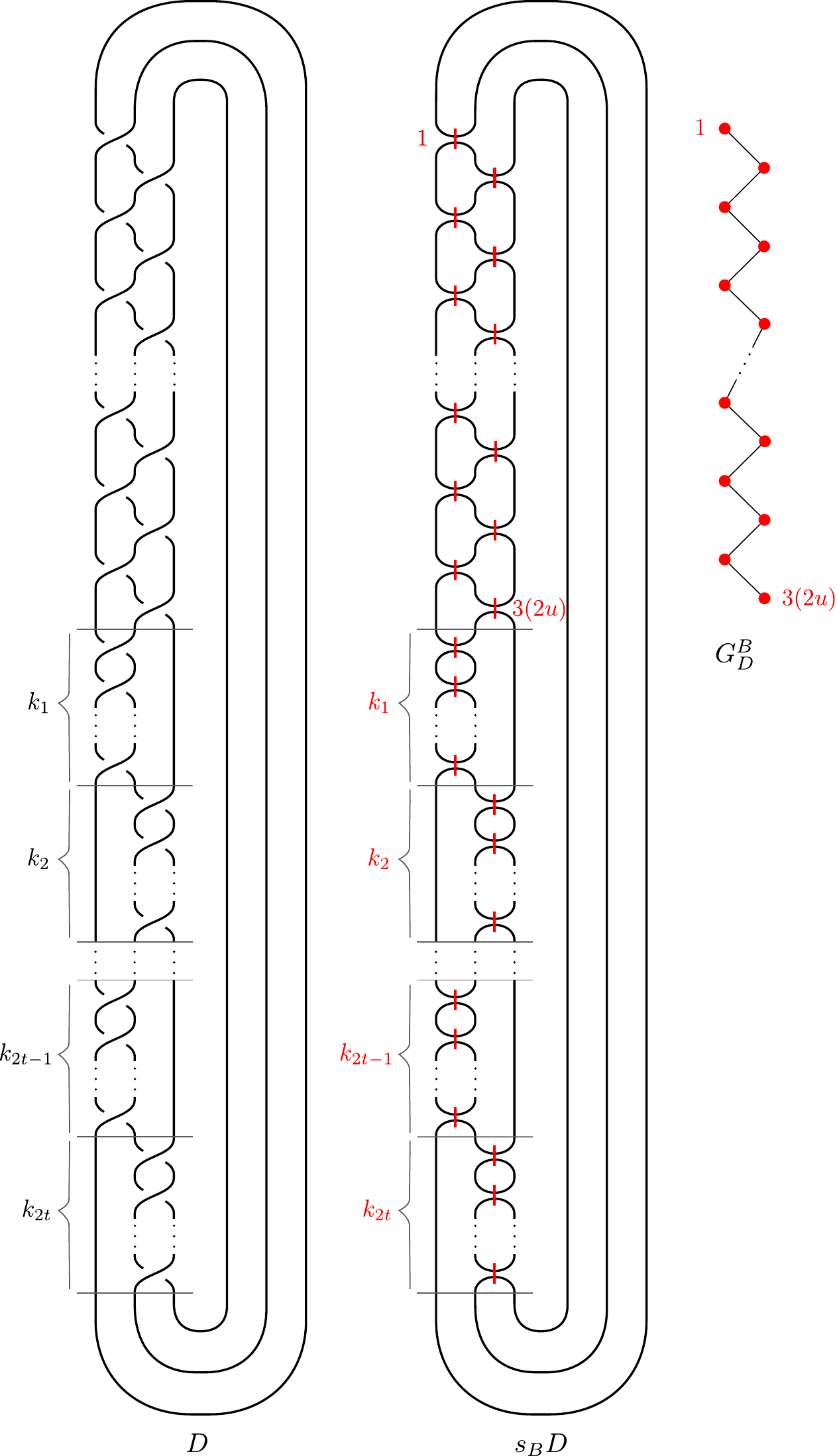}
        \caption{}
        \label{fig:Lambda_4}
    \end{subfigure}
    \begin{subfigure}[b]{.47\textwidth}
        \centering
        \includegraphics[width=0.999\linewidth]{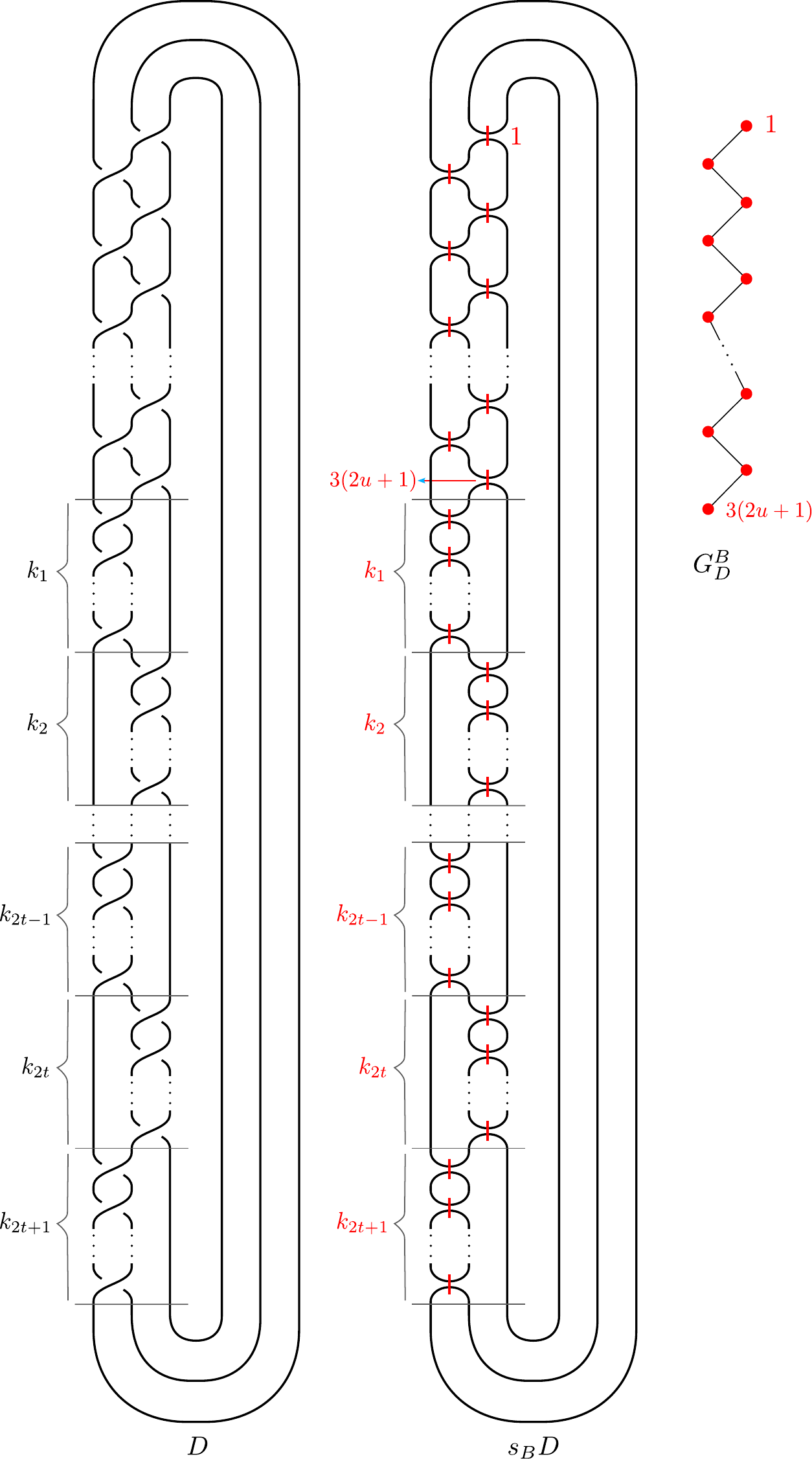}
        \caption{}
        \label{fig:Lambda_5}
    \end{subfigure}
    \caption{Generic diagram of $\widehat{\beta} $, with $\beta \in \Lambda_4$ (a) and $ \beta \in \Lambda_5$ (b), their all-$B$ states and their $B$-Lando graphs.}
    \label{fig:diagram1_Lambda5}
\end{figure}

The next theorem gives bounds on the arc index of the closure of a positive $3$-braid in terms of the classification in Proposition~\ref{prop:positive_infimum_conjugate_5_families}.

\newpage

\begin{theorem}\label{th:arc_index_closed_positive_3-braids}
    	Let $\beta$ be a 3-braid in $\Lambda_i$, for some $i\in \{1,2,3,4,5\}$.
    	\begin{enumerate}
    	\item Let $\beta = \Delta^p \in \Lambda_1$. Then
    	\[\begin{array}{l l}
    	\alpha(\widehat{\beta})=6 & \text{if $p=0$,}\\
	& \\ [-1em]
     	p+4 \leq \alpha(\widehat{\beta}) \leq \frac{1}{2}(3p+1)+3 & \text{if $p>0$ and odd,}\\
	& \\ [-1em]
	p+4\leq \alpha(\widehat{\beta}) \leq \frac{1}{2}(3p)+3 & \text{if    $p>0$ and even.}
    \end{array}\]

    \item  Let $\beta = \Delta^{p} \sigma_1^{k_1} \in \Lambda_2$. Then
     \[ \begin{array}{ll}
    \alpha(\widehat{\beta}) =  k_1+4   & \text{if $p=0$,}\\
        & \\ [-1em]
    p+k_1+3 \leq \alpha(\widehat{\beta}) \leq \frac{1}{2}(3p+1)+k_1+2  & \text{if  $p>0$ and odd,}\\ 
        & \\ [-1em]
     p+k_1+4 \leq \alpha(\widehat{\beta}) \leq \frac{1}{2}(3p)+k_1+3 & \text{if $p>0$ and even.}
    \end{array}\]
    
 \item Let $\beta=\Delta^{2u} \sigma_1 \sigma_2 \in \Lambda_3$. Then
 \[\begin{array}{l l}
 \alpha(\widehat{\beta}) = 2 & \text{if $u=0$,}\\
 & \\ [-1em]
 2u + 5 \leq \alpha(\widehat{\beta}) \leq 3u+4 & \text{if $u>0$.}
 \end{array}\]

    \item Let $\beta = \Delta^{2u}\sigma_1^{k_1}\sigma_2^{k_2}\cdots \sigma_2^{k_{2t}} \in \Lambda_4$. Then
    \[\begin{array}{l l}
     \alpha(\widehat{\beta}) = \sum_{i=1}^{2t} k_i - 2t +4 & \text{if $u=0$,}\\
     & \\ [-1em]
     2u+ \sum_{i=1}^{2t} k_i - 2t +4  \leq \alpha(\widehat{\beta}) \leq 3u + \sum_{i=1}^{2t} k_i - 2t + 3 & \text{if $u>0$.}
     \end{array}\]
    
    \item Let $\beta = \Delta^{2u+1}\sigma_1^{k_1}\sigma_2^{k_2}\cdots \sigma_1^{k_{2t+1}} \in \Lambda_5$. Then
    \[ \textstyle 2u+ \sum_{i=1}^{2t+1} k_i - 2t +4 \leq \alpha(\widehat{\beta}) \leq 3u + \sum_{i=1}^{2t+1} k_i -2t + 4. \]

    \end{enumerate}
\end{theorem}
\begin{proof}
We follow a common strategy for almost every case. Unless otherwise stated, let $D$ be the diagram derived from the normal form of $\beta$; note that in each case we should set a convention for $\Delta^p$. 

As $D$ is positive and hence $A$-adequate, by Theorem~\ref{thm:Aadeq} we have $ \overline{tb}(\widehat{\beta}) = c(D) - 3$. Let us define $\mathbf{m}= \min\{ j-i \; |\; H^{i,j}(\widehat{\beta}^*) \neq 0\}$. Inequality (\ref{eq:Ng_2005}) implies $\overline{tb}(\widehat{\beta}^*) \leq \mathbf{m}$. Furthermore, by \cite[Cor.~11]{Khovanov_2000} we have $H^{i,j}(\widehat{\beta}^*) \cong H^{-i,-j}(\widehat{\beta})$ if these are free $\mathbb{Z}$-modules. Then, in the setting of Proposition~\ref{prop:upper_extreme_Kh_3-braids}, we have 
\begin{equation}\label{tb_eq2}
    \overline{tb}(\widehat{\beta}^*) \leq \mathbf{m} \leq - \overline{j}(\widehat{\beta}) + i_0(\widehat{\beta}).
\end{equation}
By Theorem~\ref{thm:tb}, we conclude that
\begin{equation}\label{lower_bound_alpha_proof}
    \alpha(\widehat{\beta}) =  -\overline{tb}(\widehat{\beta}) -  \overline{tb}(\widehat{\beta}^*) \geq  - c(D) +3+ \overline{j}(\widehat{\beta}) - i_0(\widehat{\beta}).
\end{equation}

On the other hand, let $r(T)$ be the number of reducible edges in a good filtered spanning tree $T$ of $D$. Therefore, by Proposition~\ref{prop:summary}
\begin{equation}\label{upper_bound_alpha_proof}
     \alpha(\widehat{\beta})  \leq c(D)+2-r(T).
\end{equation}

To finish the proof, it suffices using Proposition~\ref{prop:positive_infimum_conjugate_5_families} to obtain the expected lower bound for the arc index, and finding an accurate good spanning tree of $D$ to obtain the expected upper bound in each case. We address each case separately. 
   \begin{enumerate}
   \item[$(\Lambda_1)$] Set the convention $\Delta^p = \sigma_2(\sigma_1 \sigma_2)^{3\left( \frac{p-1}{2}\right)+1}$ if $p$ is odd and $\Delta^p = (\sigma_1 \sigma_2)^{3\left( \frac{p}{2}\right)}$ if $p$ is even. If $p=0$, then $\widehat{\beta}$ is the 3-component trivial link and $\alpha(\widehat{\beta}) = 6$. Let $p>0$. If $p$ is even, as $c(D)=3p$ we have 
        \[ \alpha(\widehat{\beta}) \geq -3p + 3 +(6p+1) -2p = p+4.\]
    Furthermore, in Figure~\ref{fig:spanning_tree_Lambda_1_even} a good spanning tree $T$ of $D$ is shaded, and $r(T)=3 \left( \frac{p}{2} \right) -1$. Then we have
    \[ \alpha(\widehat{\beta}) \leq 3p + 2  - \left( 3 \left( \frac{p}{2} \right) -1 \right) = \frac{1}{2}(3p)+3.\]

    \noindent If $p$ is odd, note that by Proposition~\ref{prop:upper_extreme_Kh_3-braids} $H^{i_0-1, \overline{j}}(\widehat{\beta})=\mathbb{Z}$. Therefore, in this case we can refine (\ref{tb_eq2}) and (\ref{lower_bound_alpha_proof}) as
        \[ \overline{tb}(\widehat{\beta}^*) \leq \mathbf{m} \leq - \overline{j}(\widehat{\beta}) + i_0(\widehat{\beta}) -1 \quad \text{ and } \quad  \alpha(\widehat{\beta}) \geq  - c(D) +3+ \overline{j}(\widehat{\beta}) - i_0(\widehat{\beta})+1, \]
    respectively. Since $c(D)=3p$, we have
    \[ \alpha(\widehat{\beta}) \geq -3p + 3 +6p -2p +1 = p+4. \] 
    In addition, in Figure~\ref{fig:spanning_tree_Lambda_1_odd} a good spanning tree $T$ of $D$ is shaded, and $r(T)=3 \left( \frac{p-1}{2} \right)$. Then we have
   \[ \alpha(\widehat{\beta}) \leq 3p + 2  - 3 \left( \frac{p-1}{2} \right) = \frac{1}{2}(3p+1)+3.\]
    
   \item[$(\Lambda_2)$] Set the convention $\Delta^p = \sigma_2(\sigma_1 \sigma_2)^{3\left( \frac{p-1}{2}\right)+1}$ if $p$ is odd and $\Delta^p = (\sigma_2 \sigma_1)^{3\left( \frac{p}{2}\right)}$ if $p$ is even. Note that $c(D)=3p+k_1$.  If $p=0$, $\widehat{\beta}$ is the disjoint union of an unknot $U$ and the torus link $T_{2,k_1}$. For $k_1=1$, $T_{2,1}$ is also the unknot and the result holds; for $k_1\geq 2$, the (standard) diagram $D'$ of $T_{2,k_1}$ is adequate and by Theorem~\ref{thm:adequate} we have $$\alpha(\widehat{\beta}) = \alpha(U)+\alpha(T_{2,k_1}) = 2 + (|s_AD'|+|s_BD'|) = 2 + (2+k_1)=k_1+4.$$

   If $p>0$ is even, then
   \[  \alpha(\widehat{\beta}) \geq -(3p+k_1) + 3 + (6p+3k_1+1) - (2p+k_1) = p+k_1+4.\]
   Furthermore, in Figure~\ref{fig:spanning_tree_Lambda_2_even} a good spanning tree $T$ of $D$ is shaded, and $r(T)=3 \left( \frac{p}{2} \right)-1$.
    Then we have
   \[ \alpha(\widehat{\beta}) \leq (3p+k_1) + 2 - \left(3\left( \frac{p}{2} \right)-1 \right) = \frac{1}{2}(3p)+k_1+3.\]
   
   If $p>0$ is odd, then
   \[ \alpha(\widehat{\beta}) \geq -(3p+k_1) + 3 + (6p+3k_1) - (2p+k_1) = p+k_1+3.\]
   In addition, in Figure~\ref{fig:spanning_tree_Lambda_2_odd} a good spanning tree $T$ of $D$ is shaded, and $r(T)=3 \left( \frac{p-1}{2} \right)+1$. Then we have
   \[ \alpha(\widehat{\beta}) \leq (3p+k_1) + 2  - \left(3 \left( \frac{p-1}{2} \right)+1 \right) = \frac{1}{2}(3p+1)+k_1+2.\]

    \item[$(\Lambda_3)$] Set the convention $\Delta^{2u} = (\sigma_1 \sigma_2)^{3u}$. If $u=0$, then $\widehat{\beta}$ is the unknot and $\alpha(\widehat{\beta})=2$. Let $u>0$, we have $c(D)=6u+2$, so
    \[\textstyle \alpha(\widehat{\beta}) \geq -(6u+2) + 3 + 12u+5 - (4u+1) = 2u + 5.\]
     In Figure~\ref{fig:spanning_tree_Lambda_3}, a good spanning tree $T$ of $D$ is shaded. For this spanning tree, the number of reducible edges is $r(T) = 3u$. Then
   \[ \textstyle  \alpha(\widehat{\beta}) \leq (6u + 2) + 2 -3u  = 3u + 4.  \]

    \item[$(\Lambda_4)$] The diagram $D$ is adequate if $u=0$, and in such case $|s_AD|=3$, $|s_BD|=\sum_{i=1}^{2t} 
    k_i - 2t +1$. By Theorem~\ref{thm:adequate} we have $\alpha(\widehat{\beta})=|s_AD|+|s_BD| = \sum_{i=1}^{2t} k_i -2t + 4$. 
    
    If $u>0$, consider $\beta'=\sigma_1^{-1}\beta \sigma_1 = \Delta^{2u}\sigma_1^{k_1-1} \sigma_2^{k_2} \cdots \sigma_{2}^{k_{2t}} \sigma_1$. Let $D'$ be the diagram derived from the given representative of $\beta'$, with the convention $\Delta^{2u} = (\sigma_2 \sigma_1)^{3u}$. We have $c(D')=6u + \sum_{i=1}^{2t} k_i$, so
    \[ \textstyle \begin{array}{rcl}
     \alpha(\widehat{\beta}) = \alpha(\widehat{\beta'}) & \geq & -(6u + \sum_{i=1}^{2t} k_i) +3 + (12u + 3\sum_{i=1}^{2t} k_i -2t +1) - (4u + \sum_{i=1}^{2t} k_i) \\
     & = & 2u + \sum_{i=1}^{2t} k_i -2t + 4. 
     \end{array} \]
    
    In Figure~\ref{fig:spanning_tree_Lambda_4}, a good spanning tree $T$ of $D'$ is shaded. For this spanning tree, the number of reducible edges is $r(T) = 3u+2t-1$. Then
         \[ \textstyle  \alpha(\widehat{\beta}) = \alpha(\widehat{\beta'}) \leq (6u + \sum_{i=1}^{2t} k_i) + 2 - (3u+2t-1)  = 3u + \sum_{i=1}^{2t} k_i -2t + 3.  \]    
       
   \item[$(\Lambda_5)$] Set the convention $\Delta^{2u+1} = \sigma_2 (\sigma_1 \sigma_2)^{3u+1}$. We have $c(D) = 6u + 3 + \sum_{i=1}^{2t+1} k_i$, so
   \[  \textstyle  \begin{array}{rcl}
   \alpha(\widehat{\beta}) & \geq & - (6u + 3 + \sum_{i=1}^{2t+1} k_i) + 3 + (12u + 3\sum_{i=1}^{2t+1} k_i - 2t +6) - (4u+2+\sum_{i=1}^{2t+1} k_i) \\
    & = & 2u + \sum_{i=1}^{2t+1} k_i - 2t + 4.
   \end{array} \]
   In Figure~\ref{fig:spanning_tree_Lambda_5}, a good spanning tree $T$ of $D$ is shaded. For this spanning tree, the number of reducible edges is $r(T) = 3u+2t+1$. Then
   \[ \textstyle  \alpha(\widehat{\beta}) \leq (6u + 3 + \sum_{i=1}^{2t+1} k_i) + 2 - (3u+2t+1)  = 3u + \sum_{i=1}^{2t+1} k_i -2t + 4. \qedhere  \]   
   \end{enumerate}

   \end{proof}

    \begin{figure}[h!]
    \begin{subfigure}[b]{.45\textwidth}
        \centering
        \includegraphics[width=0.6\linewidth]{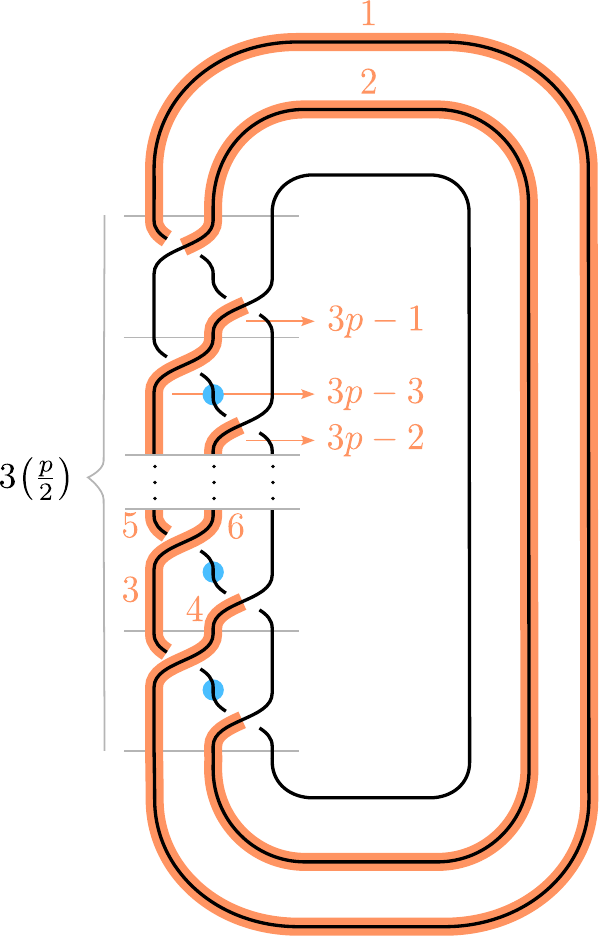}
        \caption{$p$ even }
        \label{fig:spanning_tree_Lambda_1_even}
    \end{subfigure}
    \begin{subfigure}[b]{.45\textwidth}
        \centering
        \includegraphics[width=0.7\linewidth]{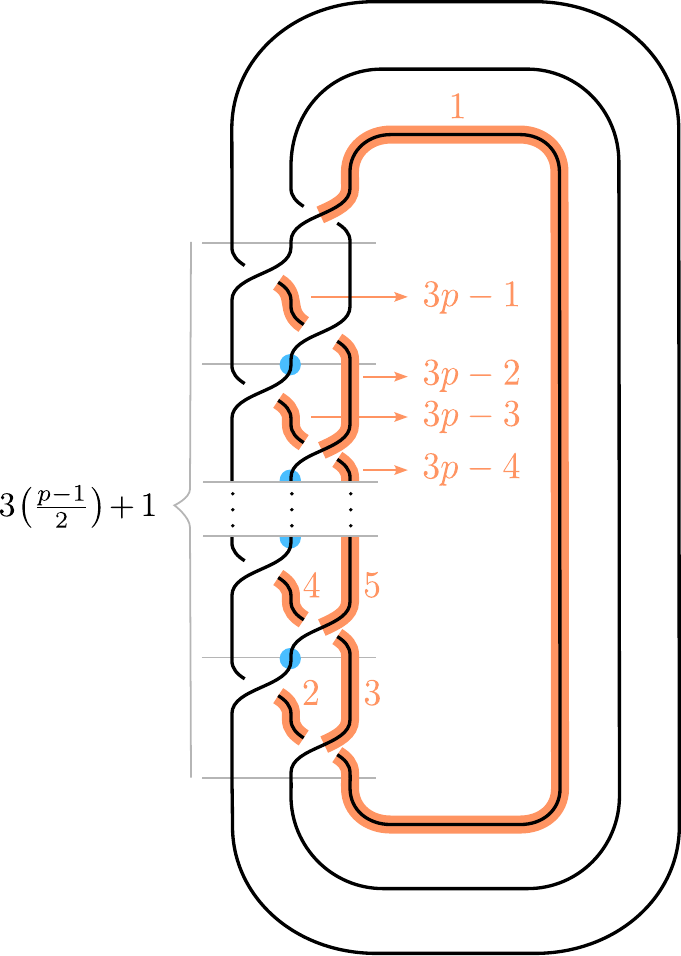}
        \caption{$p$ odd }
        \label{fig:spanning_tree_Lambda_1_odd}
    \end{subfigure}
    \caption{Generic diagram of $\widehat{\beta}$ and a good spanning tree shaded on it, with $\beta = \Delta^p \in \Lambda_1$. Edges marked with a small disk are reducible edges.}
    \label{fig:spanning_trees_1.1}
    \end{figure}

    \begin{figure}[h!]
    \centering
    \begin{subfigure}[b]{.45\textwidth}
        \centering
        \includegraphics[width=0.62\linewidth]{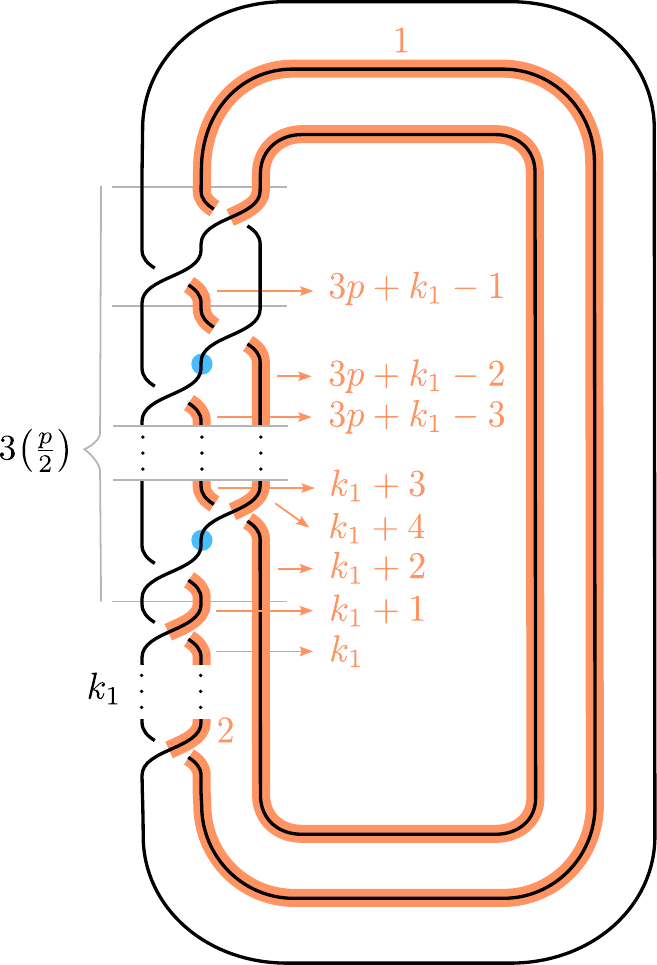}
        \caption{$p$ even}
        \label{fig:spanning_tree_Lambda_2_even}
    \end{subfigure}
     \begin{subfigure}[b]{.45\textwidth}
        \centering
        \includegraphics[width=0.68\linewidth]{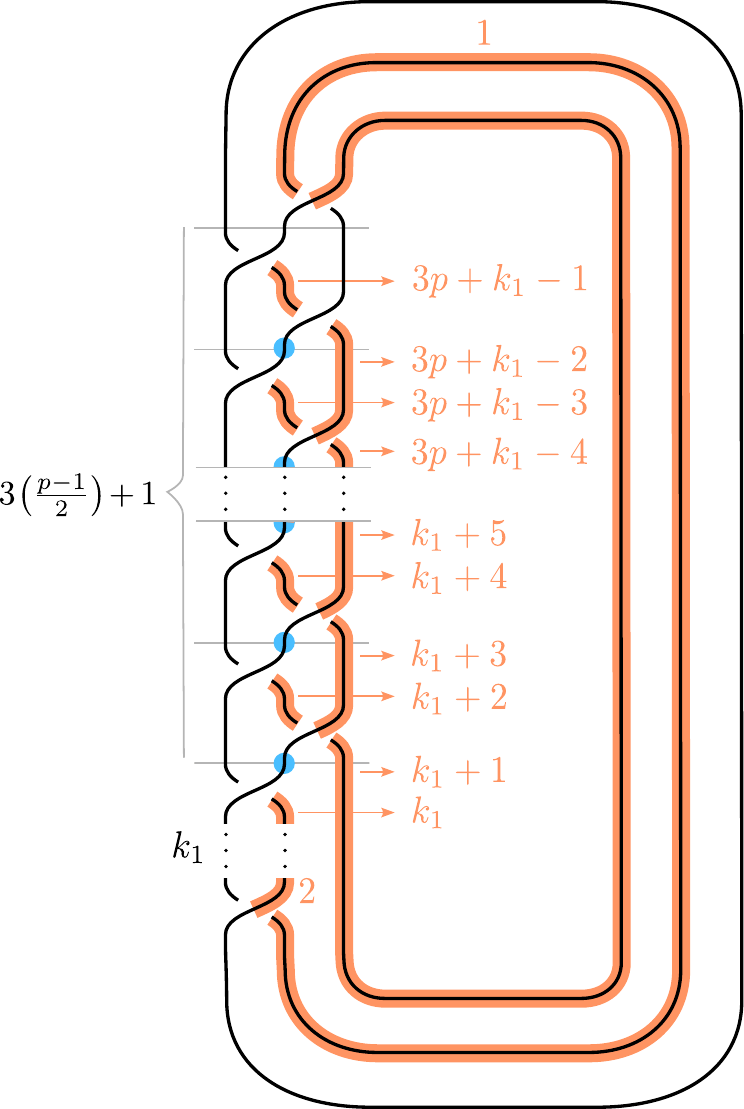}
        \caption{$p$ odd}
        \label{fig:spanning_tree_Lambda_2_odd}
    \end{subfigure}
    \caption{Generic diagram of $\widehat{\beta}$ and a good spanning tree shaded on it, with $\beta = \Delta^{p} \sigma_1^{k_1}\in \Lambda_2$. Edges marked with a small disk are reducible edges.}
    \label{fig:spanning_trees_1.2}
    \end{figure}
    
    \begin{figure}[h!]
        \centering
        \includegraphics[width=0.25\linewidth]{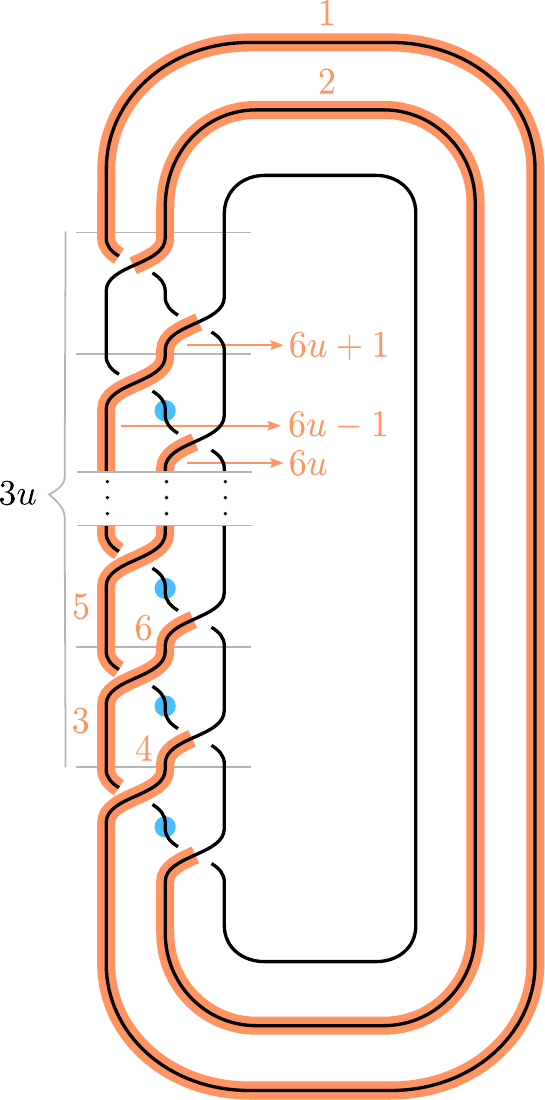}
    \caption{Generic diagram of $\widehat{\beta}$ and a good spanning tree shaded on it, with $\beta = \Delta^{2u}\sigma_1\sigma_2 \in \Lambda_3$. Edges marked with a small disk are reducible edges.}
    \label{fig:spanning_tree_Lambda_3}
    \end{figure}
    
    \begin{figure}[h!]
    \begin{subfigure}[b]{.48\textwidth}
        \centering
        \includegraphics[width=0.7\linewidth]{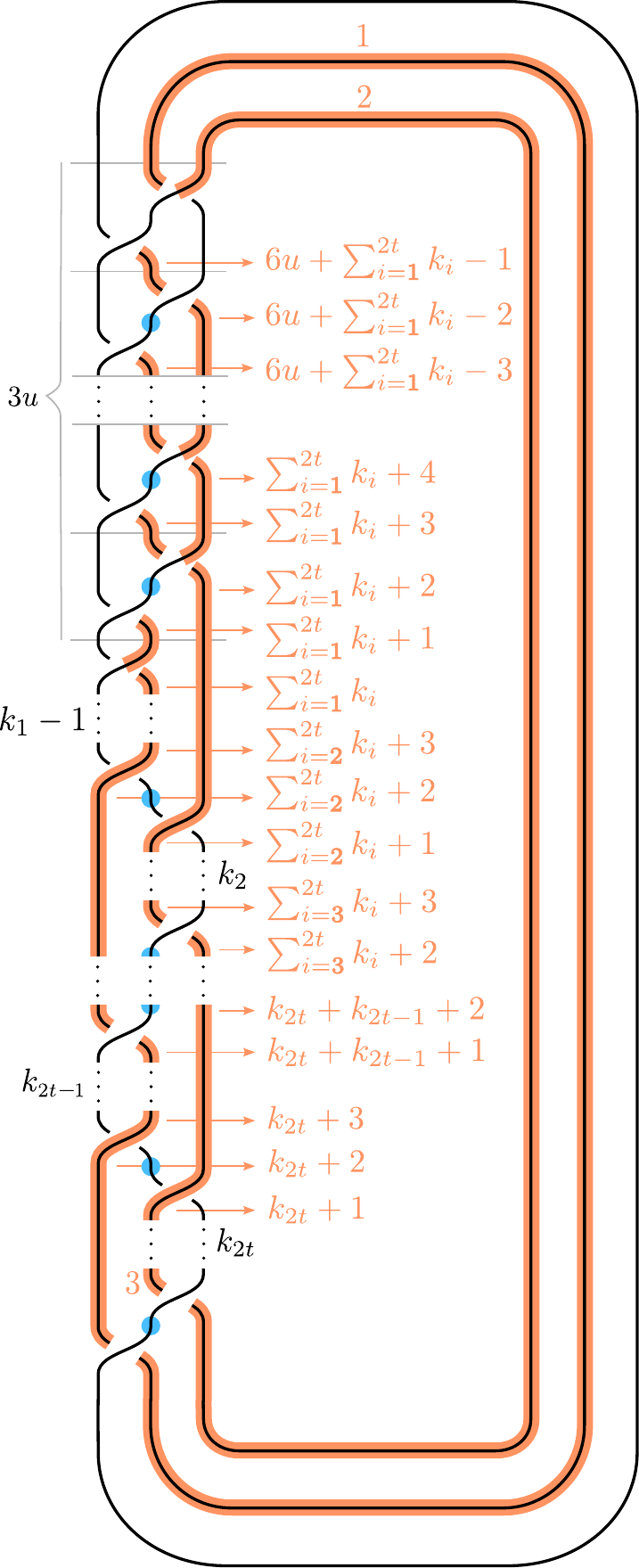}
        \caption{}
        \label{fig:spanning_tree_Lambda_4}
    \end{subfigure}
    \begin{subfigure}[b]{.48\textwidth}
        \centering
        \includegraphics[width=0.9\linewidth]{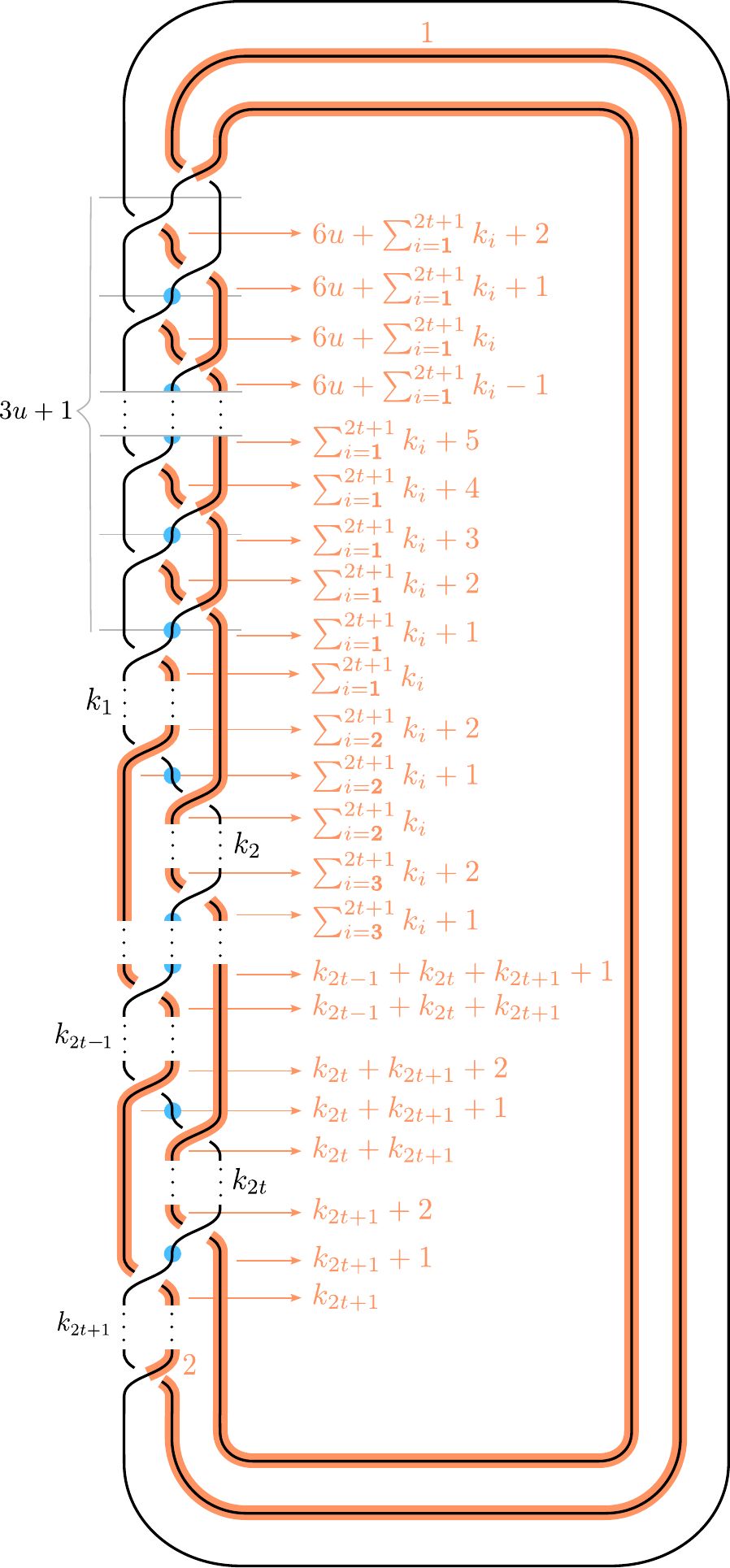}
        \caption{}
        \label{fig:spanning_tree_Lambda_5}
    \end{subfigure}
    \caption{Generic diagrams of $\widehat{\beta'}$ (a) and $\widehat{\beta}$ (b), and a good spanning tree shaded on them, with $\beta' = \Delta^{2u}\sigma_1^{k_1-1}\sigma_2^{k_2}\cdots \sigma_2^{k_{2t}} \sigma_1$ and $\beta = \Delta^{2u+1}\sigma_1^{k_1}\sigma_2^{k_2}\cdots \sigma_1^{k_{2t+1}} \in \Lambda_5$. Edges marked with a small disk are reducible edges.}
    \label{fig:spanning_trees_2.2}
\end{figure}

\begin{remark}
If $\gamma \in \Lambda_i$,  $i \in \{1, 2,3,4,5\}$, its infimum is explicitly described by the representatives given in Proposition~\ref{prop:positive_infimum_conjugate_5_families}. In addition, it turns out that $\inf_s(\gamma)=\inf(\gamma)$ (see, e.g., \cite[Rem. 2.3]{dVV-Gonzalez-Meneses_Silvero_2025}). Consequently, Theorem~\ref{th:arc_index_closed_positive_3-braids} explicitly calculates the arc index for the closure of any 3-braid with a summit infimum of $0$, $1$, or $2$, as the stated lower and upper bounds coincide.
\end{remark}

\subsection{Proof of Conjecture~\ref{conj:c+2-a_gT} for closed positive 3-braids.}

We are interested in links that are prime and non-split. Then, to discard closed 3-braids that do not satisfy these two conditions, we have the following two results. The conclusion is that every $\widehat{\beta}$, with $\beta \in \mathbb{B}_3$, is prime and non-split unless $\beta$ is conjugate to $\sigma_1^a\sigma_2^b$ for some $a,b \in \mathbb{Z} \setminus \{-1,1\}$, or to $\sigma_1^{a}$ for some $a \in \mathbb{Z}$.

\begin{theorem}[{\cite[Sec. 2]{Morton_1979}}]\label{th:prime_closed_3-braids} 
Let $\beta$ be a 3-braid. Then, $\widehat{\beta}$ is not a prime link if and only if $\beta$ is conjugate to $\sigma_1^a \sigma_2^b$ for some $a,b \in \mathbb{Z}\setminus \{-1,1\}$.
\end{theorem}

\begin{theorem}[{\cite[Th. 5.1]{Murasugi_1974}}]\label{th:split_closed_3-braids}
    Let $\beta$ be a $3$-braid. Then, $\widehat{\beta}$ is a split link if and only if $\beta$ is conjugate to $\sigma_1^a$ for some $a \in \mathbb{Z}$.
\end{theorem}

\newpage

We are now ready to prove Conjecture~\ref{conj:c+2-a_gT} for closed positive 3-braids.

\begin{theorem}\label{conj:positive_3-braids}
    Let $\beta$ be a positive 3-braid such that $\widehat{\beta}$ is prime and non-split. Then, $$c(\widehat{\beta}) + 2 - \alpha(\widehat{\beta}) \geq 2g_T(\widehat{\beta}),$$
    and thus Conjecture~\ref{conj:c+2-a_gT} is true for closed positive 3-braids.
\end{theorem}
\begin{proof}
According to Theorem~\ref{th:Murasugi_classification}, every 3-braid is conjugate to another one falling into exactly one of the described families. We will address each of the three cases separately.
    \begin{enumerate}[label=(\roman*), leftmargin=*, labelindent=0pt]
        \item As stated in Lemma~\ref{lemma:braid_classifications}, given $$\beta=\Delta^{2v} \sigma_1^{a_1} \sigma_2^{-b_1} \sigma_1^{a_2} \sigma_2^{-b_2} \cdots \sigma_1^{a_r} \sigma_2^{-b_r},$$ with $v \in \mathbb{Z}$ and $a_i,b_i,r > 0$, it is positive if and only if $2v-\mathbf{b} \geq 0$, where $\mathbf{b}=\sum_{i=1}^r b_i$. Let us define $\mathbf{a}=\sum_{i=1}^r a_i$.
        
        The braid $\beta$ is conjugate to 
       \[ \gamma = \Delta^{2v-\mathbf{b}} \left( \sigma_{[\mathbf{b}]}^{a_1+2} \sigma_{[\mathbf{b}-1]}^2 \cdots \sigma_{[\mathbf{b}-b_1+1]}^2 \right) \left( \sigma_{[\mathbf{b}-b_1]}^{a_2+2} \sigma_{[\mathbf{b}-b_1-1]}^2 \cdots \sigma_{[\mathbf{b}-b_1-b_2+1]}^2 \right) \cdots \left( \sigma_{[b_r]}^{a_r+2} \sigma_{[b_r-1]}^2 \cdots \sigma_{[1]}^2  \right), \] which belongs either to $\Lambda_4$ if $\mathbf{b}$ is even or to $\Lambda_5$ if $\mathbf{b}$ is odd. In any case, note that
       $$c(\widehat{\gamma})=c(D)=3(2v-\mathbf{b}) + \mathbf{a} + 2 \mathbf{b} = 6v+\mathbf{a}-\mathbf{b}.$$
       
       \begin{itemize}

         \item Let $\mathbf{b}$ be even. Then $\gamma \in \Lambda_4$, with $2u=2v-\mathbf{b}$, $2t=\mathbf{b}$ and $\sum k_i = \mathbf{a}+2\mathbf{b}$. Let $D$ be the diagram derived from the given representative of $\gamma$. Notice that $D$ is an adequate diagram if $u=0$, and the result holds for this case by Theorem~\ref{thm:adequate}. 
         
         Let $u \geq 1$. We have 
         \[ \begin{array}{rcl}
         c(\widehat{\gamma}) + 2 - \alpha(\widehat{\gamma}) & \geq & (6v+\mathbf{a}-\mathbf{b}) +   2 - \left[ 3 \left(v-\frac{\mathbf{b}}{2} \right) + (\mathbf{a}+2\mathbf{b})-\mathbf{b} + 3 \right]\\ & = & 2v + \frac{1}{2}(2v-\mathbf{b}-2) \geq 2v \geq 2g_T(\widehat{\beta}), 
         \end{array} \]
        where the last inequality is a consequence of Theorem~\ref{th:turaev_genus_closed_3-braids}.

        \item Let $\mathbf{b}$ be odd. Then $\gamma \in \Lambda_5$, with $2u+1=2v-\mathbf{b}$, $2t+1=\mathbf{b}$ and $\sum k_i = \mathbf{a}+2\mathbf{b}$. We have
         \[ \begin{array}{rcl}
          c(\widehat{\gamma}) + 2 - \alpha(\widehat{\gamma}) & \geq &  (6v+\mathbf{a}-\mathbf{b}) +   2 - \left[ 3 \left(v-\frac{\mathbf{b}}{2} \right) + (\mathbf{a}+2\mathbf{b})-\mathbf{b} + 3 \right] \\ 
          & = & 2v + \frac{1}{2}(2v-\mathbf{b}-3). 
          \end{array} \]
         
        If $2v-\mathbf{b} \geq 3$, we have $c(\widehat{\gamma}) + 2 - \alpha(\widehat{\gamma}) \geq 2v \geq 2g_T(\widehat{\beta})$, by Theorem~\ref{th:turaev_genus_closed_3-braids}. Otherwise, $2v-\mathbf{b} = 1$, and it is not difficult to see that $g_T(D)=t=v-1$, and therefore $g_T(\widehat{\beta})=v-1$ again by Theorem~\ref{th:turaev_genus_closed_3-braids}. In this case we have $c(\widehat{\gamma}) + 2 - \alpha(\widehat{\gamma}) \geq 2v-1 > 2(v-1)= 2g_T(\widehat{\beta})$. 
    
       \end{itemize}

\item Given a positive braid $\beta=\Delta^{2v}\sigma_1^a$, we distinguish three cases:
        \begin{itemize}
            \item If $a=0$, then $v\geq 0$ and $\beta \in \Lambda_1$, with $p=2v$. The special case $v=0$ provides a disjoint union of three unknots, so we can assume $v \geq 1$. Note that $c(\widehat{\beta}) = c(D) = 3p$ by Theorems~\ref{th:crossing_number_braid_closures} and \ref{th:braid_index_closed_3-braids}. Then, by Theorem~\ref{th:arc_index_closed_positive_3-braids} we have
            \[ c(\widehat{\beta})+2-\alpha(\widehat{\beta}) \geq 3p+2- \left(\frac{3}{2}p+3\right) = p + \frac{1}{2}(p-2) \geq p = 2v \geq 2g_T(\widehat{\beta}),
            \]
             where the last inequality is a consequence of Theorem~\ref{th:turaev_genus_closed_3-braids}.

             \item If $a>0$, then $v\geq 0$ and $\beta \in \Lambda_2$, with $p=2v$ and $k_1=a$. The special case $p=0$ provides a disjoint union of a torus link $T(2,a)$ and an unknot, so we can assume $p \geq 2$. Note that $c(\widehat{\beta})=c(D)=3p+k_1$ by Theorems~\ref{th:crossing_number_braid_closures} and \ref{th:braid_index_closed_3-braids}. Then, by Theorem~\ref{th:arc_index_closed_positive_3-braids} we have
             \[ \begin{array}{rcl} c(\widehat{\beta})+2-\alpha(\widehat{\beta}) & \geq & (3p+k_1)+2- \left(\frac{1}{2}(3p+1)+k_1+2\right) \\
             & = & p + \frac{1}{2}(p-2) \geq p = 2v \geq 2g_T(\widehat{\beta}),
             \end{array}
            \]
            where the last inequality is a consequence of Theorem~\ref{th:turaev_genus_closed_3-braids}.

            \item If $a<0$, then $2v-|a|\geq 0$ and $\beta$ is conjugate to $\gamma = \Delta^{2v-|a|} \sigma_{[|a|]}^2 \sigma_{[|a|-1]}^2 \cdots \sigma_{[1]}^2$ and to $\gamma'=\Delta^{2v-|a|} \sigma_{[|a|-1]}^2 \sigma_{[|a|-2]}^2 \cdots \sigma_{[0]}^2$ (note that $\gamma'=\Delta\gamma\Delta^{-1}$). Let $D$ and $D'$ be the diagrams derived from the given representatives of $\gamma$ and $\gamma'$, respectively. We distinguish two cases depending on the parity of $a$.


            If $a$ is even, then $\gamma \in \Lambda_{4}$, with $2u=2v-|a|$, $2t=|a|$ and $\sum k_i = 4t$. The special case $u=0$ provides an adequate diagram, and hence the result holds by Theorem~\ref{thm:adequate}. Let us assume that $u > 0$. Note that $c(\widehat{\gamma})=c(D)=6u+4t$ by Theorems~\ref{th:crossing_number_braid_closures} and \ref{th:braid_index_closed_3-braids}. Then, by Theorem~\ref{th:arc_index_closed_positive_3-braids} we have
            \[ \begin{array}{rcl} c(\widehat{\beta})+2-\alpha(\widehat{\beta}) & \geq &  (6u+4t)+2- \left(3u + \sum k_i - 2t + 3\right) \\
            & = & 2u +(u-1)+2t \geq 2u+2t = 2v \geq 2g_T(\widehat{\beta}),
            \end{array}
            \]
             where the last inequality is a consequence of Theorem~\ref{th:turaev_genus_closed_3-braids}.


            If $a$ is odd, then $\gamma' \in \Lambda_{5}$, with $2u+1=2v-|a|$, $2t+1=|a|$ and $\sum k_i = 4t+2$. Note that $c(\widehat{\gamma'})=c(D')=6u+4t+5$ by Theorems~\ref{th:crossing_number_braid_closures} and \ref{th:braid_index_closed_3-braids}. 
             If $u \geq 1$, by Theorem~\ref{th:arc_index_closed_positive_3-braids} we have
            \[ \begin{array}{rcl}
                 c(\widehat{\beta})+2-\alpha(\widehat{\beta}) & \geq  &(6u+4t+5)+2- \left(3u + \sum k_i - 2t + 4\right)  \\
                 & = & (2u+1) +(u-1)+(2t+1)  \\ & \geq & (2u+1)+(2t+1) = 2v \geq 2g_T(\widehat{\beta}),
            \end{array} 
            \]
             where the last inequality is a consequence of Theorem~\ref{th:turaev_genus_closed_3-braids}. Otherwise, $u=0$, and it is not difficult to see that $g_T(D)=t=v-1$, and therefore $g_T(\widehat{\beta})=v-1$ again by Theorem~\ref{th:turaev_genus_closed_3-braids}. In this case we have $c(\widehat{\gamma}) + 2 - \alpha(\widehat{\gamma}) = 2v-1 > 2(v-1)= 2g_T(\widehat{\beta})$.
        \end{itemize}
        
         \item Given a positive braid $\beta = \Delta^{2v} \sigma_1^a \sigma_2^{-1}$, with $a \in \{ -1,-2,-3 \}$, we distinguish three cases depending on the value of $a$:
        \begin{itemize}
            \item If $a=-1$, then $2v-1 \geq 1$ and $\beta$ is conjugate to $\gamma = \Delta^{2v-1} \sigma_1$. Note that $\gamma \in \Lambda_2$, with $p=2v-1$ and $k_1=1$. If $p=1$ (i. e. $v=1$), then $\widehat{\beta}$ is a trefoil knot, which is alternating, and the result holds. Let $p \geq 3$ (i. e. $v\geq 2$) and let $D$ be the diagram derived from the given representatives of $\gamma$. Note that $c(\widehat{\gamma})=c(D)=3p+1$ by Theorems~\ref{th:crossing_number_braid_closures} and \ref{th:braid_index_closed_3-braids}. As in previous cases, by Theorem~\ref{th:arc_index_closed_positive_3-braids} we have
            \[ \begin{array}{rcl}
                 c(\widehat{\beta})+2-\alpha(\widehat{\beta}) & \geq  &(3p+1)+2- \left(\frac{1}{2}(3p+1)+3\right)  \\
                 & = & 2v + (v-2) \geq 2v > 2(v-1) = 2g_T(\widehat{\beta}).
            \end{array} 
            \]
            
            \item If $a=-2$, then $2v-2 \geq 0$ and $\beta$ is conjugate to $\gamma = \Delta^{2v-1}$. Note that $\gamma \in \Lambda_1$ with $p=2v-1$. If $p=1$ (i. e. $v=1$), then $\widehat{\beta}$ is a Hopf link, which is alternating, and the result holds. Let $p \geq 3$ (i. e. $v\geq 2$) and let $D$ be the diagram derived from the given representatives of $\gamma$. Note that $c(\widehat{\gamma})=c(D)=3p$ by Theorems~\ref{th:crossing_number_braid_closures} and \ref{th:braid_index_closed_3-braids}. As in previous cases, by Theorem~\ref{th:arc_index_closed_positive_3-braids} we have
             \[ \begin{array}{rcl}
                 c(\widehat{\beta})+2-\alpha(\widehat{\beta}) & \geq  &3p+2- \left(\frac{1}{2}(3p+1)+3\right)  \\
                 & = & 3(v-1) = 3g_T(\widehat{\beta}) > 2g_T(\widehat{\beta}).
            \end{array} 
            \]
            
            \item If $a=-3$, then $2v-3 \geq 1$ (and hence $v \geq 2$) and $\beta$ is conjugate to $\gamma = \Delta^{2v-2} \sigma_1 \sigma_2$. Note that $\gamma \in \Lambda_3$, with $2u=2v-2$. Let $D$ be the diagram derived from the given representatives of $\gamma$. Since $c(\widehat{\gamma})=c(D)=6u+2$,  by Theorem~\ref{th:arc_index_closed_positive_3-braids} we have
            \[ \begin{array}{rcl}
                 c(\widehat{\beta})+2-\alpha(\widehat{\beta}) & \geq  &(6u+2)+2-(3u+4)  \\
                 & = &6(v-1) >  2(v-1) = 2g_T(\widehat{\beta}).
            \end{array} 
            \qedhere
            \]
        \end{itemize}
        
    \end{enumerate}    
        
\end{proof} 

\section{Torus links}\label{sec:torus_links}

In this section, we confirm Conjecture~\ref{conj:c+2-a_gT} for torus links. The crossing number of a torus link $T_{p,q}$, where $p, q \in \mathbb{Z},$ is 
\begin{equation}
    c(T_{p,q})=\min \{ |p|(|q|-1), |q|(|p|-1) \}.
\end{equation}
Matsuda \cite[Th. 1.1]{Matsuda_2006} computed the arc index of a torus knot. Dalton, Etnyre, and Traynor \cite[Th. 1.1]{DET_2024} computed the maximum Thurston-Bennequin number of a torus link with more than one component. This computation along with Theorem~\ref{thm:tb} yields a formula for the arc index of a torus link. Both computations lead to the following formula for the arc index of a torus knot or link. If $T_{p,q}$ is a torus knot or link, then 
\begin{equation}
    \alpha(T_{p,q})=|p|+|q|.
\end{equation}

    When establishing Conjecture~\ref{conj:c+2-a_gT} for torus links, we reduce our study to those torus links $T_{p,q}$ with $2 \leq p \leq q$ for the following reasons.
    \begin{enumerate}
        \item If $p$ or $q$ is equal to $\pm 1$, then $T_{p,q}$ is the unknot.
        \item Given $p,q \in \mathbb{Z}$, $T_{p,q}$, $T_{q,p}$ and $T_{-p,-q}$ are equivalent.
        \item Given $p,q \in \mathbb{Z}$, $T_{p,-q}$ and $T_{-p,q}$ are equivalent and both represent the mirror image of $T_{p,q}$.
    \end{enumerate}

\begin{theorem}
   Let $T_{p,q}$ be the $(p,q)$-torus link. Then
    $$ c(T_{p,q})+2-\alpha(T_{p,q}) \geq 2 g_T(T_{p,q}) $$
    for $2 \leq p \leq q$, and thus Conjecture~\ref{conj:c+2-a_gT} is true for torus links.
\end{theorem}
\begin{proof}
If $p=2$, then $T_{2,q}$ is alternating. Hence $\alpha(T_{2,q}) = c(T_{2,q})+2$, $g_T(T_{2,q}) = 0$, and the result holds. If $p=3$, then $T_{3,q}$ is the closure of a positive $3$-braid, and the result holds by Theorem~\ref{conj:positive_3-braids}.

 Let $4\leq p \leq q$, and let $D$ be the standard diagram of $T_{p,q}$ obtained as the closure of the braid word $(\sigma_1 \cdots \sigma_{p-1})^q$. The diagram $D$ realizes the crossing number of $T_{p,q}$, and there are $p$ components in the state $s_AD$. Therefore
\begin{equation}
   2g_T(T_{p,q})  \leq  2g_T(D) = c(D) + 2 - |s_AD| - |s_BD| = pq  - q + 2 - p - |s_BD|,
\end{equation}
and
\begin{equation}
    c(T_{p,q}) + 2 - \alpha(T_{p,q}) = pq - 2q - p +2.
\end{equation}
In order to prove the desired inequality, it suffices to find a diagram $D'$ of $T_{p,q}$ (maybe $D$ itself) such that $c(D')=c(D)$, $|s_AD'|=|s_AD|$ and $|s_BD'| \geq q$. This is exactly what we will do. We will address some cases separately.

\textbf{Case 1} ($p$ even). By inspection of $D$, it is straightforward to check that $|s_BD|=q$ in this case. It suffices to take $D'=D$.

\textbf{Case 2} ($p$ odd and $q$ odd). Note that in this case $|s_BD|=1$. Our goal will be to transform $D$ into another diagram of the same link having a sufficiently large number of simple $n$-twists between some two strands, since such a $n$-twist provides $n-1$ components in the corresponding all-$B$ state.

Consider the transformations of $D$ shown in Figure~\ref{fig:torus_link_p_odd_q_odd}. We denote by $D_{x}$ the diagram obtained by closing (in the natural way) the braid in Figure~\ref{fig:torus_link_p_odd_q_odd}($x$), with $x\in \{ 1, 2, 3, 4\}$. Then, it is obvious that $D_{1}=D$.

On each $D_{x}$, we define a \textit{diagonal} as a subset of crossings of $D_{x}$ associated to a sequence of letters $\sigma_u\sigma_{u+1} \cdots \sigma_v$, with $1\leq u<v \leq p-1$, in the corresponding braid. A diagonal is said to be \textit{maximal} if it is not properly contained in any longer such sequence. It is easy to check that if $x\in \{1,2,3\}$, then $D_x$ has $q$ maximal diagonals: $d_1^{x}, \dots, d_q^{x} $. 

In $D_1$, between each pair of diagonals $d_i^1$ and $d_{i+1}^1$, $i=1,3,\dots, q-2$, there is a shaded area that can be compressed into a simple double twist between the first two strands. The result of doing this is $D_2$. Passing from $D_1$ to $D_2$ increases the number of components in the all-$B$ state by $\frac{q-1}{2}$. 
Consequently, \[ |s_BD_2|=|s_BD_1|+ \frac{q-1}{2} = 1 + \frac{q-1}{2}. \] 

To obtain $D_3$ from $D_2$, we shift the uppermost $\frac{p-1}{2}$ simple double twists between the first two strands toward the beginning of the braid. This preserves the number of components in the all-$B$ state. We also move the other $\frac{q-p}{2}$ simple double twists up to place one below each diagonal $d_2^2, d_4^2, \dots, d_{q-p}^2$; when doing this, each of these $\frac{p-q}{2}$ simple double twists becomes a simple $3$-twist. This increases the number of components in the all-$B$ state by $\frac{q-p}{2}$. Therefore,
\[|s_BD_3| = |s_BD_2| + \frac{q-p}{2} = 1 + \frac{2q-p-1}{2}.\]

In $D_3$ (similarly to $D_1$), between each pair of diagonals $d_i^3$ and $d_{i+1}^3$, $i=q-p+2, \dots, q-1$, there is a shaded area that can be compressed into a simple double twist between two strands. The result of doing this is $D_4$. Passing from $D_3$ to $D_4$ increases the number of components in the all-$B$ state by $\frac{p-1}{2}$. As a result,
\[|s_BD_4| = |s_BD_3| + \frac{p-1}{2} = q.\]

Observe that $c(D_x)=c(D)$ and $|s_AD_x| = |s_AD|$ for every $x \in \{1,2,3,4\}$. Then, taking $D'=D_4$, we are done.

\textbf{Case 3} ($p$ odd and $q$ even) In this case, as in the previous one, $|s_BD|=1$. We could proceed in a similar manner. A few minor adjustments of Figure~\ref{fig:torus_link_p_odd_q_odd} would provide four diagrams of $T_{p,q}$: $D_x$ with $x\in \{ 1,2,3,4\}$. We would have again $D_1=D$ and, for every $x \in \{ 1,2,3,4\}$,  $c(D_x)=c(D)$ and $|s_AD_x|=|s_AD|$. In addition,
\[
\begin{array}{rcccl}
 |s_BD_2| & = &   \displaystyle  |s_BD_1|+\frac{q}{2} & = &   \displaystyle  1 + \frac{q}{2}, \\ \\
 |s_BD_3| & = &   \displaystyle  |s_BD_2| + \frac{q-p+1}{2} & = &    \displaystyle  1 + \frac{2q-p+1}{2}, \\ \\
 |s_BD_4| & = &   \displaystyle |s_BD_3| + \frac{p-3}{2} & = &    \displaystyle q.
\end{array}
\]
Therefore, the result holds by taking $D'=D_4$. \qedhere

\begin{figure}[!h]
    \centering
    \includegraphics[width=0.96\linewidth]{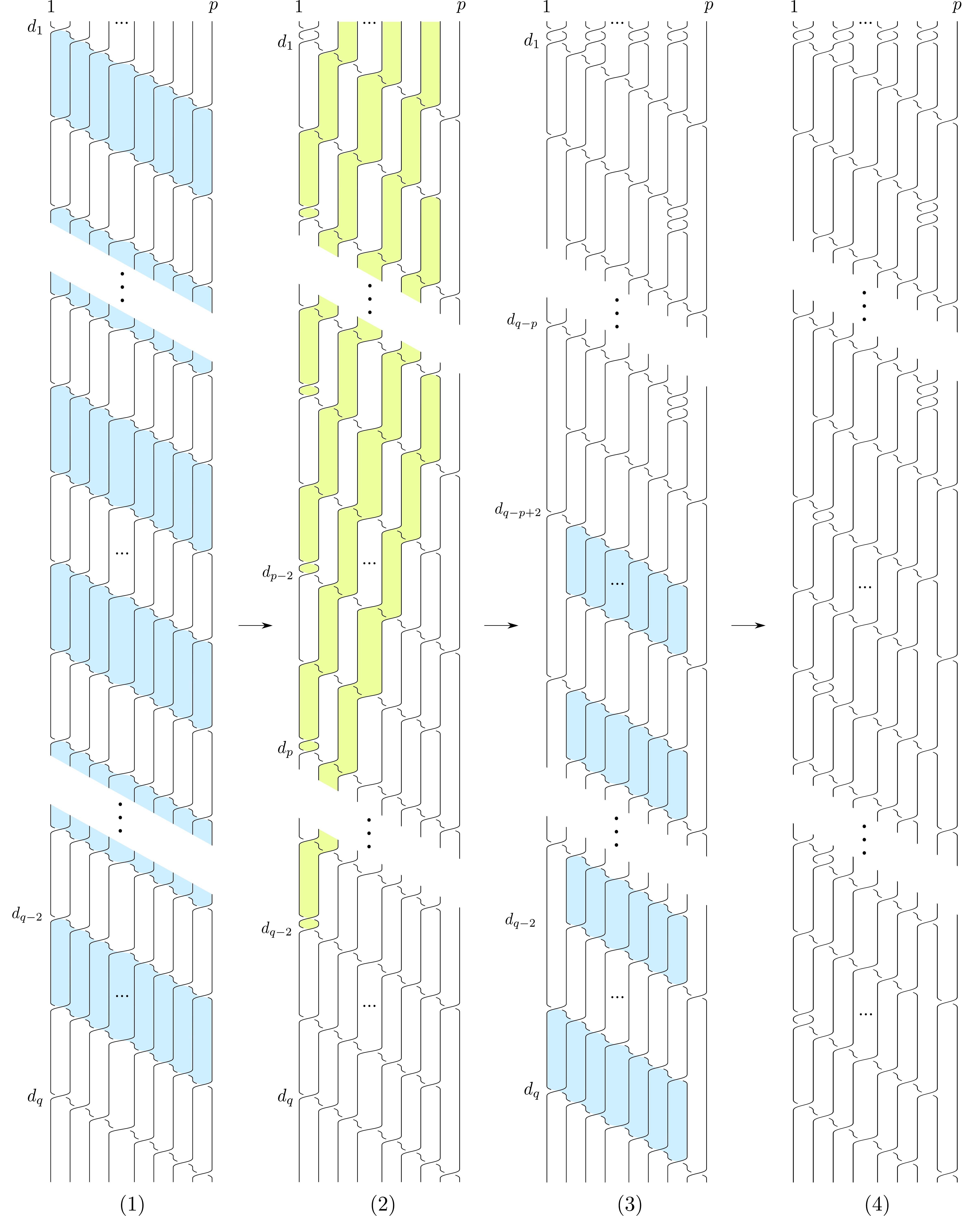}
    \caption{$p$ and $q$ odd}
    \label{fig:torus_link_p_odd_q_odd}
\end{figure}
\end{proof}

\section{Kanenobu knots}
\label{sec:Kanenobu}

Kanenobu \cite{Kanenobu1, Kanenobu2} constructed a family of knots $K(p,q)$, now known as \textit{Kanenobu knots}, such that the family $\{K(p,q)\}_{p,q\in\mathbb{Z}}$ contains infinitely many knots with the same HOMFLY-PT polynomial but distinct Alexander module structure.  The Kanenobu knot $K(p,q)$ for integers $p$ and $q$ is the knot with diagram $D(p,q)$ as in Figure~\ref{fig:Kanenobu}. Kanenobu showed that $K(p,q)=K(q,p)$ and that $K(p,q)^* = K(-p,-q)$.

\begin{figure}[h]
\input{kanenobu.tex}
\caption{The Kanenobu knot $K(p,q)$. The rectangle labeled $n$ contains a twist region with $n$ crossings, as indicated.}
\label{fig:Kanenobu}
\end{figure}
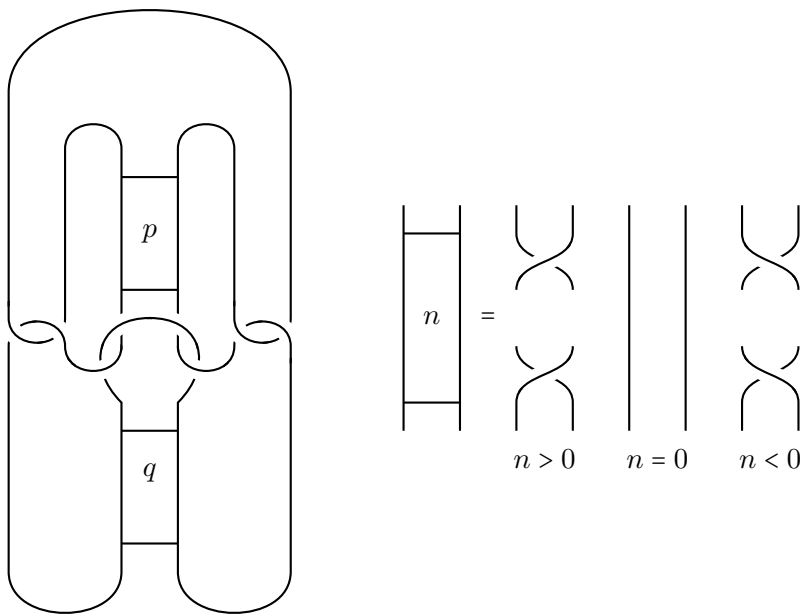

Qazaqzeh and Mansour \cite{QM_2016} computed the crossing number of many Kanenobu knots, and Lee and Takioka \cite{LT_2017} computed the arc index for many Kanenobu knots. We combine their results to confirm Conjecture~\ref{conj:c+2-a_gT} for most Kanenobu knots.

\begin{theorem}
\label{thm:Kanenobu}
Let $p,q\in\mathbb{Z}$ with $|p|\leq |q|$. Conjecture~\ref{conj:c+2-a_gT} holds for the Kanenobu knot $K(p,q)$, that is,
\[c(K(p,q)) + 2 -\alpha(K(p,q)) \geq 2g_T(K(p,q)),\]
except possibly in the case where $pq<0$, $|p|=2$, $|q|\geq 7$, and $q$ is odd.
\end{theorem}
\begin{proof}
Let $D(p,q)$ be the diagram of $K(p,q)$ in Figure~\ref{fig:Kanenobu}. The genus $g_T(D(p,q))$ of the Turaev surface of $D(p,q)$ is zero when $p=q=0$, one when exactly one of $p$ or $q$ is zero, and two when $p$ and $q$ are both nonzero. Since Conjecture~\ref{conj:c+2-a_gT} holds for all links with Turaev genus zero or one, we  assume that $p$ and $q$ are both nonzero.

Suppose that $pq>0$. Since $K(p,q)^* = K(-p,-q)$, we assume that $p>0$ and $q>0$. Lee and Takioka \cite{LT_2017} identified $K(1,1)$ as $10_{162}$ and $K(1,2)=K(2,1)$ as $11n_{132}$. Since both these knots have Turaev genus one, the conjecture holds. Now suppose that $pq\geq 3$. Qazaqzeh and Mansour \cite{QM_2016} proved that $c(K(p,q)) = p+q+8$, and Lee and Takioka \cite{LT_2017} proved that $\alpha(K(p,q))=p+q+6$. Therefore,
\[ c(K(p,q))+2 - \alpha(K(p,q)) = (p+q+8) + 2 - (p+q+6) = 4\geq 2g_T(K(p,q)).\]

Suppose that $pq<0$ and $|p|=1$. For concreteness, let $p=-1$ and $q>0$. The last diagram in Figure 7 of \cite{QM_2016} is a diagram $D$ of $K(-1,q)$ such that $g_T(D)=1$. Hence if $pq<0$ and $|p|=1$, then $g_T(K(p,q))\leq 1$, and the conjecture holds.

Suppose that $pq<0$ and $|p|=2$. For concreteness, let $p=-2$ and $q>0$. If $2\leq q \leq 6$, Lee and Takioka \cite{LT_2017} showed that $c(K(p,q))= |p|+|q|+8 = q+10$ and that $\alpha(K(p,q)) \leq q+8$. Therefore,
\[c(K(p,q))+2 - \alpha(K(p,q)) \geq (q+10) + 2 - (q+8) = 4\geq 2g_T(K(p,q)).\]
If $q>6$ and even, then Qazaqzeh and Mansour \cite{QM_2016} proved that $q+9\leq c(K(p,q))$ and Lee and Takioka \cite{LT_2017} proved that $\alpha(K(p,q)) \leq q+7$. Therefore
\[c(K(p,q))+2 - \alpha(K(p,q)) \geq (q+9) + 2 - (q+7) = 4 \geq 2g_T(K(p,q)).\]

Finally, suppose that $pq<0$ and $\min\{|p|,|q|\} \geq 3$. Qazaqzeh and Mansour \cite{QM_2016} proved that $|p|+|q|+7 \leq c(K(p,q)) \leq |p|+|q|+8$, and Lee and Takioka \cite{LT_2017} proved that $\alpha(K(p,q)) \leq |p| + |q| + 5$. Therefore,
\[c(K(p,q)) + 2 - \alpha(K(p,q)) \geq (|p|+|q| +7) + 2 - (|p| + |q| +5) = 4\geq 2g_T(K(p,q)). \qedhere \]

\end{proof}

In the case excluded from Theorem~\ref{thm:Kanenobu}, i.e. when $pq<0$, $|p|=2$, $|q|\geq 7$, and $q$ is odd, Qazaqzeh and Mansour conjectured that the crossing number of $K(p,q)$ is $c(K(p,q)) = |p|+ |q| + 8$. If their crossing number conjecture is true, then Conjecture \ref{conj:c+2-a_gT} will also hold in this case.

\bibliographystyle{plain}
\bibliography{dVVL}

\end{document}

%% file: arcpres.tex
\[\begin{tikzpicture}

\draw[red, thick] (0,.5) -- (0,5.5);

\begin{knot}[	
	consider self intersections,
 	clip width = 6,
 	ignore endpoint intersections = true,
	end tolerance = 2pt
	]
	\flipcrossings{1,2,3};
	\strand[thick] (0,2) arc (-90:90:1.5 cm and 1.5cm);
	\strand[thick] (0,1) arc (-90:90:1.3 cm and 1cm);
	\strand[thick] (0,2) arc (-90:90:.8 cm and 1cm);
	\strand[thick] (0,1) arc (270:90:1.5cm);
	\strand[thick] (0,3) arc (270:90:1cm);
	
\end{knot}

\draw (.8,3.7) node{1};
\draw (1.65, 3.8) node{3};
\draw (1.4,1.5) node{2};
\draw (-1.65,2.5) node{4};
\draw (-1.15,4) node{5};

\begin{scope}[xshift = 7cm,yshift = 3cm]

\draw[thick] (0:0) -- (18:1.7);
\draw (18:2) node {25};
\draw[thick] (0:0) -- (90:1.7);
\draw (90:2) node {14};
\draw[thick] (0:0) -- (162:1.7);
\draw (162:2) node {35};
\draw[thick] (0:0) -- (234:1.7);
\draw (234:2) node {24};
\draw[thick] (0:0) -- (306:1.7);
\draw (306:2) node {13};

\end{scope}


\begin{scope}[yshift = -6cm]

\draw[red, thick] (0,.5) -- (0,5.5);

\begin{knot}[	
	consider self intersections,
 	clip width = 6,
 	ignore endpoint intersections = true,
	end tolerance = 2pt
	]
	\flipcrossings{1,4};
	\strand[thick] (0,1) -- (-2,.5) -- (-2,3.5) -- (0,4) -- (.7,3.5) -- (.7,1.5) -- (0,2) -- (2.5,1.5) -- (2.5,4.5) -- (0,5) -- (-1,4.5) -- (-1,2.5) -- (0,3) -- (1.5,2.5) -- (1.5,.5) -- (0,1);
	
\end{knot}
\draw[thick] (.7,1.5) -- (0,2);

\end{scope}


\begin{scope}[xshift = 4.5cm, yshift = -6cm]

\begin{knot}[	
	consider self intersections,
 	clip width = 8,
 	ignore endpoint intersections = true,
	end tolerance = 2pt
	]
	\flipcrossings{2};
	\strand[thick] (1,1) -- (4,1) -- (4,3) -- (2,3) -- (2,5) -- (5,5) -- (5,2) -- (3,2) -- (3,4) -- (1,4) -- (1,1);
	
\end{knot}

\end{scope}

\end{tikzpicture}\]

%% file: front.tex
\[\begin{tikzpicture}[scale = .7]

\begin{scope}[rotate=45]
\begin{knot}[	
	consider self intersections,
 	clip width = 8,
 	ignore endpoint intersections = false,
	end tolerance = 1pt
	]
	\flipcrossings{3,4,5};
	\strand[thick] (1.25,1) to [out = 0, in = 135]
	(4,1) to [out = 135, in = 270]
	(4,2.75) to [out = 90, in = 0]
	(3.75,3) to [out = 180, in = 0]
	(2.25,3) to [out = 180, in = 270]
	(2,3.25) to [out = 90, in = 315]
	(2,5) to [out = 315, in= 180]
	(4.75,5) to [out = 0, in = 90]
	(5,4.75) to [out =270, in = 135]
	(5,2) to [out = 135, in = 0]
	(3.25,2) to [out = 180, in = 270]
	(3,2.25) to [out = 90, in = 270]
	(3,3.75) to [out = 90, in = 0]
	(2.75,4) to [out = 180, in = 315]
	(1,4) to [out = 315, in = 90]
	(1,1.25) to [out = 270, in = 180]
	(1.25,1);

\end{knot}
\end{scope}

\begin{scope}[rotate = -45, yshift = 5cm, xshift = -1cm]
\begin{knot}[	
	consider self intersections,
 	clip width = 8,
 	ignore endpoint intersections = true,
	end tolerance = 10pt
	]
	\flipcrossings{1,3};
	\strand[thick] (1,1) to [out = 45, in = 180]
	(3.75,1) to [out = 0, in = 270]
	(4,1.25) to [out = 90, in = 225]
	(4,3) to [out = 225, in = 45]
	(2,3) to [out = 45, in = 270]
	(2,4.75) to [out = 90, in = 180]
	(2.25,5) to [out = 0, in = 225]
	(5,5) to [out = 225, in = 90]
	(5,2.25) to [out = 270, in = 0]
	(4.75, 2) to [out = 180, in = 45]
	(3,2) to [out = 45, in = 225]
	(3,4) to [out = 225, in = 0]
	(1.25,4) to [out = 180, in = 90]
	(1,3.75) to [out = 270, in= 45]
	(1,1);
	
	
\end{knot}
\end{scope}
\draw (0,1) node {$F_D$};
\draw (7,1) node{$F_{D^*}$};

\end{tikzpicture}\]

%% file: tree1.tex
\[\begin{tikzpicture}[scale=.7]


\begin{scope}[xshift = -8cm]

\draw [black,ultra thick,domain=125:235] plot ({cos(\x)}, {sin(\x)});
  \draw [black, ultra thick,domain=5:115] plot ({cos(\x)}, {sin(\x)});
 \draw [black,thick,domain= -15:-5] plot ({cos(\x)}, {sin(\x)});
 \draw [black,thick,domain= -45:-15] plot ({cos(\x)}, {sin(\x)});
 \draw [black, thick, domain = -45:-75] plot ({cos(\x)}, {sin(\x)});
  \draw [black, thick, domain = 255:285] plot ({cos(\x)}, {sin(\x)});
   \draw [black, thick, domain = 245:255] plot ({cos(\x)}, {sin(\x)});
  
   \draw [black, ultra thick, domain = 125:235] plot ({2*cos(\x)}, {2*sin(\x)});
    \draw [black, ultra thick,domain=5:115] plot ({2*cos(\x)}, {2*sin(\x)});
     \draw [black,thick,domain= -15:-5] plot ({2*cos(\x)}, {2*sin(\x)});
 \draw [black,thick,domain= -50:-15] plot ({2*cos(\x)}, {2*sin(\x)});
 \draw [black, thick, domain = -50:-70] plot ({2*cos(\x)}, {2*sin(\x)});
  \draw [black, thick, domain = 255:290] plot ({2*cos(\x)}, {2*sin(\x)});
   \draw [black, thick, domain = 245:255] plot ({2*cos(\x)}, {2*sin(\x)});
   
   \draw[black, thick] (240:2.2) to [out = 60, in = 240]
   (240:1) to [out = 60, in = 120]
   (-60:.9);
   \draw[black, ultra thick] (240:2) to [out = 60, in = 240]
   (240:1) to [out = 60, in = 120]
   (-60:.9);
   \draw[black, thick] (240:2.2) to [out = 240, in = 180, looseness=1.5] (180:2.3);
   \draw[black,thick] (180:2.3) to [out = 0, in =180] (180:2.1);
   \draw[black,thick] (180:1.9) -- (180:1.7);
   \draw[black,thick] (180:1.7) -- (180:1.3);
   \draw[black,thick] (180:1.3) -- (180:1.1);
   \draw [black,thick] (180:.9) -- (180:.7);
   \draw [black,thick] (120:.8) -- (120:1.2);
   \draw [black,thick] (180:.7) to [out = 0, in = -60] (120:.8);
   \draw[black,thick] (120:1.2) -- (120:1.8);
   \draw[black,thick] (120:1.8) -- (120:2.2);
   \draw[black,thick] (120:2.2) to [out = 120, in = 60,looseness=1.5] (60:2.3);
   \draw[black, thick] (60:2.3) -- (60:2.1);
   \draw[black, thick] (60:1.9) -- (60:1.7);
   \draw[black,thick] (60:1.7) -- (60:1.3);
   \draw[black,thick] (60:1.3) -- (60:1.1);
   \draw[black,thick] (60:.9) -- (60:.7);
   \draw[black,thick] (0:.8) -- (0:1.2);
   \draw[black,thick] (60:.7) to [out = 240, in = 180] (0:.8);
   \draw[black,ultra thick] (300:1.1) -- (300:1.9);
   \draw[black,thick] (300:2.1) -- (300:2.3);
   \draw[black,thick] (0:1.8) -- (0:2.2);
   \draw[black,thick] (0:1.2) -- (0:1.8);
   \draw[black,thick] (0:2.2) to [out = 0, in = 300, looseness=1.5] (300:2.3);

\end{scope}


\draw[densely dotted, black!70!white, thick] (-15:1) to [out = -15, in = 270]
(0:1.3) to [out = 90, in = -30]
(60:1.3) to [out = 150, in = 30]
(120:1.3) to [out = 210, in = 90]
(180:1.3) to [out = 270, in = 150]
(225:1.3) to [out = 330, in = 60, looseness=1.5]
(235:1.5) to [out = 240, in = 330, looseness=1.5]
(225:1.7) to [out = 150, in = 270]
(180:1.7) to [out = 90, in = 210]
(120:1.7) to [out = 30, in = 150]
(60:1.7) to [out = -30, in = 90]
(0:1.7) to [out = 270, in = 180]
(-10:2) to [out = 0, in = 270]
(0: 2.3) to [out = 90, in = -30]
(60:2.3) to [out = 150, in = 30]
(120:2.3) to [out = 210, in = 90]
(180:2.3) to [out = 270, in = 150]
(240:2.3) to [out = 330, in = 250]
(250:2) to [out = 70, in = 250]
(255:1) to [out = 70, in = 105]
(-75:1) to [out = 285, in = 105]
(-70:2) to [out = 285, in = 210]
(300:2.3) to [out = 30, in = 285]
(310:2) to [out = 105, in = 285]
(315:1) to [out = 105, in = 0]
(270:.4) to [out = 180, in = 270, looseness = 1.5]
(180:.7) to [out = 90, in = 210]
(120:.8) to [out = 30, in = 150]
(60:.7) to [out = -30, in = 90]
(0:.8) to [out = 270, in = 180]
(-15:1);

 \draw [blue,thick,domain=125:235] plot ({cos(\x)}, {sin(\x)});
  \draw [blue,thick,domain=5:115] plot ({cos(\x)}, {sin(\x)});
 \draw [blue,thick,domain= -15:-5] plot ({cos(\x)}, {sin(\x)});
 \draw [red,thick,domain= -45:-15] plot ({cos(\x)}, {sin(\x)});
 \draw [blue, thick, domain = -45:-75] plot ({cos(\x)}, {sin(\x)});
  \draw [red, thick, domain = 255:285] plot ({cos(\x)}, {sin(\x)});
   \draw [blue, thick, domain = 245:255] plot ({cos(\x)}, {sin(\x)});
  
   \draw [blue, thick, domain = 125:235] plot ({2*cos(\x)}, {2*sin(\x)});
    \draw [blue,thick,domain=5:115] plot ({2*cos(\x)}, {2*sin(\x)});
     \draw [blue,thick,domain= -10:-5] plot ({2*cos(\x)}, {2*sin(\x)});
 \draw [red,thick,domain= -50:-10] plot ({2*cos(\x)}, {2*sin(\x)});
 \draw [blue, thick, domain = -50:-70] plot ({2*cos(\x)}, {2*sin(\x)});
  \draw [red, thick, domain = 250:290] plot ({2*cos(\x)}, {2*sin(\x)});
   \draw [blue, thick, domain = 245:250] plot ({2*cos(\x)}, {2*sin(\x)});
   
   \draw[blue, thick] (240:2.3) to [out = 60, in = 240]
   (240:1) to [out = 60, in = 120]
   (-60:.9);
   \draw[red, thick] (240:2.3) to [out = 240, in = 180, looseness=1.5] (180:2.3);
   \draw[blue,thick] (180:2.3) to [out = 0, in =180] (180:2.1);
   \draw[blue,thick] (180:1.9) -- (180:1.7);
   \draw[red,thick] (180:1.7) -- (180:1.3);
   \draw[blue,thick] (180:1.3) -- (180:1.1);
   \draw [blue,thick] (180:.9) -- (180:.7);
   \draw [blue,thick] (120:.8) -- (120:1.3);
   \draw [red,thick] (180:.7) to [out = 0, in = -60] (120:.8);
   \draw[red,thick] (120:1.3) -- (120:1.7);
   \draw[blue,thick] (120:1.7) -- (120:2.3);
   \draw[red,thick] (120:2.3) to [out = 120, in = 60,looseness=1.5] (60:2.3);
   \draw[blue, thick] (60:2.3) -- (60:2.1);
   \draw[blue, thick] (60:1.9) -- (60:1.7);
   \draw[red,thick] (60:1.7) -- (60:1.3);
   \draw[blue,thick] (60:1.3) -- (60:1.1);
   \draw[blue,thick] (60:.9) -- (60:.7);
   \draw[blue,thick] (0:.8) -- (0:1.3);
   \draw[red,thick] (60:.7) to [out = 240, in = 180] (0:.8);
   \draw[blue,thick] (300:1.1) -- (300:1.9);
   \draw[blue,thick] (300:2.1) -- (300:2.3);
   \draw[blue,thick] (0:1.7) -- (0:2.3);
   \draw[red,thick] (0:1.3) -- (0:1.7);
   \draw[red,thick] (0:2.3) to [out = 0, in = 300, looseness=1.5] (300:2.3);

\end{tikzpicture}\]

%% file: tree2.tex
\[\begin{tikzpicture}[scale=.6]

\begin{knot}[
	consider self intersections,
 	clip width = 5,
 	ignore endpoint intersections = false,
	end tolerance = 1pt
	]
	\flipcrossings{1,3,5,6,8,10,11,12};
	\strand[blue, thick] (0:4) to [out= 180, in = {180+4*360/26}] ({4*360/26}:4);
	\strand[blue, thick] ({360/26}:4) to [out= {180+360/26}, in = {180+3*360/26}, looseness=1] ({3*360/26}:4);
	\strand[blue, thick] ({2*360/26}:4) to [out= {180+2*360/26}, in = {180+18*360/26}] ({18*360/26}:4);
	\strand[blue, thick] ({5*360/26}:4) to [out= {180+5*360/26}, in = {180+25*360/26}, looseness = .9] ({25*360/26}:4);
	\strand[blue, thick] ({6*360/26}:4) to [out= {180+6*360/26}, in = {180+24*360/26}] ({24*360/26}:4);
	\strand[blue, thick] ({7*360/26}:4) to [out= {180+7*360/26}, in = {180+15*360/26}] ({15*360/26}:4);
	\strand[blue, thick] ({8*360/26}:4) to [out= {180+8*360/26}, in = {180+14*360/26}] ({14*360/26}:4);
	\strand[blue, thick] ({9*360/26}:4) to [out= {180+9*360/26}, in = {180+13*360/26}] ({13*360/26}:4);
	\strand[blue, thick] ({10*360/26}:4) to [out= {180+10*360/26}, in = {180+12*360/26}] ({12*360/26}:4);
	\strand[blue, thick] ({11*360/26}:4) to [out= {180+11*360/26}, in = {180+17*360/26}] ({17*360/26}:4);
	\strand[blue, thick] ({16*360/26}:4) to [out= {180+16*360/26}, in = {180+21*360/26}] ({21*360/26}:4);
	\strand[blue, thick] ({19*360/26}:4) to [out= {180+19*360/26}, in = {180+23*360/26}] ({23*360/26}:4);
	\strand[blue, thick] ({20*360/26}:4) to [out= {180+20*360/26}, in = {180+22*360/26}] ({22*360/26}:4);
\end{knot}

\begin{scope}[red,thick]

\draw (0:4) to [out = 0, in = {360/26}, looseness=1.5] ({360/26}:4);
\draw ({2*360/26}:4) to [out = {2*360/26}, in = {23*360/26}, looseness=1.5] ({23*360/26}:4);
\draw ({3*360/26}:4) to [out = {3*360/26}, in = {10*360/26}, looseness=3] ({10*360/26}:4);
\draw ({4*360/26}:4) to [out = {4*360/26}, in = {9*360/26}, looseness=2] ({9*360/26}:4);
\draw ({5*360/26}:4) to [out = {5*360/26}, in = {8*360/26}, looseness=1.5] ({8*360/26}:4);
\draw ({6*360/26}:4) to [out = {6*360/26}, in = {7*360/26}, looseness=1.5] ({7*360/26}:4);
\draw ({11*360/26}:4) to [out = {11*360/26}, in = {16*360/26-90}, looseness=1.5]
({16*360/26}:6.5) to [out = {16*360/26+90}, in = {22*360/26}, looseness=1.5] ({22*360/26}:4);
\draw ({12*360/26}:4) to [out = {12*360/26}, in = {16*360/26-90}, looseness=1.5]
({16*360/26}:5.5) to [out = {16*360/26+90}, in = {21*360/26}, looseness=1.5] ({21*360/26}:4);
\draw ({13*360/26}:4) to [out = {13*360/26}, in = {14*360/26}, looseness=1.5] ({14*360/26}:4);
\draw ({15*360/26}:4) to [out = {15*360/26}, in = {16*360/26}, looseness=1.5] ({16*360/26}:4);
\draw ({17*360/26}:4) to [out = {17*360/26}, in = {20*360/26}, looseness=1.5] ({20*360/26}:4);
\draw ({18*360/26}:4) to [out = {18*360/26}, in = {19*360/26}, looseness=1.5] ({19*360/26}:4);
\draw ({24*360/26}:4) to [out = {24*360/26}, in = {25*360/26}, looseness=1.5] ({25*360/26}:4);

\end{scope}

\draw ({0*360/26-2}:4.4) node{\footnotesize{2}};
\draw ({1*360/26+2}:4.4) node{\footnotesize{10}};
\draw ({2*360/26+2.5}:4.4) node{\footnotesize{6}};
\draw ({3*360/26-3}:4.4) node{\footnotesize{10}};
\draw ({4*360/26-2}:4.4) node{\footnotesize{2}};
\draw ({5*360/26-2}:4.4) node{\footnotesize{8}};
\draw ({6*360/26-2}:4.4) node{\footnotesize{4}};
\draw ({7*360/26+2}:4.4) node{\footnotesize{3}};
\draw ({8*360/26+2}:4.4) node{\footnotesize{9}};
\draw ({9*360/26+2}:4.4) node{\footnotesize{1}};
\draw ({10*360/26+3}:4.4) node{\footnotesize{11}};
\draw ({11*360/26-3}:4.4) node{\footnotesize{5}};
\draw ({12*360/26-3}:4.4) node{\footnotesize{11}};
\draw ({13*360/26-2.5}:4.4) node{\footnotesize{1}};
\draw ({14*360/26+2.5}:4.4) node{\footnotesize{9}};
\draw ({15*360/26-2.5}:4.4) node{\footnotesize{3}};
\draw ({16*360/26+2.5}:4.4) node{\footnotesize{7}};
\draw ({17*360/26-2.5}:4.4) node{\footnotesize{5}};
\draw ({18*360/26-2.5}:4.4) node{\footnotesize{6}};
\draw ({19*360/26+3}:4.4) node{\footnotesize{12}};
\draw ({20*360/26+3}:4.4) node{\footnotesize{13}};
\draw ({21*360/26+2.5}:4.4) node{\footnotesize{7}};
\draw ({22*360/26+3.5}:4.4) node{\footnotesize{13}};
\draw ({23*360/26-2.5}:4.4) node{\footnotesize{12}};
\draw ({24*360/26-2.5}:4.4) node{\footnotesize{4}};
\draw ({25*360/26+2.5}:4.4) node{\footnotesize{8}};

\draw  (0,0) circle (4cm);

\draw (-4.3,-4.3) node{*};

\end{tikzpicture}\]

%% file: extarcs.tex
\[ \begin{tikzpicture}


\draw[red,thick] (2,0) arc (180:0:.5cm);
\draw[red,thick] (1,0) arc (180:0:1.5cm);
\draw[blue, thick] (1,0) -- (1,-1);
\draw[blue, thick] (2,0) -- (2,-1);
\draw[blue, thick] (3,0) -- (3,-1);
\draw[blue, thick] (4,0) -- (4,-1);
\draw[thick] (.5,0) -- (4.5,0);
\draw (.8,.2) node {$a$};
\draw (1.8,.2) node {$b$};
\draw (3.2,.2) node {$c$};
\draw (3.8,.2) node {$d$};

\draw (2.5,-1.5) node{Initial position};


\begin{scope}[xshift = 7cm]

\draw[thick] (.5,0) -- (4.5,0);
\draw[red,thick] (1,0) arc (180:0:.5cm);
\draw[red,thick] (3,0) arc (180:0:.5cm);
\draw (.8,.2) node {$a$};
\draw (2.2,.2) node {$d$};
\draw (2.8,.2) node {$b$};
\draw (4.2,.2) node {$c$};
\begin{knot}[	
	consider self intersections,
 	clip width = 8,
 	ignore endpoint intersections = true,
	end tolerance = 2pt
	]
	\flipcrossings{1,2};
	\strand[blue, thick]  (3,0) to [out = 270, in = 90] (2,-1);
	\strand[blue, thick]  (4,0) to [out = 270, in = 90] (3,-1);
	\strand[blue, thick] (2,0) to [out = 270, in = 90] (4,-1);
	\strand[blue, thick] (1,0) to [out = 270, in = 90] (1,-1);

\end{knot}

\draw (2.5,-1.5) node{$d > \max\{b,c\}$};

\end{scope}

\begin{scope}[yshift = -3cm]

\draw[thick] (.5,0) -- (4.5,0);
\draw[red,thick] (1,0) arc (180:0:.5cm);
\draw[red,thick] (3,0) arc (180:0:.5cm);
\draw (.8,.2) node {$b$};
\draw (2.2,.2) node {$c$};
\draw (2.8,.2) node {$a$};
\draw (4.2,.2) node {$d$};
\begin{knot}[	
	consider self intersections,
 	clip width = 8,
 	ignore endpoint intersections = true,
	end tolerance = 2pt
	]
	\strand[blue, thick]  (1,0) to [out = 270, in = 90] (2,-1);
	\strand[blue, thick]  (2,0) to [out = 270, in = 90] (3,-1);
	\strand[blue, thick] (3,0) to [out = 270, in = 90] (1,-1);
	\strand[blue, thick] (4,0) to [out = 270, in = 90] (4,-1);

\end{knot}

\draw (2.5,-1.5) node{$a < \min\{b,c\}$};

\end{scope}

\begin{scope}[xshift = 6cm, yshift = -3cm]

\draw[thick] (.5,0) -- (6.5,0);
\draw[red,thick] (1,0) arc (180:0:.5cm);
\draw[red,thick] (3,0) arc (180:0:.5cm);
\draw[red,thick] (5,0) arc (180:0:.5cm);
\draw (.8,.2) node {$a$};
\draw (2.2,.2) node {$M$};
\draw (2.8,.2) node {$b$};
\draw (4.2,.2) node {$c$};
\draw (4.8,.2) node {$M$};
\draw (6.2,.2) node {$d$};
\begin{knot}[	
	consider self intersections,
 	clip width = 8,
 	ignore endpoint intersections = true,
	end tolerance = 2pt
	]
	\flipcrossings{1,2};
	\strand[blue, thick]  (1,0) to [out = 270, in = 90] (1,-1);
	\strand[blue, thick]  (3,0) to [out = 270, in = 90] (3,-1);
	\strand[blue, thick] (4,0) to [out = 270, in = 90] (4,-1);
	\strand[blue, thick] (6,0) to [out = 270, in = 90] (6,-1);
	\strand[blue,thick] (2,0) to [out = 270, in = 270, looseness = .75] (5,0);

\end{knot}

\draw (3.5,-1.5) node{$b<a<d<c$};

\end{scope}

\end{tikzpicture}\]

%% file: tree3.tex
\[\begin{tikzpicture}[scale=.75]

\begin{knot}[
	consider self intersections,
 	clip width = 5,
 	ignore endpoint intersections = false,
	end tolerance = 1pt
	]
	\flipcrossings{1,2,3,4,13,16,17,18,19,20,21,22,23,24,25,26,27,28,29,30,31,33,34,36,38,39};
	\strand[blue, thick] (0:4) to [out= 180, in = {180+4*360/28},looseness = .3] ({4*360/28}:4);
	\strand[blue, thick] ({360/28}:4) to [out= {180+360/28},  in = -45]
	(45:1.7) to [out = 135, in = -5]
	(100:3) to [out = 175, in = {180+10*360/28}, looseness=1] ({10*360/28}:4);
	\strand[blue, thick] ({2*360/28}:4) to [out= {180+2*360/28}, in = {180+23*360/28}, looseness=.5] ({23*360/28}:4);
	\strand[blue, thick] ({3*360/28}:4) to [out= {180+3*360/28}, in = {180+22*360/28}, looseness=1] ({22*360/28}:4);
	\strand[blue, thick] ({5*360/28}:4) to [out= {180+5*360/28}, in = 10]
	(90:.5) to [out = 190, in = 90]
	(190:1.5) to [out = 270, in = {180+16*360/28}, looseness=1] ({16*360/28}:4);
	\strand[blue, thick] ({6*360/28}:4) to [out= {180+6*360/28}, in = 140]
	(50:2.8) to [out = 320, in = 60]
	(25:2.7) to [out = 240, in = 100]
	(-30:1.8) to [out = 280, in = {180+26*360/28}] ({26*360/28}:4);
	\strand[blue, thick] ({7*360/28}:4) to [out= {180+7*360/28}, in = {180+18*360/28}, looseness=1] ({18*360/28}:4);
	\strand[blue, thick] ({8*360/28}:4) to  [out= {180+8*360/28}, in = 135, looseness=1] (45:.5) to [out = -45, in = {180+27*360/28}, looseness=1] ({27*360/28}:4);
	\strand[blue, thick] ({9*360/28}:4) to [out= {180+9*360/28}, in = 90]
	(135:2) to [out = 270, in = 90]
	(185:2.5) to [out = 270, in = 120]
	(210:2.5) to [out = 300, in = {180+17*360/28}, looseness=1] ({17*360/28}:4);
	\strand[blue, thick] ({11*360/28}:4) to [out= {180+11*360/28}, in = {180+14*360/28}, looseness=1] ({14*360/28}:4);
	\strand[blue, thick] ({12*360/28}:4) to [out= {180+12*360/28}, in = {180+21*360/28}, looseness=1] ({21*360/28}:4);
	\strand[blue, thick] ({13*360/28}:4) to [out= {180+13*360/28}, in = {180+20*360/28}, looseness=1] ({20*360/28}:4);
	\strand[blue, thick] ({15*360/28}:4) to [out= {180+15*360/28}, in=150]
	(240:3.5) to [out = 330, in = {180+24*360/28}] ({24*360/28}:4);
	\strand[blue, thick] ({19*360/28}:4) to [out= {180+19*360/28}, in = {180+25*360/28}, looseness=.75] ({25*360/28}:4);

	\end{knot}
	
	\begin{scope}[red,thick]

\draw (0:4) to [out = 0, in = {360/28}, looseness=1.5] ({360/28}:4);
\draw ({2*360/28}:4) to [out = {2*360/28}, in = {3*360/28}, looseness=1.5] ({3*360/28}:4);
\draw ({4*360/28}:4) to [out = {4*360/28}, in = {5*360/28}, looseness=1.5] ({5*360/28}:4);
\draw ({6*360/28}:4) to [out = {6*360/28}, in = {7*360/28}, looseness=1.5] ({7*360/28}:4);
\draw ({8*360/28}:4) to [out = {8*360/28}, in = {9*360/28}, looseness=1.5] ({9*360/28}:4);
\draw ({10*360/28}:4) to [out = {10*360/28}, in = {11*360/28}, looseness=1.5] ({11*360/28}:4);
\draw ({12*360/28}:4) to [out = {12*360/28}, in = {13*360/28}, looseness=1.5] ({13*360/28}:4);
\draw ({14*360/28}:4) to [out = {14*360/28}, in = {15*360/28}, looseness=1.5] ({15*360/28}:4);
\draw ({16*360/28}:4) to [out = {16*360/28}, in = {17*360/28}, looseness=1.5] ({17*360/28}:4);
\draw ({18*360/28}:4) to [out = {18*360/28}, in = {19*360/28}, looseness=1.5] ({19*360/28}:4);
\draw ({20*360/28}:4) to [out = {20*360/28}, in = {21*360/28}, looseness=1.5] ({21*360/28}:4);
\draw ({22*360/28}:4) to [out = {22*360/28}, in = {23*360/28}, looseness=1.5] ({23*360/28}:4);
\draw ({24*360/28}:4) to [out = {24*360/28}, in = {25*360/28}, looseness=1.5] ({25*360/28}:4);
\draw ({26*360/28}:4) to [out = {26*360/28}, in = {27*360/28}, looseness=1.5] ({27*360/28}:4);

\end{scope}

\draw ({0*360/28-2.5}:4.4) node{\footnotesize{2}};
\draw ({1*360/28+2.5}:4.4) node{\footnotesize{10}};
\draw ({2*360/28-2.5}:4.4) node{\footnotesize{12}};
\draw ({3*360/28+2.5}:4.4) node{\footnotesize{6}};
\draw ({4*360/28-2.5}:4.4) node{\footnotesize{2}};
\draw ({5*360/28+2.5}:4.4) node{\footnotesize{1}};
\draw ({6*360/28-2.5}:4.4) node{\footnotesize{4}};
\draw ({7*360/28+2.5}:4.4) node{\footnotesize{3}};
\draw ({8*360/28-2.5}:4.4) node{\footnotesize{8}};
\draw ({9*360/28+2.5}:4.4) node{\footnotesize{9}};
\draw ({10*360/28-2.5}:4.4) node{\footnotesize{10}};
\draw ({11*360/28+2.5}:4.4) node{\footnotesize{11}};
\draw ({12*360/28-2.5}:4.4) node{\footnotesize{13}};
\draw ({13*360/28+2.5}:4.4) node{\footnotesize{5}};
\draw ({14*360/28-2.5}:4.4) node{\footnotesize{11}};
\draw ({15*360/28+2.5}:4.4) node{\footnotesize{14}};
\draw ({16*360/28-2.5}:4.4) node{\footnotesize{1}};
\draw ({17*360/28+2.5}:4.4) node{\footnotesize{9}};
\draw ({18*360/28-2.5}:4.4) node{\footnotesize{3}};
\draw ({19*360/28+2.5}:4.4) node{\footnotesize{7}};
\draw ({20*360/28-2.5}:4.4) node{\footnotesize{5}};
\draw ({21*360/28+2.5}:4.4) node{\footnotesize{13}};
\draw ({22*360/28-2.5}:4.4) node{\footnotesize{6}};
\draw ({23*360/28+2.5}:4.4) node{\footnotesize{12}};
\draw ({24*360/28-2.5}:4.4) node{\footnotesize{14}};
\draw ({25*360/28+2.5}:4.4) node{\footnotesize{7}};
\draw ({26*360/28-2.5}:4.4) node{\footnotesize{4}};
\draw ({27*360/28+2.5}:4.4) node{\footnotesize{8}};

\draw  (0,0) circle (4cm);
\begin{scope}[xshift = 12cm]
\begin{knot}[
	consider self intersections,
 	clip width = 5,
 	ignore endpoint intersections = false,
	end tolerance = 1pt
	]
	\flipcrossings{1,2,3,4,5,6,7,8,9,10,11,12,18,22,25,26,27,28,29,30,31,32,33,35,36,38,39};
	\strand[blue, thick] (0:4) to [out= 180, in = 300]
	(30:3.8) to [out = 120, in = {180+4*360/28}] 
	({4*360/28}:3.7)
	to [out = 4*360/28, in ={5*360/28}]  ({5*360/28}:3.7) to [out= {180+5*360/28}, in = 10]
	(90:.5) to [out = 190, in = 90]
	(190:1.5) to [out = 270, in = {180+16*360/28}, looseness=1] ({16*360/28}:4);
	{5*360/28}:4) to [out= {180+5*360/28}, in = 10]
	(90:.5) to [out = 190, in = 90]
	(190:1.5) to [out = 270, in = {180+16*360/28}, looseness=1] ({16*360/28}:4);
	\strand[blue, thick] ({360/28}:4) to [out= {180+360/28},  in = -45]
	(45:1.7) to [out = 135, in = -5]
	(100:3) to [out = 175, in = {180+10*360/28}, looseness=1] 
	({10*360/28}:3.7) to [out = 10*360/28, in = {11*360/28}]
	 ({11*360/28}:3.7) to [out= {180+11*360/28}, in = {180+14*360/28}, looseness=1] ({14*360/28}:4);
	\strand[blue, thick] ({2*360/28}:4) to [out= {180+2*360/28}, in = {180+23*360/28}, looseness=.5] ({23*360/28}:4);
	\strand[blue, thick] ({3*360/28}:4) to [out= {180+3*360/28}, in = {180+22*360/28}, looseness=1] ({22*360/28}:4);
	\strand[blue, thick]  ({18*360/28}:4) to [out ={180+18*360/28}, in ={180+7*360/28}]
	({7*360/28}:3.7) to [out = {7*360/28}, in = {6*360/28}] 
	({6*360/28}:3.7) to [out= {180+6*360/28}, in = 140]
	(50:2.8) to [out = 320, in = 60]
	(25:2.7) to [out = 240, in = 100]
	(-30:1.8) to [out = 280, in = {180+26*360/28}] ({26*360/28}:4);
	\strand[blue, thick] ({27*360/28}:4) to [out = {180+27*360/28}, in = -45]
	(45:.5) to [out = 135, in = {180+8*360/28}]
	({8*360/28}:3.7) to [out = {8*360/28}, in = {9*360/28}]
	({9*360/28}:3.7) to [out= {180+9*360/28}, in = 90]
	(135:2) to [out = 270, in = 90]
	(185:2.5) to [out = 270, in = 120]
	(210:2.5) to [out = 300, in = {180+17*360/28}, looseness=1] ({17*360/28}:4);
	\strand[blue, thick] ({12*360/28}:4) to [out= {180+12*360/28}, in = {180+21*360/28}, looseness=1] ({21*360/28}:4);
	\strand[blue, thick] ({13*360/28}:4) to [out= {180+13*360/28}, in = {180+20*360/28}, looseness=1] ({20*360/28}:4);
	\strand[blue, thick] ({15*360/28}:4) to [out= {180+15*360/28}, in=150]
	(240:3.5) to [out = 330, in = {180+24*360/28}] ({24*360/28}:4);
	\strand[blue, thick] ({19*360/28}:4) to [out= {180+19*360/28}, in = {180+25*360/28}, looseness=.75] ({25*360/28}:4);

	\end{knot}
	
	\begin{scope}[red,thick]

\draw (0:4) to [out = 0, in = {360/28}, looseness=1.5] ({360/28}:4);
\draw ({2*360/28}:4) to [out = {2*360/28}, in = {3*360/28}, looseness=1.5] ({3*360/28}:4);
\draw ({12*360/28}:4) to [out = {12*360/28}, in = {13*360/28}, looseness=1.5] ({13*360/28}:4);
\draw ({14*360/28}:4) to [out = {14*360/28}, in = {15*360/28}, looseness=1.5] ({15*360/28}:4);
\draw ({16*360/28}:4) to [out = {16*360/28}, in = {17*360/28}, looseness=1.5] ({17*360/28}:4);
\draw ({18*360/28}:4) to [out = {18*360/28}, in = {19*360/28}, looseness=1.5] ({19*360/28}:4);
\draw ({20*360/28}:4) to [out = {20*360/28}, in = {21*360/28}, looseness=1.5] ({21*360/28}:4);
\draw ({22*360/28}:4) to [out = {22*360/28}, in = {23*360/28}, looseness=1.5] ({23*360/28}:4);
\draw ({24*360/28}:4) to [out = {24*360/28}, in = {25*360/28}, looseness=1.5] ({25*360/28}:4);
\draw ({26*360/28}:4) to [out = {26*360/28}, in = {27*360/28}, looseness=1.5] ({27*360/28}:4);

\end{scope}

\draw ({0*360/28-2.5}:4.4) node{\footnotesize{1}};
\draw ({1*360/28+2.5}:4.4) node{\footnotesize{7}};
\draw ({2*360/28-2.5}:4.4) node{\footnotesize{8}};
\draw ({3*360/28+2.5}:4.4) node{\footnotesize{4}};
\draw ({12*360/28-2.5}:4.4) node{\footnotesize{9}};
\draw ({13*360/28+2.5}:4.4) node{\footnotesize{3}};
\draw ({14*360/28-2.5}:4.4) node{\footnotesize{7}};
\draw ({15*360/28+2.5}:4.4) node{\footnotesize{10}};
\draw ({16*360/28-2.5}:4.4) node{\footnotesize{1}};
\draw ({17*360/28+2.5}:4.4) node{\footnotesize{6}};
\draw ({18*360/28-2.5}:4.4) node{\footnotesize{2}};
\draw ({19*360/28+2.5}:4.4) node{\footnotesize{5}};
\draw ({20*360/28-2.5}:4.4) node{\footnotesize{3}};
\draw ({21*360/28+2.5}:4.4) node{\footnotesize{9}};
\draw ({22*360/28-2.5}:4.4) node{\footnotesize{4}};
\draw ({23*360/28+2.5}:4.4) node{\footnotesize{8}};
\draw ({24*360/28-2.5}:4.4) node{\footnotesize{10}};
\draw ({25*360/28+2.5}:4.4) node{\footnotesize{5}};
\draw ({26*360/28-2.5}:4.4) node{\footnotesize{2}};
\draw ({27*360/28+2.5}:4.4) node{\footnotesize{6}};

\draw  (0,0) circle (4cm);

\end{scope}

\end{tikzpicture}\]

%% file: arcexample.tex
\[\begin{tikzpicture}
\foreach \i in {0,...,9}
{
	\draw[thick] (0:0) -- ({\i*36}:2);
}

\draw (0:2.3) node{$1\;7$};
\draw ({1*36}:2.3) node{$4\;8$};
\draw ({2*36}:2.3) node{$3\;9$};
\draw ({3*36}:2.3) node{$7\;10$};
\draw ({4*36}:2.3) node{$1\;6$};
\draw ({5*36}:2.3) node{$2\;5$};
\draw ({6*36}:2.3) node{$3\;9$};
\draw ({7*36}:2.3) node{$4\;8$};
\draw ({8*36}:2.3) node{$5\;10$};
\draw ({9*36}:2.3) node{$2\;6$};

\end{tikzpicture}\]

%% file: filtered.tex
\[\begin{tikzpicture}[scale=.7, thick]
 
 \draw (.3,.3) to [out = 225, in = 225, looseness=1]
 (4.5,-.5);
 \draw[ultra thick] (4.5,-.5) to [out = 45, in = 315, looseness=1]
 (4.7,1.3);
 \draw (5.3,.5) node{$1$};
 \draw[ultra thick] (.7,.7) to [out = 45, in = 135]
 (1.5,.5);
 \draw (1,1.1) node{$4$};
 \draw[ultra thick] (1.5,.5)  to [out = 315, in = 225]
 (2.3,.3);
 \draw (2,-.1) node{$3$};
 \draw (1.7,.7) to [out = 45, in = 135]
 (2.5,.5) to [out = 315, in = 135]
 (4.3,-.3);
 \draw[ultra thick] (2.7,.7) to [out = 45, in = 200]
 (4.5,1.5);
 \draw (3.4,1.5) node{$2$};
 \draw (4.5,1.5) to [out = 20, in = 225]
 (6.3,2.3);
 \draw (6.7,2.7) to [out = 45, in = 45, looseness=1.5]
 (5.5,3.5) to [out = 225, in=315, looseness=1]
 (3.7,3.3);
 \draw (3.3,3.7) to [out = 135, in = 135, looseness=1.2]
 (.5,.5) to [out = 315, in = 225]
 (1.3,.3);
 \draw (4.7,-.7) to [out = 315, in = 315, looseness=1.2]
 (6.5,2.5) to [out = 135, in = 315]
 (5.7,3.3);
 \draw (5.3,3.7) to [out = 135, in = 45]
 (3.5,3.5) to [out = 225, in = 135]
 (4.3,1.7);
 \draw (2,-1.4) node{$f$};
 
 
\begin{scope}[xshift = 9 cm]
 \draw (.3,.3) to [out = 225, in = 225, looseness=1]
 (4.5,-.5);
 \draw[ultra thick] (4.5,-.5) to [out = 45, in = 315, looseness=1]
 (4.7,1.3);
  \draw (5.3,.5) node{$1$};
 \draw (.7,.7) to [out = 45, in = 135]
 (1.5,.5) to [out = 315, in = 225]
 (2.3,.3);
 \draw (1.7,.7) to [out = 45, in = 135]
 (2.5,.5) to [out = 315, in = 135]
 (4.3,-.3);
 \draw (2.7,.7) to [out = 45, in = 200]
 (4.5,1.5) to [out = 20, in = 225]
 (6.3,2.3);
 \draw (6.7,2.7) to [out = 45, in = 45, looseness=1.5]
 (5.5,3.5) to [out = 225, in=315, looseness=1]
 (3.7,3.3);
 \draw (3.3,3.7) to [out = 135, in = 135, looseness=1.2]
 (.5,.5) to [out = 315, in = 225]
 (1.3,.3);
 \draw (4.7,-.7) to [out = 315, in = 315, looseness=1.2]
 (6.5,2.5) to [out = 135, in = 315]
 (5.7,3.3);
 \draw (5.3,3.7) to [out = 135, in = 45]
 (3.5,3.5);
 \draw[ultra thick] (3.5,3.5) to [out = 225, in = 135]
 (4.3,1.7);
 \draw (3.3,2.5) node{$2$};
 \draw[densely dotted] (3.5,3.5) to [out = 90, in = 90, looseness=1.5]
 (-.5,1) to [out = 270, in = 270, looseness=1.5] 
 (4.5,-.5);
\end{scope}
  
\end{tikzpicture}\]

%% file: tree4.tex
\[ \begin{tikzpicture}


\fill[myblue] (180:1.5) circle (.1cm);
\fill[myblue] (120:1.5) circle (.1cm);
\fill[myblue] (60:1.5) circle (.1cm);
\fill[myblue] (0:1.5) circle (.1cm);

\draw [black,ultra thick,domain=125:235] plot ({cos(\x)}, {sin(\x)});
  \draw [black, ultra thick,domain=5:115] plot ({cos(\x)}, {sin(\x)});
 \draw [black,thick,domain= -15:-5] plot ({cos(\x)}, {sin(\x)});
 \draw [black,thick,domain= -45:-15] plot ({cos(\x)}, {sin(\x)});
 \draw [black, thick, domain = -45:-75] plot ({cos(\x)}, {sin(\x)});
  \draw [black, thick, domain = 255:285] plot ({cos(\x)}, {sin(\x)});
   \draw [black, thick, domain = 245:255] plot ({cos(\x)}, {sin(\x)});
  
   \draw [black, ultra thick, domain = 125:235] plot ({2*cos(\x)}, {2*sin(\x)});
    \draw [black, ultra thick,domain=5:115] plot ({2*cos(\x)}, {2*sin(\x)});
     \draw [black,thick,domain= -15:-5] plot ({2*cos(\x)}, {2*sin(\x)});
 \draw [black,thick,domain= -50:-15] plot ({2*cos(\x)}, {2*sin(\x)});
 \draw [black, thick, domain = -50:-70] plot ({2*cos(\x)}, {2*sin(\x)});
  \draw [black, thick, domain = 255:290] plot ({2*cos(\x)}, {2*sin(\x)});
   \draw [black, thick, domain = 245:255] plot ({2*cos(\x)}, {2*sin(\x)});
   
   \draw[black, thick] (240:2.2) to [out = 60, in = 240]
   (240:1) to [out = 60, in = 120]
   (-60:.9);
   \draw[black, ultra thick] (240:2) to [out = 60, in = 240]
   (240:1) to [out = 60, in = 120]
   (-60:.9);
   \draw[black, thick] (240:2.2) to [out = 240, in = 180, looseness=1.5] (180:2.3);
   \draw[black,thick] (180:2.3) to [out = 0, in =180] (180:2.1);
   \draw[black,thick] (180:1.9) -- (180:1.7);
   \draw[black,thick] (180:1.7) -- (180:1.3);
   \draw[black,thick] (180:1.3) -- (180:1.1);
   \draw [black,thick] (180:.9) -- (180:.7);
   \draw [black,thick] (120:.8) -- (120:1.2);
   \draw [black,thick] (180:.7) to [out = 0, in = -60] (120:.8);
   \draw[black,thick] (120:1.2) -- (120:1.8);
   \draw[black,thick] (120:1.8) -- (120:2.2);
   \draw[black,thick] (120:2.2) to [out = 120, in = 60,looseness=1.5] (60:2.3);
   \draw[black, thick] (60:2.3) -- (60:2.1);
   \draw[black, thick] (60:1.9) -- (60:1.7);
   \draw[black,thick] (60:1.7) -- (60:1.3);
   \draw[black,thick] (60:1.3) -- (60:1.1);
   \draw[black,thick] (60:.9) -- (60:.7);
   \draw[black,thick] (0:.8) -- (0:1.2);
   \draw[black,thick] (60:.7) to [out = 240, in = 180] (0:.8);
   \draw[black,ultra thick] (300:1.1) -- (300:1.9);
   \draw[black,thick] (300:2.1) -- (300:2.3);
   \draw[black,thick] (0:1.8) -- (0:2.2);
   \draw[black,thick] (0:1.2) -- (0:1.8);
   \draw[black,thick] (0:2.2) to [out = 0, in = 300, looseness=1.5] (300:2.3);
   
   \draw (250:1.5) node {1};
   \draw (210:1.2) node {2};
    \draw (210:2.2) node {3};
   \draw (150:1.2) node {4};
     \draw (150:2.2) node {5};
   \draw (90:1.2) node {6};
     \draw (90:2.2) node {7};
   \draw (30:1.2) node {8};
     \draw (30:2.2) node {9};
     \draw (270:.3) node{10};
     \draw (310:1.5) node{11};

\end{tikzpicture}\]

%% file: saddle.tex
\[\begin{tikzpicture}[scale=.8]
\begin{scope}[thick]
\draw [rounded corners = 10mm] (0,0) -- (3,1.5) -- (6,0);
\draw (0,0) -- (0,1);
\draw (6,0) -- (6,1);
\draw [rounded corners = 5mm] (0,1) -- (2.5, 2.25) -- (0.5, 3.25);
\draw [rounded corners = 5mm] (6,1) -- (3.5, 2.25) -- (5.5,3.25);
\draw [rounded corners = 5mm] (0,.5) -- (3,2) -- (6,.5);
\draw [rounded corners = 7mm] (2.23, 2.3) -- (3,1.6) -- (3.77,2.3);
\draw (0.5,3.25) -- (0.5, 2.25);
\draw (5.5,3.25) -- (5.5, 2.25);
\end{scope}

\begin{pgfonlayer}{background2}
\fill [lsugold]  [rounded corners = 10 mm] (0,0) -- (3,1.5) -- (6,0) -- (6,1) -- (3,2) -- (0,1); 
\fill [lsugold] (6,0) -- (6,1) -- (3.9,2.05) -- (4,1);
\fill [lsugold] (0,0) -- (0,1) -- (2.1,2.05) -- (2,1);
\fill [lsugold] (2.23,2.28) --(3.77,2.28) -- (3.77,1.5) -- (2.23,1.5);

\fill [white, rounded corners = 7mm] (2.23,2.3) -- (3,1.6) -- (3.77,2.3);
\fill [lsugold] (2,2) -- (2.3,2.21) -- (2.2, 1.5) -- (2,1.5);
\fill [lsugold] (4,2) -- (3.7, 2.21) -- (3.8,1.5) -- (4,1.5);
\end{pgfonlayer}

\begin{pgfonlayer}{background4}
\fill [lsupurple] (.5,3.25) -- (.5,2.25) -- (3,1.25) -- (2.4,2.2);
\fill [rounded corners = 5mm, lsupurple] (0.5,3.25) -- (2.5,2.25) -- (2,2);
\fill [lsupurple] (5.5,3.25) -- (5.5,2.25) -- (3,1.25) -- (3.6,2.2);
\fill [rounded corners = 5mm, lsupurple] (5.5, 3.25) -- (3.5,2.25) -- (4,2);
\end{pgfonlayer}

\draw [thick] (0.5,2.25) -- (1.6,1.81);
\draw [thick] (5.5,2.25) -- (4.4,1.81);
\draw [thick] (0.5,2.75) -- (2.1,2.08);
\draw [thick] (5.5,2.75) -- (3.9,2.08);

\begin{pgfonlayer}{background}
\draw [black!50!white, rounded corners = 8mm, thick] (0.5, 2.25) -- (3,1.25) -- (5.5,2.25);
\draw [black!50!white, rounded corners = 7mm, thick] (2.13,2.07) -- (3,1.7)  -- (3.87,2.07);
\end{pgfonlayer}
\draw [thick, dashed, rounded corners = 2mm] (3,1.85) -- (2.8,1.6) -- (2.8,1.24);
\draw (0,0.5) node[left]{$D$};
\draw (1.5,3.2) node{$A$};
\draw (4.5,3.2) node{$A$};
\draw (3.8,.8) node{$B$};
\draw (5.3, 1.85) node{$B$};
\end{tikzpicture}\]

%% file: gt2.tex
\[\begin{tikzpicture}[scale = 1.2, rounded corners = 5mm]


\begin{knot}[
 	clip width = 5,
 	ignore endpoint intersections = true,
	end tolerance = 1pt
	]
	\flipcrossings{2,4,6,7,10,12,13,20,21,16,24,17};
	\strand[thick] (0,0) rectangle (1,3);
	\strand[thick] (.5,-.5) rectangle (2.5,.5);
	\strand[thick] (.5,2.5) rectangle (2.5,3.5);
	\strand[thick] (-.5,1) rectangle (1.5,2);
	\strand[thick] (2,0) rectangle (5,1);
	\strand[thick] (2,3) rectangle (5,2);
	\strand[thick, rounded corners = 3mm] (2.5,.5) rectangle (3,2.5);
	\strand[thick, rounded corners = 3mm] (3.25,.5) rectangle (3.75,2.5);
	\strand[thick, rounded corners = 3mm] (4,.5) rectangle (4.5,2.5);

\end{knot}

\end{tikzpicture}\]

%% file: kanenobu.tex
\[\begin{tikzpicture}[scale = .75, thick]

\begin{knot}[	
	consider self intersections,
 	clip width = 3,
 	ignore endpoint intersections = true,
	end tolerance = 2pt
	]
	\flipcrossings{1,4,7,5};
	\strand(0,0) to [out = 90, in = 270] 
	(0,4) to [out = 90, in = 90, looseness=1.5]
	(1,4) to [out = 270, in = 270, looseness=1.5]
	(2,4) to [out = 90, in = 270]
	(2,7.5) to [out = 90, in = 90, looseness=1.5]
	(1,7.5) to [out = 270, in = 90]
	(1,4.5) to [out = 270, in = 270, looseness = 1.5]
	(0,4.5) to [out = 90, in = 270] 
	(0,8.5) to [out = 90, in = 90, looseness=1]
	(5,8.5) to [out = 270, in = 90]
	(5,4.5) to [out = 270, in = 270, looseness=1.5]
	(4,4.5) to [out = 90, in = 270]
	(4,7.5) to [out = 90, in = 90, looseness=1.5]
	(3,7.5) to [out = 270, in = 90]
	(3,4) to [out = 270, in = 270, looseness=1.5]
	(4,4) to [out = 90, in = 90, looseness=1.5]
	(5,4) to [out = 270, in = 90]
	(5,0) to [out = 270, in = 270 ,looseness=1.25]
	(3,0) to [out = 90, in = 270]
	(3,3) to [out = 45, in = 0, looseness=1.5]
	(2.5,4.5) to [out = 180, in = 135, looseness=1.5]
	(2,3) to [out = 270, in = 90]
	(2,0) to [out = 270, in = 270, looseness=1.25]
	(0,0);
	
\end{knot}

\draw[thick] (2,.5) -- (3,.5);
\draw[thick] (2,2.5) -- (3,2.5);
\draw (2.5,1.75) node{$q$};

\draw[thick] (2,7) -- (3,7);
\draw[thick] (2,5) -- (3,5);
\draw (2.5, 6) node{$p$};

\draw (7,6.5) -- (7,2.5);
\draw (8,6.5) -- (8,2.5);
\draw (7,6) -- (8,6);
\draw (7,3) -- (8,3);
\draw (7.5,4.5) node{$n$};
\draw (8.5,4.5) node{$=$};

\begin{knot}[
 	clip width = 3,
 	ignore endpoint intersections = true,
	end tolerance = 1pt
	]
	\flipcrossings{1};
	\strand (9,6.5) to (9,6) to [out = 270, in = 90] (10,5);
	\strand (10,6.5) to (10,6) to [out = 270, in = 90] (9,5);
	\strand (9,2.5) to (9,3) to [out = 90, in = 270] (10,4);
	\strand (10,2.5) to (10,3) to [out = 90, in = 270] (9,4);
\end{knot}

\draw (11,6.5) -- (11,2.5);
\draw (12,6.5) -- (12,2.5);

\draw( 9.5, 2) node{$n>0$};
\draw (11.5,2) node{$n=0$};
\draw (13.5,2) node{$n<0$};

\begin{scope}[xshift = 4cm]

\begin{knot}[
 	clip width = 3,
 	ignore endpoint intersections = true,
	end tolerance = 1pt
	]
	\flipcrossings{2};
	\strand (9,6.5) to (9,6) to [out = 270, in = 90] (10,5);
	\strand (10,6.5) to (10,6) to [out = 270, in = 90] (9,5);
	\strand (9,2.5) to (9,3) to [out = 90, in = 270] (10,4);
	\strand (10,2.5) to (10,3) to [out = 90, in = 270] (9,4);
\end{knot}

\end{scope}

\end{tikzpicture}\]